\documentclass[11pt, letterpaper]{article}

\usepackage{preamble}
\comfalse
\author{
    Yueqi Sheng\thanks{School of Engineering \& Applied Sciences, Harvard University, Cambridge, Massachusetts, USA.
    \newline Email: \textcolor{red}{ysheng@g.harvard.edu}.}
    \and
    Qiang Wu\thanks{School of Mathematics, University of Minnesota, MN, USA.  \newline Email: \textcolor{red}{wuq@umn.edu}.}
}

\title{Central Limit Theorem of Overlap for the Mean Field Ghatak-Sherrington model}
\date{\today}

\begin{document}
\maketitle

\begin{abstract}
\setlength{\parindent}{0em}
\setlength{\parskip}{1em}
The Ghatak-Sherrington (GS) spin glass model is a random probability measure defined on the configuration space $\{0,\pm1,\pm2,\ldots, \pm \cS \}^N$ with system size $N$ and $\cS\ge1$ finite. This generalizes the classical Sherrington-Kirkpatrick (SK) model on the boolean cube $\{-1,+1\}^N$ to capture more complex behaviors, including the spontaneous inverse freezing phenomenon.
We give a quantitative joint central limit theorem for the overlap and self-overlap array at sufficiently high temperature under arbitrary crystal and external fields. Our proof uses the moment method combined with the cavity approach. Compared to the SK model, the main challenge comes from the non-trivial self-overlap terms that correlate with the standard overlap terms. 
\end{abstract}

\thispagestyle{empty}
\newpage

\setcounter{tocdepth}{3}
{
    \hypersetup{linkcolor=blue} 
    \thispagestyle{empty}
    \tableofcontents
    \addtocontents{toc}{\protect\thispagestyle{empty}}
    \pagenumbering{gobble}
}

\newpage
\clearpage
\pagenumbering{arabic}

\newpage

\section{Introduction}
We consider the Ghatak-Sherrington (GS) model: for each configuration
\begin{align}
\bm{\sigma} = (\sigma_1, \sigma_2, \cdots, \sigma_N) \in \Sigma_{N,\cc{S}}:=\{ 0, \pm 1, \cdots, \pm \cc{S}\}^N,
\end{align}
where $\cc{S}\ge 1$, the Hamiltonian of the GS model is defined as
\begin{align}\label{eq:Hamilton}
H_N(\bm{\sigma})= \frac{\beta}{\sqrt{N}}\sum_{i <j} g_{i, j} \sigma_i \sigma_j + D\sum_{i =1}^N\sigma_i^2 + h \sum_{i=1}^N\sigma_i,
\end{align}
where the interaction parameters are $g_{i, j} \overset{\text{i.i.d}}{\sim}  \mathcal{N}(0, 1)$ for $1 \leq i < j \leq N$, $\beta>0$ is the inverse temperature, and $h\ge0$ and $D\in \mathbb{R}$ represent the \textit{external} and \textit{crystal} fields respectively.


We are interested in the fluctuation of overlap array of $n$ configurations, $\{\inner{\bm{\sigma}^i, \bm{\sigma}^j}: i, j \in [n]\}$, as the number of spins $N \to \infty$ and for high enough temperature. The overlap of two configurations or \emph{replicas} $\bm{\sigma}^1, \bm{\sigma}^2$ from $\Sigma_{N, \cc{S}}$, $R_{1, 2}$, is defined as
\begin{align}
R_{1,2} = \frac{1}{N} \sum_{i = 1}^N \sigma^1_i\sigma^2_i.
\end{align}
If $\sigma^1 = \sigma^2$, the overlap becomes the self-overlap
\[R_{1, 1} = \frac{1}{N} \sum_{i=1}^N\sigma^1_i\sigma^1_i.\]
In the SK model, the "centered" overlaps behave asymptotically, as $N \to \infty$, like a family of correlated Gaussian under the Gibbs measure in the high temperature regime~\cite{Tal11, GT02, CN95}. The goal of this paper is to extend their result to the GS model and show that the overlaps and self-overlaps in the GS model converge to a family of correlated Gaussian when the temperature is high. (see Theorem \ref{thm:main}). \YS{Replica symmetry?}


The key idea of the proof is similar to that of the SK model  ~\cite[Chapter 1.10]{Tal11}, which is to decompose the overlaps into independent components (see ~\eqref{def:expand}) and use the cavity method to 
show that the mixed moments of the independent components are approximately the corresponding moments of a family of Gaussian r.v. (see Lemma \ref{lem:general}). Compared to the SK model, where $R_{1, 1} = 1$, the spin configurations in the GS model are in $\{ 0, \pm 1, \cdots, \pm \cc{S}\}^N$ thus $R_{1, 1}$ becomes a random variable. One can expect the overlap terms to be affected by the distribution of the norm of the configuration, i.e., the self-overlap. The main challenge is to characterize the correlation between overlap and self-overlap, which makes the analysis much more involved than in the SK case. 

The overlap array acts as an order parameter of mean-field spin glass models ~\cite{Par79,Par80,Par83}, which contains crucial information about the system. In the high-temperature regime for the SK model, moment estimates of overlap arrays were important for establishing the limiting law of free energy ~\cite{GT02}, the limiting law of spin covariances (Hanen's theorem ~\cite{Han07}), and a sharp upper bound of operator norm for the spin covariance matrix ~\cite{AG23}. To the best of our knowledge, the number of mathematically rigorous results concerning the GS model is quite limited. In~\cite{Pan05}, Panchenko first proved a variational formula for limiting free energy by generalizing Talagrand's method to the GS model and also later in~\cite{Pan18} via a different approach. Recently, Auffinger and Chen~\cite{AC21} used the cavity method to establish the Thouless-Anderson-Palmer equation for local magnetization. Our result could be used to extend the limiting laws in the SK model to the GS model.  




\subsection{Main result}
Given the Hamiltonian defined in \eqref{eq:Hamilton}, the corresponding GS Gibbs measure is 
\begin{align}
   dG_{\beta,h,D}(\bm{\sigma}) = \frac{\exp(H_N(\bm{\sigma}))}{Z_N(\beta,h,D)} \cdot d\bm{\sigma},
\end{align}
where $d\sigma$ is the uniform reference measure on $\Sigma_{N,\cc{S}}$, and the partition function $Z_N$ is is given by
\[
Z_N(\beta,D,h):= \sum_{\sigma \in \Sigma_{N,S}} \exp(H_N(\sigma)).
\]
\textit{In the following, we will suppress the dependence on $\beta, h, D$ for the above objects unless it causes confusion.}
Let $\bm{\sigma}^1, \bm{\sigma}^2, \cdots, \bm{\sigma}^n \in \Sigma_{N, \cc{S}}^n$ be a set of configurations or \emph{replicas}. For any function $f: \Sigma_{N, \cc{S}}^n \to \mathbf{R}$, denote $\langle f \rangle$ as the expectation of $f$ under the product Gibbs measure, that is,
\[\langle f \rangle = \sum_{\bm{\sigma}^1, \bm{\sigma}^2, \cdots, \bm{\sigma}^n} f(\bm{\sigma}^1, \bm{\sigma}^2, \cdots, \bm{\sigma}^n) \Pi_{i = 1}^n dG(\bm{\sigma}^i)\]
Let $\nu(f) = \E[\langle f \rangle]$ be the expectation of $\langle f \rangle$ under interaction parameters.

At sufficiently high temperature, it is expected that the GS model is replica symmetric in the sense that the overlap and self-overlap concentrate on some fixed points respectively.  The explicit form of the system of equations and the following concentration results were given in~\cite{AC21}.

\begin{proposition} [{\cite[Proposition 2]{AC21}}] \label{prop:ac21}
    There exist a $\tilde{\beta} > 0$ s.t for $\beta \in [0,  \tilde{\beta})$, $h \geq 0$ and $\cc{D} \in \R$, there exists unique $p, q \in \mathbf{R}$ s.t.
    \[\nu((R_{1,2} -q)^2) \le \frac{16\cS^2}{N}, \quad 
    \nu((R_{1,1} -p)^2) \le \frac{16\cS^4}{N}.\]
\end{proposition}
In this note, we will use the following notation for where the overlap and self-overlap between arbitrary pairs of replicas concentrate. 
\begin{definition} \label{def:q}
For $k, l \in [n]$, for the pair of replicas $\sigma^k, \sigma^l$, denote
\[Q_{k, l} := \begin{cases}
    p, & \text{if } k \neq l,\\
    q, & \text{if } k = l.
\end{cases}\] 
\end{definition}

Our main result is a quantitative joint central limit theorem for the overlap and self-overlap array among a set of replicas $[n]$, i.e. $\{R_{k, l}\}_{(k, l) \in \cc{C}_n}$ (see $\cc{C}_n$ in Section~\ref{ssec:notation}). We show that for sufficiently high temperatures, the overlap and self-overlap array behave like a family of correlated Gaussians asymptotically as $N \to \infty$. Specifically, all mixed moments of the "recentered" (self-)overlap $\{R_{k, l} - Q_{k, l}: (k, l) \in \cc{C}_n\}$ converges to the corresponding mixed moments of a family of correlated Gaussian r.v.


\begin{theorem} \label{thm:main}
Consider a set of nonnegative integers $\{m(k, l) \in \mathbb{N}: 1 \le k \le l \le n\}$. Set 
\[m = \sum_{1 \le k \le l \le n} m(k, l),\]
and let $\{\eta_{k, l}: (k, l) \in \cc{C}_n\}$ be a family of centered Gaussian with covariances
 \[\cov(\eta_{k, l}, \eta_{k', l'}) := \begin{cases}
        A_2^2 \delta(|(k, l) \cap (k', l')| = 2) + |(k, l) \cap (k', l')| A_1^2 + A_0^2, & \text{if} \ |(k, l)| = |(k', l')| = 2,\\
        B_1^2 \delta(|(k, l)\cap (k', l')| = 1) + B_0^2, & \text{if}\ |(k, l)| = |(k', l')| = 1,\\
        C_1^2 \delta(|(k, l)\cap (k', l')| = 1) + C_0^2,& \text{if} \ |(k, l)| \neq |(k', l')|.
    \end{cases}\]
where the constants $A_2^2, A_1^2, B_1^2, C_1^2, A_0^2, B_0^2, C_0^2$ are given in Lemma~\ref{lem: general formal}. There exists  $\beta'\in (0,\tilde{\beta}]$ s.t. for $\beta<\beta'$, we have 
\[N^{\frac{m}{2}} \cdot \nu \left(\Pi_{k\le l}(R_{k, l} - Q_{k, l})^{m(k, l)}\right)  = \E\left[\Pi_{(k, l) \in \cc{C}_n}\eta_{k, l}^{m(k, l)}\right] + O(N^{-1/2}).\]
\end{theorem}
Theorem \ref{thm:main} states that moments of (self-)overlap array asymptotically equals to the corresponding moment of a correlated Gaussian. The structure of the covariance matrix is inherently given by decomposition of (self-)overlaps using "basis" random variables $\{T_{k, l}, T_k, S_k, T, S\}$ as shown in \eqref{def:expand}. We will show that the family of basis is asymptotically Gaussian (Lemma \ref{lem: general formal}), the theorem then follows by expanding the product of mixed moments of (self-)overlap using "basis" random variables. Note that when $\abs{(k,l)}= \abs{(k',l')}=2$, ~\ie $k\neq l, k' \ne l'$, $\cov(\eta_{k, l}, \eta_{k', l'})$ corresponds to the variance/covariances of the overlaps.

\subsection{Relation to prior works}
In this section, we give some background on the GS model and review some existing fluctuation results on the overlap in the mean field spin glass theory.
\subsubsection{Mean field spin glass models}
Mean field spin glass theory has undergone a flourishing development in the last 20 years, a key breakthrough was the proof for the celebrated Parisi's formula by Talagrand~\cite{Tal06} and Panchenko~\cite{Pan14}. After that, many rigorous results for the mean field spin glass system have been established ~\cite{RSB40, Tal11}. The most notable models are the Sherrington-Kirkpatrick model and its $p$-spin variants, in which the spin configuration space is the hypercube $\{-1,+1\}^N$. There are more realistic but complicated models whose spin could take values from a larger finite set or general vectors in Euclidean space. Some examples include the Ghatak-Sherrington model~\cite{GS77}, Potts spin glass ~\cite{Pan16}, XY-spin glass, etc.

In this work, we consider the Ghatak-Sherrington model, where configuration space is the general hypercube, was first introduced in~\cite{GS77} to study the so-called \emph{inverse freezing phenomenon}. The inverse freezing phenomenon predicts that at low enough temperature there is another replica symmetric regime~\cite{DD00,DYS94,KH99,LD82,MS85,Leu07}. This is in sharp contrast to the binary spin-valued models, such as SK and its $p$-spin variants, where the model in the low-temperature regime is widely believed to exhibit replica symmetry breaking only. 

\subsubsection{Existing fluctuation results}
For the classical SK model, a central limit theorem of overlap in the zero external fields was first proved in~\cite{CN95} via a stochastic calculus approach. In the presence of a positive external field, the central limit theorem for overlap for the array of overlaps was proved in~\cite{Tal11,GT02} using the moment method combined with cavity method computations. 

Establishing a CLT for overlap in the high-temperature regime has many implications: Hanen's theorem~\cite{Han07} uses the moment estimates for the overlap arrays to establish the limiting law of spin covariances; the CLT of overlap for the SK model was crucially used while deriving a sharp upper bound for the operator norm of the spin covariance matrix ~\cite{AG23}. Investigating overlap in the low-temperature regime is a highly challenging open problem for Ising spin glass models. In the spherical SK model, due to a nice contour integral representation of the partition function, the fluctuation results for the overlap have been well understood in the near-critical temperature~\cite{NS19} and low-temperature regime~\cite{LS22}. Moreover, a recent result by \cite{cam23} proved a central limit theorem for the overlap in the Ising SK model on the so-called Nishimori line. On the other hand, oftentimes establishing fluctuation results for the overlap to other generalized spin glass models can be a quite challenging task, even at the high-temperature regime. In~\cite{DW21}, for the multi-species SK model, some second-moment computation was done to compute the variance-covariance matrix for the overlap array. However, the general moments' computation involves many matrix operations and can be highly technical if not impossible. 

Besides the classical SK type model, the central limit theorems of overlap in various regimes for the Hopfield model~\cite{Hop82,Tal98} were also established  in~\cite{GL99,Gen96a,Gen96b} by Gentz et.al. In both Hopefield and SK models, the spin values are restricted as binary. The goal of this work is to extend the fluctuation results to the non-binary spin settings. 
\subsection{Acknowledgement}
We are grateful for the feedback of 
 Boaz Barak and Partha S. Dey.
We also thank Juspreet Singh Sandhu for the discussions in the initial stage.
\section*{Funding}
Y.S. acknowledges support from Simons Investigator Fellowship, NSF grant DMS-2134157, DARPA grant W911NF2010021, and DOE grant DE-SC0022199.

\section{Proof outline} \label{sec: proof outline}
We sketch the outline of the proof in this section. The first step is to decompose the (self-)overlaps as the sum of some "basis" that are mostly pairwise independent. This allows us to rewrite the moments of (self-)overlaps as a homogeneous polynomial over the moment of "basis". Our main technical Lemma (Lemma~\ref{lem: general formal}) says the moments of the basis behave like moments of Gaussian asymptotically. 

The "\emph{basis}" we use to decompose (self-) overlap is a generalization of those of the SK model (see ~\cite[Chapter 1.8]{Tal11}).
\begin{definition} \label{def:basis}
For overlap, let $b := \inner{\sigma_1}$, we define the following basis components,   
\[T_{k, l} := \frac{1}{N}\sum_{i=1}^N(\sigma^k_i - b)(\sigma^l_i - b), \quad T_k := \frac{1}{N} \sum_{i = 1}^N (\sigma^k_i - b)b, \quad T := b^2 - q.\]
For self-overlap, similarly denote $\tilde{b} := \inner{\sigma_1 \sigma_1}$, and the corresponding basis components are
\[
S_{l} = \frac{1}{N}\sum_{i=1}\sigma^{l}_i\sigma^{l}_i - \tilde{b},  \ \text{and} \  S = \tilde{b} - p.
\] 
\end{definition}
By definition, we have the following decomposition of the (self-)overlaps:
\begin{align} \label{def:expand}
    R_{k, l} - q = T_{k, l} + T_k + T_l + T, \quad 
    \text{and} \quad R_{l, l} - p = S_l + S.
\end{align}
The following lemma states that the terms in the above decomposition are mostly pair-wise independent of each other under $\nu(\cdot)$. We defer the proof to section \ref{sec: setup}.

\YS{comments: check the new statement? list pairs that are independent, state where it will be prove}
\begin{restatable}{lemma}{independence}\label{lem:independence}
Let $(X, Y)$ be a pair of random variable from $\{\{T_{k,l}\}_{1\le k< l \le n}, \{T_k\}_{k\le n}, T, \{S_k\}_{k\le n}, S\}$ as defined above, $\cov(XY) \neq 0$ iff $(X, Y)$ is of the form $\{S_k,T_k\}$ for $k\le n$ or $\{S,T\}$.

\end{restatable}


Now to show Theorem \ref{thm:main}, it suffices to show that the set of basis $\{T_{k, l}, T_k, S_k, T, S: k, l \in [n]\}$ are asymptotically Gaussian. This is the statement of our main technical lemma below.

\YS{list of comments: this informal version is too cumbersome to read. define "basis", formal version hasn't show up, need to state that first}
\begin{lemma} [Informal version of Lemma \ref{lem: general formal}] \label{lem:general}
Consider the family of all possible "basis" given in Definition \ref{def:basis}, i.e.  $\{\{T_{k,l}\}_{1\le k< l \le n}, \{T_k\}_{k\le n}, T, \{S_k\}_{k\le n}, S\}$. 

There exist $\beta' \in (0, \beta]$ and a family of centered Gaussians indexed by all possible "basis", i.e. $\{\{g_{T_{k,l}}\}_{1\le k< l \le n}, \{g_{T_k}\}_{k\le n}, g_T, \{g_{S_k}\}_{k\le n}, g_S\}$, s.t. for $0 \leq \beta' \leq \beta$, the family of basis converges in distribution to the family of Gaussians as $N \to \infty$.
\end{lemma}
The explicit variance-covariance structure of the family of Gaussians is given in Lemma \ref{lem: general formal}. Note that the family of Gaussains in Lemma \ref{lem:general} are independent except the cases  $\{S_k, T_k\}$ and $\{S, T\}$. It's easy to check that Theorem \ref{thm:main} follows from Lemma \ref{lem:general} by setting  
\[\eta_{k, l} := \begin{cases}
    g_{T_{k, l}} + g_{T_k} + g_{T_l} + g_T, & \text{if } k \neq l,\\
    g_{S_k} + g_S, & \text{if } k = l.
\end{cases}\] 
In the rest of this paper, we will focus on the proof of Lemma~\ref{lem:general}.
\subsection{Organization of the paper}
The paper is structured as follows. In Section~\ref{sec: preli}, we introduce the setup for the cavity method and give some technical preliminaries. The second-moment computations for the variance-covariance estimation are carried out in Section~\ref{sec:var}. In Section~\ref{sec: general}, we generalize the second moment computation in Section~\ref{sec:var} to general moments of the "basis" $T_{k, l}, T_k, S_k, T, S$. The results are formally stated in Lemma~\ref{lem: general formal}. More specifically, the inductive relations on different "basis" are given in Section~\ref{ssec:induct-Tkl} and~\ref{ssec:induct-TS}. Some lemmas involving technical but repetitive computations are deferred to the Appendix Section~\ref{sec:app}. Finally, as we pointed out in Section~\ref{sec: proof outline}, to prove Theorem~\ref{thm:main}, it suffices to prove Lemma~\ref{lem: general formal}, whose proof is included in Section~\ref{ssec:pf-main-lem}.

\paragraph{Notations}\label{ssec:notation} 

\begin{itemize}
\item We denote $\langle \cdot \rangle$ as the Gibbs average and $\nu(\cdot) := \E[\inner{\cdot}]$, where $\E$ denotes the average w.r.t the disorder $g_{i, j}$.
\item Let $n$ be the number of replicas, $N$ be the number of spins (or the system size) and $\cc{S}:= \|\bm{\sigma}\|_{\infty}$ be the largest spin value. 
\item For $k, l \in [n]$, denote $R_{k, l}$ as the overlap for the configuration $\bm{\sigma}^k, \bm{\sigma}^l \in \Sigma_{N,\cc{S}}$. (Setting $k = l$ gives the self-overlaps, $R_{k, k}$). We use $Q_{k, l}$ to denote where the overlap/self-overlap concentrates. 
\item We use $b := \inner{\sigma_1}$ and $\tilde{b} := \inner{\sigma_1^2}$ to denote the first and second moment for a single spin under quenched Gibbs measure, \ie for the fixed disorder.
\item We use $\epsilon_l$ to denote the last spin of $\bm{\sigma}^l$ and $\epsilon_{k, l} := \epsilon_k \epsilon_l$. Moreover, $R^-_{k, l} = R_{k, l} - \frac{1}{N}\epsilon_{k, l}$ is the overlap without counting contribution from the last spin.

\item For a positive integer $n$, denote $[n] = \{1, 2, \cdots, n\}$ as the set of all positive integer up to $n$. Let $\cc{C}_n := \{(k, l): k, l \in [n], k \le l\}$ be the set of all replica pairs contained in $[n]$.

\item Finally, denote $O_N(H)$ as $O(N^{-H/2})$

\end{itemize}
\section{Cavity method and second moment estimates} \label{sec: preli}
We begin with the idea of the cavity method and show how one can use it to obtain the second-moment estimation of the "basis". The idea of the cavity method is based on studying the effect of isolating the $N$th spin from the rest of the system, 
which is formally formulated into the following interpolation scheme. 

For $t\in [0,1]$, the interpolated Hamiltonian at time $t$ is given by
\begin{align}
    H_N^t(\bm{\sigma}) = H_{N-1}(\bm{\rho}) & + \sigma_N\left(\sqrt{t} \cdot \frac{\beta}{\sqrt{N}}\sum_{i=1}^{N-1} g_{i,N}\sigma_i + \sqrt{1-t} \cdot \beta \eta \sqrt{q} \right) \\
     & \qquad \qquad + (1-t) \cdot \frac{\beta^2}{2}(p-q) \sigma_N^2 + D \sigma_N^2 + h\sigma_N, \label{eq:inter}
\end{align}
where $\bm{\rho}:= (\sigma_1,\ldots, \sigma_{N-1})$, $\eta\sim N(0,1)$ independent of $g_{i,j}$ and 
\[
H_{N-1}(\bm{\rho}):= \frac{\beta}{\sqrt{N}}\sum_{1\le i<j \le N-1} g_{i,j} \sigma_i \sigma_j + D \sum_{i=1}^{N-1}\sigma_i^2 + h \sum_{i=1}^{N-1} \sigma_i.
\]
At $t=0$, the last spin is decoupled from the original system, which brings out a small change, heuristically known as "cavity"; at $t=1$, $H_N^1(\bm{\sigma})$ is just the original GS Hamiltonian. 

In the following, we use $\epsilon_l$ to denote the last spin of $l$-th replica, that is, $\epsilon_l := \sigma_N^l$. For a pair of replica $k, l \in [n]$, we denote the (self-)overlap without the last spin as
\[R^-_{k, l}:= R_{k, l} - \frac{1}{N}\epsilon_{k, l}.\]
In this paper, we use $\inner{\cdot}_t$ as the corresponding Gibbs average at time $t$ and $\nu_t(\cdot) := \E[\inner{\cdot}_t]$. In particular, at $t=1$, $\nu_1(\cdot) = \nu(\cdot)  = \E [\langle \cdot \rangle]$. By ~\cite[Lemma 1]{AC21}, for any pair of replicas $k, l \in [n]$, 
\[Q_{k, l} = \nu_0(\epsilon_{k, l})\]
\subsection{Set-ups and Preliminaries 
}\label{sec: setup}
Recall that the goal is to compute the joint moments of (self-)overlaps, the first step is the decomposition to "basis" terms given in \eqref{def:expand}. We begin by proving some basic properties of the "basis".

\paragraph{Properties of "basis"}
First, we show that the set of random variables $\{T_{k, l}, T_k, S_k, T, S\}$ are mostly pari-wise independent as stated in Lemma~\ref{lem:independence}.

\independence*

\begin{proof}[Proof of Lemma~\ref{lem:independence}]
For pairs of random variable that doesn't involve $S_k$, the proof is the same as in SK mode (see e.g. ~\cite[Proposition 1.8.8]{Tal11}). We present the proof for the pairwise independence of $S_l$ and $\{T_{k, k'}, T_k, S_k: k \neq l\}$. For  $X, Y \in \{T_{k, l}, T_k\}$, $\nu(XY) = 0$ follows directly from symmetry of types of (self-)overlaps.

For pairs of term involving $T_{k, l}, S_h$: Consider a set of constants $\{k, l, h\}$ s.t. $k \neq l$ and some constant $h' \notin \{k, l, h\}$. 
\begin{align*}
    \nu(T_{k, l}S_h) = \nu((R_{h,h} - R_{h', h'}) T_{k, l})
\end{align*}
Note that there exists a replica in $\{k, l\}$ that does not appear in $S_h$. WLOG, assume $h \neq l$, the integrate w.r.t. $\sigma^l$ gives 
\[\nu(T_{k, l}S_h) = 0\]
For pair of term involving $T_k, S_h$: if $k \neq h$, then by symmetry
\[\nu(T_k S_h) = \nu((R_{h,h} - R_{h', h'})T_k) = 0\]
\end{proof}
To continue, we introduce another trick to express the "basis" random variables with (self)-overlaps by introducing a new replica for each occurrence of $b, \tilde{b}$. This trick has been used many times in~\cite[Chapter 1.8]{Tal11}, and we record it here for completeness.
\begin{claim} \label{claim: basis to overlap}
Fix some integer $n$. For $k, l \notin [n]$ s.t. $k \neq l$, 
\[\nu(T_{1, 2}) = \nu(R_{1, 2} - R_{1, l} - R_{k, 2} + R_{k, l}),\]
\[\nu(T_1) = \nu(R_{1, l} - R_{k, l}), \quad \nu(S_1) = \nu(R_{1, 1} - R_{k, k}),\]
\[\nu(T) = \nu(R_{k, l} - q), \quad \nu(S) = \nu(R_{k, k} - p).\]
\end{claim}
\begin{proof}
The proof follows from the linearity of expectation. We will show a proof for $S_{1}$, the other terms can be proved using the same technique. 
\begin{align*}
    \nu(S_1) &= \E[\inner{\frac{1}{N}\sum_i(\sigma_1^i)^2 - \tilde{b}}]
    = \E[\inner{\frac{1}{N}\sum_i(\sigma_1^i)^2 - \inner{\epsilon_k^2}}]\\
    & = \E[\inner{\frac{1}{N}\sum_i\left((\sigma_1^i)^2 - (\sigma_i^k)^2\right)}]\\
    &= \nu(R_{1, 1} - R_{k, k}).
\end{align*}
where the second equality is the definition of $\tilde{b}$ and the third equality uses symmetry between sites.
\end{proof}
This implies that we can expand moments of basis as a homogeneous polynomial of (self-)overlaps over a set of replicas. 

\paragraph{Approximation of moments} We use the following definition to capture the degree of a term.
\begin{definition} \label{def: order of a function}
    For $f: \Sigma^{\otimes n}_{N, S} \to \R$, we say $f$ is of order $H$ if $f$ is a product of $H$ centered overlaps or self-overlaps, $R_{k, l} - Q_{k, l}$ for $k, l \in [n]$. 
\end{definition}
Estimating the magnitude of order $H$ functions follows a standard application of concentration of overlaps and H\"older's inequality. The following Lemma generalizes the second-moment estimates of centered (self-)overlaps in Proposition~\ref{prop:ac21}.
\begin{lemma}[{\cite[Proposition 5]{chen22}}] \label{lem:moment bound}
For $\beta < \tilde{\beta}$, there exist some constant $C>0$ such that for any $k\ge 1$ and $l, l' \in [n]$, we have 
\[
\nu\left((R_{l,l'}-Q_{l,l'})^{2k}\right) \le \left(\frac{Ck}{N}\right)^k,
\]
\end{lemma}
This implies that if $f$ is an order $H$ function, then there exists a constant $C$ that doesn't depend on $N$ s.t. 
\[\nu(f) \leq C \cdot N^{-\frac{H}{2}}.\]
To lighten the notation, we overwrite the big $O$ notation and say a quantity $A = O_N(H)$ if 
\[|A| \leq K \cdot N^{-\frac{H}{2}},\]
for some constant $K$ that does not depend on $N$. Note that the constant $K$ can depend on other parameters such as $\beta, n, \cc{S}$.


One of the main tools we use in the cavity method is $\nu_1(f) \approx v_0(f) + v'_0(f)$. Let's first recall the structure of $\nu'_t(f)$.
\begin{lemma}[{\cite[Lemma 3]{AC21}}]  \label{lem:derivative}
Let $f: \Sigma_{N,\cS}^{\otimes n} \to \vvR$ be any function of $n$ replicas, for $t\in(0,1)$, we have
\begin{align*}
    2\frac{d}{dt} \nu_t(f) =& \beta^2\sum_{1 \leq k, l \leq n} \nu_t (\epsilon_k\epsilon_l (R^-_{k, l} - Q_{k, l})) f)\\
    &-  2 \beta^2\sum_{1 \leq k \leq n; n + 1 \leq l \leq 2n} \nu_t(\epsilon_k\epsilon_l (R^-_{k, l} - q)f)\\
    &+ \beta^2 n(n + 1) \nu_t (\epsilon_{n+1, n+2}(R^-_{n + 1, n +2} - q)f) - \beta^2 n \nu_t (\epsilon_{n + 1}^2 (R^-_{n +1, n +1} - p)f).
\end{align*}
\end{lemma}

\begin{remark}

We present a convenient way of rewriting the above lemma. For $a,b \in [2n]$, let 
\begin{align}
 \text{sgn}(a, b) := (-1)^{|\{a,b\} \cap [n]|},  
\end{align}
then we have
\begin{align} \label{derivative-2}
    \frac{d}{dt} \nu_t(f) =& \cc{R}_{n, f} + \frac{\beta^2}{2}\sum_{a, b \in [2n]} \text{sgn}(a, b)\nu_t(\epsilon_{a, b} (R^-_{a, b} - Q_{a, b})f),
\end{align}
where $\cc{R}_{n, f}$ corresponds to additional terms from replicas independent from $f$. For $a \in [2n]$, denote $a' = 2n + a$
\begin{align}\label{eq:Rnf}
    \cc{R}_{n, f} := \frac{\beta^2}{2}\sum_{a, b \in [2n]} \text{sgn}(a, b)\nu_t(\epsilon_{a', b'} (R^-_{a', b'} - Q_{a', b'})f).
\end{align}  
\end{remark}
To quantify the difference between $\nu_1(f)$ and $\nu_0(f)$ by the "degree" of $f$, we have
\begin{proposition} \label{prop: approx}
For $f: \Sigma^{\otimes n}_{N, S} \to \R$ s.t. $f$ is a product of $H$ centered overlaps or self-overlaps, $R_{a,b} - Q_{a, b}$ for $a, b \in [n]$,
    \begin{align}
    |\nu_0(f) - \nu(f)| = O_N(H + 1), \label{eq:1st approx}\\
    |\nu_0(f) + \nu_0'(f) - \nu(f)| = O_N(H + 2). \label{eq:2nd approx}
\end{align}
\end{proposition}
The proof of Proposition \ref{prop: approx} is based on the concentration of overlaps and H\"older's inequality. For the mean field GS spin glass model, those types of results were already established in~\cite{AC21,chen22}. First, we have an upper bound for $\nu_t(f)$.
\begin{lemma}[{\cite[Lemma 4]{AC21}}] \label{lem:exp decay}
For $f: \Sigma_{N, S}^{\otimes n} \to [0,\infty)$, we have 
    \[\nu_t(f) \leq \exp(6n^2 \beta^2 \cc{S}^4) \nu(f).\]
\end{lemma}
The overlap and self-overlap concentration results already stated in Proposition~\ref{prop:ac21}, the following presents higher order moments estimate.
Consequently, we get similar results for $R^-_{l, l'}$.
\begin{proof}[Proof of Proposition \ref{prop: approx}]
To prove \eqref{eq:1st approx}, note that 
\[|\epsilon_{k, l}(R^-_{k, l} - Q_{k, l})| \leq |\epsilon_{k, l}(R_{k, l} - Q_{k, l})| + \frac{\cc{S}^4}{N},\]
we will bound $\nu'_t(f)$ using Lemma \ref{lem:derivative}, \ref{lem:exp decay} and H\"older's inequality.
\begin{align*} 
    &|\nu_1(f) - \nu_0(f)|\\
    &\leq \sup_t \Bigg\{3n^2 \cc{S}^2 \beta^2 \left(\nu_t(|f|^p)^{\frac{1}{p}}\nu_t(|R_{1, 2} - Q_{1, 2}|^q)^{\frac{1}{q}} + \nu_t(|f|^p)^{\frac{1}{p}}\nu_t(|R_{1, 1} - Q_{1, 1}|^q)^{\frac{1}{q}} + \nu_t(|f|)\frac{\cc{S}^2}{N}\right) \Bigg\}\\
    &\leq \exp(6n^2 \beta^2 S^4) 3n^2 \cc{S}^2 \beta^2 \left(\nu_1(|f|^p)^{\frac{1}{p}} \nu_1(|R^-_{1, 2} - Q_{1, 2}|^q)^{\frac{1}{q}} + \nu_t(|f|^p)^{\frac{1}{p}}\nu_t(|R_{1, 1} - Q_{1, 1}|^q)^{\frac{1}{q}} + \nu_1(|f|)\frac{\cc{S}^2}{N}\right), \label{holder}
\end{align*}
Since $f$ is of order $H$, by Lemma \ref{lem:moment bound}, apply H\"older's  inequality with $p = q = 2$ gives the desired result.

For \eqref{eq:2nd approx}, observe that 
\[|\nu_1(f) - \nu_0(f) - \nu'_0(f)| \leq \sup\limits_{\substack{0 \leq t \leq 1}} |\nu_t''(f)|.\]
By Lemma \ref{lem:derivative}, $\nu''(f)$ brings an additional factor of $R^-_{i, j} - Q_{i, j}$. Apply the above proof on $f(R_{1, 2} - p)$ and $f(R_{1, 1} - q)$ gives the desired result.
\end{proof}

Later in the proof, we will need to study terms involving (self-)overlaps without the last spin, i.e. $R^-_{k, l} - Q_{k, l}$. Here we establish the analog of the above results on $R^-_{k, l} - Q_{k, l}$.

\begin{corollary}[of Lemma \ref{lem:moment bound}] \label{cor: moment bound}
For $\beta < \tilde{\beta}$, there exist some constant $C'>0$ such that for any $k\ge 1$, we have 
    \[
\nu\left((R^-_{l,l'}-Q_{l,l'})^{2k}\right) \le \left(\frac{C'}{N}\right)^k,
\]
\end{corollary}
\begin{proof}
By Minkowski’s inequality, 
\[\nu\left((R^-_{l,l'}-Q_{l,l'})^{2k}\right)^{\frac{1}{2k}} \leq \nu\left((R_{l,l'}-Q_{l,l'})^{2k}\right)^{\frac{1}{2k}} + \frac{\cc{S}^2}{N} \leq \left(\frac{C'}{N}\right)^\frac{1}{2},\]
where the last inequality follows from Lemma \ref{lem:moment bound}. Raise both sides to $2k$-th power gives the desired result.
\end{proof}
\begin{lemma} \label{lemma:last spin}
Fix an integer $n$ and $H$, for each $1 \leq v \leq H$, consider $v_1, v_2 \in [n]$ and $U_v \in \{T_{v_1, v_2}, T_{v_1}, T, S_{v_1}, S\}$. Let $f: \Sigma_{N, \cc{S}}^{\otimes n} \to \R$ s.t. $f = \Pi_{1 \leq v \leq H} U_{v}$. Denote $f^- := \Pi_v U^-_{v}$. We have
\[|\nu(f) - \nu(f^-)| = O_N(H + 1). \]
\end{lemma}
\begin{proof}
Observe that each term of $f$ is of the form $\nu(T_{1, 2}), \nu(T_1), \nu(S_1), \nu(T), \nu(S)$ and can be written as linear combination of (self-)overlaps where each occurance of $b := \inner{\sigma_1}$ and $\tilde{b} := \inner{\sigma_1^2}$ corresponds to a new replica. For example, it's easy to check that 
\[\nu(T_{1, 2}) = \nu(\inner{\sigma_1 - \sigma_3, \sigma_2 - \sigma_4}) = \nu((R_{1, 2} - q) - (R_{1, 4} - q) - (R_{2, 3} - q) + (R_{1, 4} - q)).\]
The rest of the terms can be rewritten in a similar way.

Since $f$ is the product of $H$ such terms, expanding the product shows that it's a sum of functions of order $H$. By Lemma~\ref{lem:moment bound} and Holder's inequality, $f = O_N(H)$. Moreover, for $I \subset [H]$, $\nu(\Pi_{u \notin I} U_{p(u), q(u)}) = O(|I|)$. Again by Lemma \ref{lem:moment bound}
 \begin{align*}
    |\nu(f) - \nu(f^-)|&\leq \Bigg|\sum_{I \subset [H]} \nu\left(\Pi_{v \in I}\frac{\epsilon(v)}{N} \Pi_{u \notin I} U_{p(u), q(u)}\right)\Bigg|= O_N(H + 1).
\end{align*}
\end{proof}
The following Corollary tells us that the error to approximate $\nu(f)$ by $\nu_0(f)$ is small. 
\begin{corollary} \label{cor: last spin}
Let $f: \Sigma^n_{N, \cc{S}} \to \R$ be an order $H$ function, we have
\[\nu_0(f^-) = \nu(f) + O_N(H + 1).\]
\end{corollary}
\begin{proof}
Applying Corollary~\ref{cor: moment bound} and equation~\eqref{eq:1st approx} for $f^-$, combining with Lemma \ref{lemma:last spin} gives
\[\nu_0(f^-) = \nu(f^-) + O_N(H + 1) = \nu(f) + O_N(H + 1).\]
\end{proof}

\subsection{Variance of overlaps and self-overlaps} \label{sec:var}
In this section, we compute the variance-covariance structure of a subset of the "basis": $T_{1, 2}, T_1, S_1$. The variance-covariance computation of $S, T$ follows the same idea as $T_1, S_1$ and we will show it as a special case of general moments in Theorem \ref{thm: var ST}. The main goal is to get a sense of how to handle the additional self-overlap terms.
We further note that the following variance results hold at sufficiently high temperature, that is, $\beta< \beta'$ from some $\beta'$ as in the Theorem~\ref{thm:main}. While stating the results in the following context, we might not repeatedly specify the high temperature condition ($\beta<\beta'$).

We begin by demonstrating how the cavity method is used to compute the second moment of the basis random variables. With some abuse of notation, let $X$ be the expansion given in Claim \ref{claim: basis to overlap} using (self-)overlaps for $\{T_{1, 2}, T_1, S_1\}$ and $\epsilon_X$ be the expression by replacing each overlap in $R_{k, l} = \frac{1}{N}\sum_i \sigma^k_i\sigma^l_i$ $X$ by the last spin $\epsilon_k\epsilon_l$. Let $X^- := X - \frac{1}{N}\epsilon_X$ be the part of the basis that depends only on the first $N - 1$ spins.

Note that for $X \in T_{1, 2}, T_1, S_1$, by symmetry of spins
\[\nu(X^2) = \nu(\epsilon_X X).\]
We can further decouple the last spin from the expression to get
\begin{align} 
    \nu(X^2) &= \frac{1}{N}\nu(\epsilon_X^2) + \nu(\epsilon_X X^-)\\
    &= \frac{1}{N}\nu(\epsilon_X^2) + \nu'_0(\epsilon_X X^-) + O_N(3),\label{eq: var init}
\end{align}
where the last equality follows from \eqref{eq:2nd approx} and $\nu_0(\epsilon_X X^-) = \nu_0(\epsilon_X)\nu_0(X^-) = 0$.
\textit{(Note that in the above expression, each copy of $X$ will introduce at least one new replica.)}

This is the starting point of the variance-covariance calculations. To simplify notations, we record some constants corresponding to the expectation of the last spins.

\YS{use a table or reduce the number of constants. Added table}
\begin{definition}
We define the following constants corresponding to terms from $\nu_0(\epsilon_X^2)$ for $X \in \{T_{1, 2}, T_1, S_1\}$.
\begin{center}\label{constants}
\begin{tabular}{ |p{5cm}|p{5cm}|  }
 \hline
 \multicolumn{2}{|c|}{List of constants} \\
 \hline
  $A:= \nu_0\left((\epsilon_1 - \epsilon_3)(\epsilon_2 - \epsilon_4) \epsilon_{1, 2}\right)$ & \\
 \hline
  $D := \nu_0((\epsilon_1^2 - \epsilon_{2}^2) \epsilon_1^2)$  &    \\
  \hline
 $F := \nu_0((\epsilon_{1, 3} - \epsilon_{2, 3}) \epsilon_{1, 3})$ & $G := \nu_0((\epsilon_{1, 3} - \epsilon_{2, 3}) \epsilon_{1, 4})$\\
 \hline
 $E := \nu_0((\epsilon_1^2 - \epsilon_2^2)\epsilon_{1, 3})$ & $H := \nu_0((\epsilon_{1, 3} - \epsilon_{2, 3}) \epsilon^2_{1}))$\\
 \hline
\end{tabular}
\end{center}

The list of constants below will occur many times in computation involving $S_1, T_1, S, T$ as normalizing constants and we record them here.
\[M_1 := 1 - \beta^2(F - 3G); \quad M_2 := 1 - \frac{\beta^2}{2}D; \quad M_3 := (1 - \beta^2(F-G)).\]
\[M = M_1M_2 + \beta^4E^2.\]
Note that by the definitions above, $E = H$, $A = F - G$, and that $M_1, M_2, M_3, M$ are independent from $N$.
\end{definition}
\begin{remark}
Note the constants defined above are close to covariances of the last spins: $F - G = \E(\inner{\epsilon_{1, 1} - \epsilon_{1, 2}}_0^2)$,  $G = \E[b^2\left(\inner{\epsilon_{1, 1}}_0 - \inner{\epsilon_{1}}_0^2\right)]$ and $D = \E[\inner{\epsilon_1^4}_0 - \inner{\epsilon_1^2}_0^2]$. For $0 \leq \beta \leq \beta' < \frac{1}{2\cc{S}^2}$, we have $F - G, F - 3G, D \in (0, 4\cc{S}^4]$ and $M_1, M_2, M_3 > 0$.
\end{remark}
\subsubsection{Variance of $T_{1, 2}$}
We begin by checking $\nu(T_{12}^2)$. By Lemma \ref{lem:independence}, we should expect this term to behave the same as in the SK model (\cite[Proposition 1.8.7.]{Tal11}). This is indeed the case as we will show below. 
\begin{lemma}\label{lem:var T12}
For $\beta<\beta'$, we have
\[\nu(T_{12}^2) = A_2^2 + O_N(3),\]
where
\[A_2^2 := \frac{A}{N(1 - \beta^2 A)}.\]
\end{lemma}
\begin{proof} 
Using \eqref{eq: var init} with $X = T_{1, 2}$, we have
\begin{align*}
\nu (T_{1,2}^2) & = \nu\left( \frac{(\bm{\sigma}^1 - \bm{b})\cdot (\bm{\sigma}^2 -\bm{b})}{N} \frac{(\bm{\sigma}^1 - \bm{b})\cdot (\bm{\sigma}^2 -\bm{b})}{N} \right) \\
 & = \nu\left( \frac{(\bm{\sigma}^1 - \bm{\sigma}^3)\cdot (\bm{\sigma}^2 -\bm{\sigma}^4)}{N} \frac{(\bm{\sigma}^1 - \bm{\sigma}^5)\cdot (\bm{\sigma}^2 -\bm{\sigma}^6)}{N} \right),
\end{align*}
where the last equality follows from replacing each occurrence of $\bm{b}=(\langle \sigma_1 \rangle, \cdots, \langle \sigma_N\rangle)$ by a new replica. 
Rewrite the above formula by expanding the inner products and replacing each term with appropriate overlaps, we get
\begin{align*}
    \nu(T_{1,2}^2) &= \nu((R_{1,2} - R_{1,4} -R_{2,3} + R_{3,4})(R_{1,2} - R_{1,6} -R_{2,5} + R_{5,6}))\\
    &= \frac{1}{N}\nu((\epsilon_1 - \epsilon_3)(\epsilon_2 - \epsilon_4)(\epsilon_1 - \epsilon_5)(\epsilon_2 - \epsilon_6)) + \nu((\epsilon_1 - \epsilon_3)(\epsilon_2 - \epsilon_4)(R^-_{1,2} - R^-_{1,6} -R^-_{2,5} + R^-_{5,6}))\\
    &= \frac{1}{N} \nu_0((\epsilon_1 - \epsilon_3)(\epsilon_2 - \epsilon_4)(\epsilon_1 - \epsilon_5)(\epsilon_2 - \epsilon_6)) + \nu_0'((\epsilon_1 - \epsilon_3)(\epsilon_2 - \epsilon_4)(R^-_{1,2} - R^-_{1,6} -R^-_{2,5} + R^-_{5,6})) + O_N(3),
\end{align*}
where the first equality follows from the symmetry of spins and isolates the last spin from the overlaps and the second step is due to Proposition~\ref{prop: approx}. For the first term
\begin{align*}
    &\nu_0((\epsilon_1 - \epsilon_3)(\epsilon_2 - \epsilon_4)(\epsilon_1 - \epsilon_5)(\epsilon_2 - \epsilon_6))\\
    & = \nu_0((\epsilon_1 - \epsilon_3)(\epsilon_2 - \epsilon_4)\epsilon_{1, 2})\\
    &= A.
\end{align*}
For the second term,
\begin{align*}
    &\nu'_0((\epsilon_1 - \epsilon_3)(\epsilon_2 - \epsilon_4)(R^-_{1,2} - R^-_{1,6} -R^-_{2,5} + R^-_{5,6})) \\
    =& \frac{\beta^2}{2}\sum_{a, b } \text{sgn}(a, b) \nu_0((\epsilon_1 - \epsilon_3)(\epsilon_2 - \epsilon_4)\epsilon_{ab})\cdot \nu_0((R^-_{a, b} - Q_{a, b})(R^-_{1,2} - R^-_{1,6} -R^-_{2,5} + R^-_{5,6})) + \cc{R}_{4, T_{1, 2}^2}.
\end{align*}
Observe that the term involving last spins, $\nu_0((\epsilon_1 - \epsilon_3)(\epsilon_2 - \epsilon_4)\epsilon_{ab})$, is non-zero only when $\{a, b\} \in \{\{1, 2\}, \{1, 4\}, \{2, 3\}, \{3, 4\}\}$. Summing over all such $a, b$, by Corollary \ref{cor: last spin}, we have
\[\nu'_0((\epsilon_1 - \epsilon_3)(\epsilon_2 - \epsilon_4)(R^-_{1,2} - R^-_{1,6} -R^-_{2,5} + R^-_{5,6})) = A\nu(T_{1, 2}^2) + O_N(3).\]
Together, the two terms give
\[\nu(T_{1, 2}^2) = \frac{A}{N} + \beta^2 A\nu(T_{1, 2}^2).\]
\YS{To show $A_2^2 \geq 0$: Since $1 - \beta^2 A = M_3 \geq 0$, we will check $A \geq 0$ \YS{assume $M_1, M_2, M_3 \geq 0$}
\begin{align*}
    A = F - G &= \nu_0(\epsilon_{1, 1}\epsilon_{3, 3} - 2\epsilon_{1, 3}\epsilon_{1, 4} + \epsilon_{1, 2}\epsilon_{3, 4})\\
    &= \nu_0(\inner{\epsilon_{1, 1} - \epsilon_{1, 2}}^2) \geq 0
\end{align*}}
\end{proof}

The following relation involving $\nu(T_{1, 2}^2)$ will be useful later and we record it here for convenience.
\begin{claim} \label{claim:T12}
By definition, $A = F - G$
\begin{align} 
    \beta^2 A_2^2 + \frac{1}{N} &:= \frac{1}{N(1 - \beta^2(F-G))} = \frac{1}{NM_3}
\end{align}
for $A_2^2$ given in Lemma~\ref{lem:var T12}.
\end{claim}
\YS{Intuitively, this term occurs mainly follows from the independence of $T_{1, 2}$ and we sum over the contributions from both the first and second term.}
\subsubsection{Variance of $T_1$ and $S_1$} 
We will now check the variance of $S_1, T_1$. Unlike in the SK model, the basis are not independent of each other anymore. This hints that we should handle $S_1, T_1$ together. 
\begin{theorem} \label{thm: variance}
For $\beta < \tilde{\beta}$, the variance of $T_1, S_1$ are given by
\[\nu(T_1^2) = A_1^2 + O_N(3),\]
where 
\[A_1^2 := \frac{GM_2 + \frac{\beta^2}{2}E^2}{M}\cdot \frac{1}{NM_3} = \frac{1}{2\beta^2N}\left(\frac{1}{M_3} - \frac{M_2}{M}\right),\]
and
 \[
       \nu(S_1^2) = B_1^2 + O_N(3),
\]
    where 
    \[B_1^2 = \frac{DM_1 - 2\beta^2E^2}{NM} = \frac{2}{N\beta^2}(\frac{M_1}{M} - 1),\]
The covariance is
 \[\nu(S_1T_1) = C_1^2 + O_N(3),\]
    where 
    \[C_1^2 := \frac{E}{NM}.\]
\end{theorem}
The above theorem could be viewed as a generalization of showing $\nu(T_1^2)$ in the SK model, with the addition of handling self-overlap terms from $\nu'_0(f)$ in \eqref{eq: var init}. We will compute each part of the theorem in Lemma \ref{lem: T1 var}, Lemma \ref{lem: S1 var}, and Lemma \ref{lem: var s1t1}.
\begin{lemma} \label{lem: T1 var} 
For $\beta \leq \beta'$, we have 
\[\nu(T_1^2) = A_1^2 + O_N(3),\]
where 
\[A_1^2 := \frac{GM_2 + \frac{\beta^2}{2}EH}{M} \cdot \frac{1}{NM_3} = \frac{1}{2\beta^2N}\left(\frac{1}{M_3} - \frac{M_2}{M}\right).\]
\end{lemma}
To prove this, we will use the following lemma to characterize the relation between $\nu(T_1^2)$ and $\nu(T_1S_1)$.
\begin{lemma} \label{lem:relation1}
We have
    \begin{align} \label{eq1}
        (1 - \beta^2 (F - 3G)) \nu(T_1^2)= G \left(\frac{1}{N} + \beta^2 A_2^2\right) + \frac{\beta^2}{2} H \nu(S_1T_1)+ O_N(3),
    \end{align}
and 
    \begin{align}\label{eq2}
        \frac{(1 - \frac{\beta^2}{2} D)}{E}\nu(S_1T_1) = (\frac{1}{N} + \beta^2A_2^2) - 2\beta^2 \nu_0(T_1^2) + O_N(3).
    \end{align}
\end{lemma}
Note that Lemma \ref{lem: T1 var} follows immediately from Lemma \ref{lem:relation1}.
\begin{proof}[Proof of Lemma \ref{lem: T1 var}]
    Plug \eqref{eq2} into \eqref{eq1} and rearrange gives
\begin{align}
    \left(1 - \beta^2 (F - 3G) + \frac{\beta^4EH}{1 - \frac{\beta^2}{2}D}\right) \nu(T_1^2)&= 
    \left(G + \frac{\beta^2EH}{2(1 - \frac{\beta^2}{2} D)}\right) \left(\frac{1}{N} + \beta^2A_2^2)\right) + O_N(3)\\
    &\stackrel{(\ref{claim:T12})} {=}\left(G + \frac{\beta^2EH}{2(1 - \frac{\beta^2}{2} D)}\right) \frac{1}{NM_3} + O_N(3).
\end{align}
\YS{To check $A_1^2 \geq 0$, we will show $G \geq 0$
\[G = \nu_0(\epsilon_{1, 3}\epsilon_{1, 4} - \epsilon_{1, 2}\epsilon_{3, 4}) = \E[\inner{\epsilon_{3, 4}}\left(\inner{\epsilon_{1, 1}} - \inner{\epsilon_{1}}^2\right)] \geq 0\]}
\end{proof}
We now turn to the proof of Lemma \ref{lem:relation1}.
\begin{proof} [Proof of Lemma \ref{lem:relation1}]
Using \eqref{eq: var init} with $X = T_1$, we can rewrite $\nu(T_1^2)$ by introducing a new replica for each occurrence of $b$ and get

\begin{align} \label{T1^2 init}
    \nu(T_1^2) &= \nu((R_{1, 3} - R_{2, 3})(R_{1, 5} - R_{4, 5}))\\
    &= \frac{1}{N}\nu((\epsilon_{1,3} - \epsilon_{2,3})(\epsilon_{15} - \epsilon_{4, 5})) + \nu'_0((\epsilon_{13} - \epsilon_{23}) (R^-_{1,5} - R^-_{4, 5})) +  O_N(3).
\end{align}
For the first term, not that by symmetry, $\nu((\epsilon_{1,3} - \epsilon_{2,3})\epsilon_{4, 5}) = 0$. Thus we have 

To expand the second term, we use \eqref{derivative-2} with $N = 5$ gives
\begin{align*}
    \nu'_0((\epsilon_{13} - \epsilon_{23}) (R^-_{1,5} - R^-_{4, 5})) =& \frac{\beta^2}{2}\sum_{a, b \in [10]} \text{sgn}(a, b)\nu_0((\epsilon_{13} - \epsilon_{23})\epsilon_{ab})\nu_0((R^-_{a, b} - \mu_{a,b})(R^-_{1,5} - R^-_{4, 5}))
    + R_{5, T_1^2}.
\end{align*}

Many terms will vanish due to $\nu_0((\epsilon_{13} - \epsilon_{23})\epsilon_{ab}) = 0$. We will see that the non-vanishing pairs of replica $(a, b)$ introduce some structures that correspond to either $T_1$ or $S_1$. 

To capture which pair of $(a, b)$ having $\nu_0((\epsilon_{13} - \epsilon_{23})\epsilon_{ab}) \neq 0$, let's expand the product into two terms. Observe the value of $\nu_0(\epsilon_{13}\epsilon_{ab})$ is characterized by the type of multiset $\{1, 3, a, b\}$ and that the replica $2$ in $\nu_0(\epsilon_{23}\epsilon_{ab})$ is equivalent to replica $1$ in $\nu_0((\epsilon_{13}\epsilon_{ab})$. Thus we have that 
\[\nu_0((\epsilon_{1,3} - \epsilon_{2,3})\epsilon_{a,b}) \neq 0 \iff |\{a, b\} \cap \{1, 2\}| = 1.\]
What's left to do is to check $\nu_0((R^-_{a, b} - \mu_{a,b})(R^-_{1,5} - R^-_{4, 5}))$ for such pair $(a, b) \in \cc{C}_{10}$
\begin{itemize}
    \item If $a = b$: in this case $a \in \{1, 2\}$. Combine the two cases gives 
    \[\frac{1}{2}\nu_0((\epsilon_{1,3} - \epsilon_{2,3})\epsilon_{1,1})\nu_0((R^-_{1, 1} - R^-_{2, 2})(R^-_{1,5} - R^-_{4, 5})) \stackrel{\ref{cor: last spin}} {=} \frac{1}{2}H \nu(S_1T_1) + O_N(3).\]
    \item For $\{a, b\} \in \{\{1, 3\}, \{2, 3\}\}$, we have 
    \[\nu_0((\epsilon_{1,3} - \epsilon_{2,3})\epsilon_{1,3})\nu_0((R^-_{1, 3} - R^-_{2, 3})(R^-_{1,5} - R^-_{4, 5})) \stackrel{\ref{cor: last spin}} {=} F\nu(T_1^2) + O_N(3).\]
    \item Now we count the case when $a \in \{1, 2\}$,  $b \notin \{1, 2, 3\}$. Here is where the rectangles appear. Recall that for each of the $5$ replicas, we introduce a new replica. Let's index them with $\{k + 5: k \leq 5\}$.
    Gather terms for $b \in \{4, 9\}$ (equivalently $\{5,10\}$)
    \[\nu_0((\epsilon_{1,3} - \epsilon_{2,3})\epsilon_{1,4})\nu_0((R^-_{1, 4} - R^-_{2, 4} - R^-_{1, 9} + R^-_{2, 9})(R^-_{1,5} - R^-_{4, 5})).\]
    Using \eqref{eq:1st approx} and Lemma \ref{lemma:last spin}, we can rewrite the second term with $T_{k, l},T_k, T_l$ involving those new replicas, 
    \[\nu_0((R^-_{1, 4} - R^-_{2, 4} - R^-_{1, 9} + R^-_{2, 9})(R^-_{1,5} - R^-_{4, 5})) = \nu((T_{1,4} - T_{2,4} - T_{1,9} + T_{2,9})(T_{1,5} - T_{4,5} + T_1 - T_4)).\]
    We see that there are no even moments of $T_{k, l}$ here, thus this term is $O_N(3)$ by Lemma \ref{lem:var T12}. For $b \in \{5,10\}$, 
     \begin{align*}
         & \nu_0((R^-_{1, 5} - R^-_{2, 5} - R^-_{1, 10} + R^-_{2, 10})(R^-_{1,5} - R^-_{4, 5})) \\
         \stackrel{\ref{cor: last spin}}{=} & \nu((T_{1,5} - T_{2,5} - T_{1, 10} + T_{2, 10})(T_{1,5} - T_{4,5} + T_1 - T_4))
         =  \nu(T_{1,5}^2).
     \end{align*}
    Thus the total contribution from this case is 
    \[G\nu(T_{1, 5}^2) + O_N(3).\]
    \item Now we left with the cases $a \in \{1, 2\}$ and $b \in \{6, 7, 8\}$ which are the new replica corresponds to $\{1, 2, 3\}$. Those terms, WLOG, are 
    \[3\nu_0((\epsilon_{1,3} - \epsilon_{2,3})\epsilon_{1,4})\nu_0((-R^-_{1, 6} + R^-_{2, 6})(R^-_{1,5} - R^-_{4, 5})).\]
    Note that since the new replica is not used by our second copy of $T_1$, namely $R^-_{1,5} - R^-_{4, 5}$, this term can be written as 
    \[-3G \nu(T_1^2) + O_N(3).\]
\end{itemize}
Combining all the terms for the second term,
\begin{align*}
    \nu'_0((\epsilon_{1,3} - \epsilon_{2,3}) (R^-_{1,5} - R^-_{4, 5})) =&  \frac{\beta^2}{2}H\nu(S_1T_1)+ \beta^2 F\nu(T_1^2)\\
    &+ \beta^2 G\nu(T_{12}^2) -3 \beta^2 G \nu(T_1^2) + O_N(3).
\end{align*}
Plugging this back into \eqref{T1^2 init}, we have
\begin{align*}
    \nu(T_1^2) =& \frac{1}{N}\nu((\epsilon_{1,3} - \epsilon_{2,3})((\epsilon_{1,5} - \epsilon_{4,5}))) + \nu((\epsilon_{1,3} - \epsilon_{2,3}) (R^-_{1,5} - R^-_{4, 5})), \\
    =& G (\frac{1}{N} + \beta^2 \nu(T_{1,2}^2)) + \frac{\beta^2}{2}H\nu(S_1T_1) + \beta^2 \left[F - 3 G\right]\nu(T_1^2) + O_N(3).\
\end{align*}
Rearranging gives \eqref{eq1},
\[(1 - \beta^2 (F - 3G)) \nu(T_1^2)= G (\frac{1}{N} + \beta^2 \nu(T_{12}^2)) + \frac{\beta^2}{2} H \nu(S_1T_1)+ O_N(3).\]
Plug in $\nu(T_{12}^2) = A_2^2 + O_N(3)$ gives \eqref{eq1}. 
\begin{remark} \label{SK T_1^2}
In SK, the mixed term $S_1T_1$ vanishes. If we look at the constant for $T_1^2$, 
\[F = \nu_0((\epsilon_{13} - \epsilon_{23})\epsilon_{13}) = b(2) - b(1),\]
\[G = \nu_0((\epsilon_{13} - \epsilon_{23})\epsilon_{14}) = b(1) - b(0).\]
Combining them, we get back the original constants $1 - 4q + 3\hat{q}$, which is one of the "eigenvalues". Thus we get the second moment of $T_1$ (see the equation (1.259) in~\cite{Tal11}).
\end{remark}
\paragraph{A way of writing covariance of $S_1, T_1$}
To handle the occurrence of $\nu(S_1T_1)$ in the final expression, we will use the symmetry of spin to write 
\begin{align} \label{S1T1}
    \nu(S_1T_1) = \nu((\epsilon_{1}^2 - \epsilon_2^2)(R_{1,4} - R_{3,4})) = \frac{1}{N}\nu((\epsilon_{1}^2 - \epsilon_2^2)\epsilon_{1,3}) + \nu((\epsilon_{1}^2 - \epsilon_2^2)(R^-_{1,4} - R^-_{3,4})).
\end{align}
This type of expansion helps reduce the moment of $S_1$. As shown above in $\nu(T_1^2)$, to control the second term, it is enough to look at $\nu'_0((\epsilon_{1}^2 - \epsilon_2^2)(R^-_{14} - R^-_{34}))$. 
\begin{align*}
    \nu'_0((\epsilon_{1}^2 - \epsilon_{2}^2) (R^-_{1,4} - R^-_{3, 4})) =& \frac{\beta^2}{2}\sum_{a, b} \text{sgn}(a, b)\nu_0((\epsilon_{1}^2 - \epsilon_{2}^2)\epsilon_{a,b})\nu_0((R^-_{a, b} - \mu_{a,b})(R^-_{1, 4} - R^-_{3, 4}))\\
    &+ R_{4, S_1T_1}.
\end{align*}
Observe that 
\[\nu_0((\epsilon_{1}^2 - \epsilon_{2}^2)\epsilon_{a,b}) \neq 0 \iff |\{a, b\} \cap \{1, 2\}| = 1,\]
Let's iterate over those pairs $(a, b)$:
\begin{itemize}
    \item For $a = b$: either $a = b = 1$ or $a = b = 2$,
    \[\frac{1}{2}\nu_0((\epsilon_{1}^2 - \epsilon_{2}^2)\epsilon_{1}^2)\nu_0((R^-_{1, 1} - R^-_{2, 2})(R^-_{1, 4} - R^-_{3, 4})) \stackrel{(\ref{lemma:last spin})} {=}  \frac{1}{2}D \nu(S_1T_1) + O_N(3).\]
    \item For $|\{a, b\} \cap \{1, 2\}| = 1$, assume $a \in \{1, 2\}$ and $b \in \{3, 4, \cdots 8\}$. As shown above, for $b \in \{3, 7\}$ or $\{4, 8\}$, we have
    \[\nu_0((\epsilon_1^2 - \epsilon_2^2)\epsilon_{1,3})\nu_0((R^-_{1,3} - R^-_{2,3} + R^-_{1,7} - R^-_{2,7})(R^-_{1, 4} - R^-_{3, 4})) = O_N(3),\]
    and 
    \[\nu_0((\epsilon_1^2 - \epsilon_2^2)\epsilon_{13})\nu_0((R^-_{1,4} - R^-_{2,4} + R^-_{1,8} - R^-_{2,8})(R^-_{1, 4} - R^-_{3, 4})) = E\nu(T_{1,4}^2) + O_N(3).\]
    For $b \in \{5, 6\}$, we have
    \[-\nu_0((\epsilon_1^2 - \epsilon_2^2)\epsilon_{1,3})\nu_0((R^-_{1,5} - R^-_{2,5})(R^-_{1,4} - R^-_{2, 4})) = -E \nu(T_1^2) + O_N(3).\]
\end{itemize}
Thus
\begin{align*}
    \nu'_0(S_1T_1) =& \frac{\beta^2}{2} D\nu(S_1T_1) + \beta^2 E\nu(T_{1,2}^2) -2\beta^2 E \nu(T_1^2) + O_N(3).
\end{align*}
Plugging this back to the equation \eqref{S1T1} gives \eqref{eq2}
\begin{align}
    (1 - \frac{\beta^2}{2} D)\nu(S_1T_1) =& (\frac{1}{N} + \beta^2\nu(T_{12}^2))E - 2\beta^2 E \nu(T_1^2) + O_N(3),
\end{align}
Plugging in $\nu(T_{12}^2) = A_2^2 + O_N(3)$ gives \eqref{eq2}.
\end{proof}
\begin{lemma}\label{lem: S1 var}
For $\beta \leq \beta'$, we have
    \[
       \nu(S_1^2) = B_1^2 + O_N(3),
    \]
    where 
    \[B_1^2 = \frac{DM_1 - 2\beta^2E^2}{NM} = \frac{2}{N\beta^2}\left(\frac{M_1}{M} - 1\right)\]
\end{lemma}
To prove the Lemma~\ref{lem: S1 var}, we need to show the following two relations.
\begin{lemma} \label{lem:relation2}
We have 
\[\left(1 - \frac{\beta^2}{2} D\right)\nu(S_1^2) = \frac{1}{N}D -2\beta^2 E \cdot  \nu(S_1T_1) + O_N(3),\]
\[(1 - \beta^2 (F - 3G))\nu(S_1T_1)= \frac{1}{N}E +\frac{\beta^2}{2} H \cdot \nu(S_1^2) + O_N(3).\]
\end{lemma}
\begin{proof}[Proof of Lemma \ref{lem: S1 var}]
As in the $\nu(T_1^2)$ case, Lemma \ref{lem: S1 var} follows from combining the above two relations and the definition of $M_1, M$.
 \[
        \left(\frac{(1 - \frac{\beta^2}{2} D)(1 - \beta^2 (F - 3G)) + \beta^4EH}{(1 - \beta^2 (F - 3G))}\right)\nu(S_1^2) = \frac{D((1 - \beta^2 (F - 3G))) - 2\beta^2E^2}{N(1 - \beta^2 (F - 3G))} + O_N(3).
    \]
Rearrange gives for $B_1^2 = \frac{2}{N\beta^2}(\frac{M_1}{M} - 1)$, $\nu(S_1^2) = B_1^2 + O_N(3)$.

\YS{To show $B_1^2 \geq 0$, it's enough to check $M \leq M_1$. Recall the definition of $M_1, M_2, M$ in Definition \ref{constants} 
\begin{align*}
    M - M_1 = M_1(M_2 - 1)+ \beta^4E^2 &= - \beta^2\frac{D}{2}(1 - \beta^2 (F - G - 2G)) + \beta^4E^2\\
    &= -\frac{\beta^2D}{2}M_3 - \beta^4 DG + \beta^4E^2
\end{align*}
Note that $D := \E[\inner{\epsilon_1^4} - \inner{\epsilon_1^2}^2] \geq 0$ is the variance of $\epsilon_1^2$ under gibbs average. Note that $G$ is the variance of $(\epsilon_1 - b)b$ and $E$ is the covariance of $\epsilon_1^2$ and $(\epsilon_1 - b)b$ under gibbs average. Since covariance matrix is PSD, the determinant $GD - E^2 \geq 0$. Plug the inequalities back into the above equation gives
\[M - M_1 \leq 0\]}
\end{proof}
\begin{proof}[Proof of Lemma~\ref{lem:relation2}]
The proof is similar to the previous case. Denote $\epsilon_{k, k} = (\sigma^k_N)^2$,
\begin{align} \label{eq:s12}
    \nu(S_1^2) &= \nu((R_{1, 1} - R_{2, 2})(R_{1, 1} - R_{3, 3})),\\
    &= \frac{1}{N}\nu((\epsilon_{1,1} - \epsilon_{2,2})(\epsilon_{1,1} - \epsilon_{3,3})) + \nu((\epsilon_{1,1} - \epsilon_{2,2})(R^-_{1,1} - R^-_{3,3})),\\
    &= \frac{1}{N}\nu((\epsilon_{1,1} - \epsilon_{2,2})\epsilon_{1,1} ) + \nu((\epsilon_{1,1} - \epsilon_{2,2})(R^-_{1,1} - R^-_{3,3})).
\end{align}
To control the second term, observe that by \eqref{eq:2nd approx}, and $\nu_0((\epsilon_{1,1} - \epsilon_{2,2})(R^-_{1,1} - R^-_{3,3})) = 0$,
\[\nu((\epsilon_{1,1} - \epsilon_{2,2})(R^-_{1,1} - R^-_{3,3})) = \nu'_0((\epsilon_{1,1} - \epsilon_{2,2})(R^-_{1,1} - R^-_{3,3})) + O_N(3).\]
By \eqref{derivative-2} 
\begin{align*}
    \nu'_0((\epsilon_{11} - \epsilon_{22})(R^-_{1,1} - R^-_{3,3})) =& \frac{\beta^2}{2}\sum_{a, b} \text{sgn}(a, b) \nu_0((\epsilon_{1,1} - \epsilon_{2,2})\epsilon_{a, b})\nu_0((R^-_{a, b} - Q_{a, b}) (R^-_{1,1} - R^-_{3,3}))\\
    & + R_{3, S_1^2}.
\end{align*}
Note that 
\[\nu_0((\epsilon_{1,1} - \epsilon_{2,2})\epsilon_{a, b}) \neq 0 \iff |\{a, b\} \cap \{1, 2\}| = 1.\]
To count the contribution for all such $a, b$, 
\begin{itemize}
    \item For $a = b$, combine the contribution of two terms gives 
    \[\frac{\beta^2}{2} D \nu_0((R^-_{1, 1} - R^-_{2, 2})(R^-_{1,1} - R^-_{3,3})) \stackrel{(\ref{lemma:last spin}), \eqref{eq:1st approx}} {=} \frac{\beta^2}{2} D \nu(S_1^2) + O_N(3).\] 
    \item If $a \neq b$, WLOG, suppose $a \in \{1, 2\}$ and $b \in \{3, 6\}$: by Lemma \ref{lemma:last spin}
    \[\beta^2 E \nu_0((R^-_{13} - R^-_{2, 3} - R^-_{1, 6} + R^-_{2, 6})(R^-_{1,1} - R^-_{3,3})) = \beta^2 E \nu_0((R_{13} - R_{2, 3} - R_{1, 6} + R_{2, 6})(R_{1,1} - R_{3,3})) + O_N(3).\]
    Rewrite the last part using \ref{def:q}, 
    \begin{align*}
        &\beta^2 E \nu_0((R_{13} - R_{2, 3} - R_{1, 6} + R_{2, 6})(R_{1,1} - R_{3,3})) + O_N(3),\\
        =& \beta^2 E \nu_0((T_{13} - T_{2, 3} - T_{1, 6} + T_{2, 6})(S_{1} - S_{3}) + O_N(3),\\
        =& O_N(3).
    \end{align*}
    \item For $a \in \{1, 2\}$ and $b \in \{4, 5\}$, combine the two terms gives 
    \[-2\beta^2 E \nu_0((R^-_{1, 4} - R^-_{2, 4})(R^-_{1,1} - R^-_{3,3})) = -2\beta^2 E \nu(S_1T_1) + O_N(3).\]
\end{itemize}
Plug this back in \eqref{eq:s12}
\begin{align}
    \nu(S_1^2) &= \frac{1}{N}\nu((\epsilon_{1,1} - \epsilon_{2,2})\epsilon_{1,1}) + \frac{\beta^2}{2} D \nu(S_1^2) -2\beta^2 E \nu(S_1T_1) + O_N(3),\\
    &= \frac{1}{N}D + \frac{\beta^2}{2} D \nu(S_1^2) -2\beta^2 E \nu(S_1T_1) + O_N(3).
\end{align}

\paragraph{Alternative way of writing $\nu(S_1T_1)$} We've seen one way of decomposing $S_1T_1$ in lemma \ref{lem:relation1}, which reduces the moment of $S_1$. While we may directly apply \eqref{eq2} here, we show another way of decomposing $S_1T_1$ by reducing the moment of $T_1$, as it will be helpful in the general case. The idea is same
\begin{align*}
    \nu(S_1T_1) = \nu((R_{1, 1} - R_{2, 2})(\epsilon_{1,4} - \epsilon_{3, 4})) = \frac{1}{N}\nu((\epsilon_{1, 1} - \epsilon_{2, 2})(\epsilon_{1,4} - \epsilon_{3, 4})) + \nu((R^-_{1, 1} - R^-_{2, 2})(\epsilon_{1,4} - \epsilon_{3, 4}))
\end{align*}
We then rewrite the second term as before  
\begin{align*}
    \nu((R^-_{1, 1} - R^-_{2, 2})(\epsilon_{1,4} - \epsilon_{3, 4})) = \nu_0'((R^-_{1, 1} - R^-_{2, 2})(\epsilon_{1,4} - \epsilon_{3, 4})) + O_N(3)\\
    = \frac{\beta^2}{2}\sum_{a, b} \nu_0((R^-_{1, 1} - R^-_{2, 2})(R^-_{a, b} - Q_{a, b}))\nu_0(\epsilon_{a, b}(\epsilon_{1,4} - \epsilon_{3, 4}))
\end{align*}
As shown in Lemma \ref{lem:relation1}, 
\[\nu_0((\epsilon_{1,4} - \epsilon_{3,4})\epsilon_{a,b}) \neq 0 \iff |\{a, b\} \cap \{1, 3\}| = 1\]
Let's go over all cases of such size two subsets:
\begin{itemize}
    \item If $a = b$: this term gives 
    \[\frac{\beta^2}{2} H \nu_0((R^-_{1, 1} - R^-_{3, 3})(R^-_{1, 1} - R^-_{2, 2})) \stackrel{(\ref{lemma:last spin}), \eqref{eq:1st approx}} {=} \frac{\beta^2}{2} H \nu(S_1^2) + O_N(3)\]
    \item For $\{a, b\} \in \{\{1, 4\}, \{3, 4\}\}$, we have 
    \[\beta^2 F \nu_0((R^-_{1, 1} - R^-_{2, 2})(R^-_{1, 4} - R^-_{3, 4})) \stackrel{(\ref{lemma:last spin}), \eqref{eq:1st approx}} {=} \beta^2 F \nu(S_1T_1) + O_N(3)\]
    \item Now we count the case when $a \in \{1, 3\}$,  $b \notin \{1, 3, 4\}$. Gather terms for $b \in \{2, 6\}$ and rewrite  
    \begin{align*}
        \beta^2 G \nu_0((R^-_{1, 1} - R^-_{2, 2})(R^-_{1, 2} - R^-_{3, 2} - R^-_{1, 6} + R^-_{3, 6})) &\stackrel{(\ref{lemma:last spin}), \eqref{eq:1st approx}} {=} \beta^2 G \nu_0((S_1 - S_2)(T_{1, 2} - T_{3, 2} - T_{1, 6} + T_{3, 6}))\\
        &= O_N(3)
    \end{align*}
    \item Now we are left with $a \in \{1, 3\}$ and $b \in \{5, 7, 8\}$ 
    \[-3\beta^2G\nu_0((R^-_{1, 1} - R^-_{2, 2})(R^-_{1, 5} - R^-_{3, 5})) \stackrel{(\ref{lemma:last spin}), \eqref{eq:1st approx}} {=} -3\beta^2G \nu(S_1T_1) + O_N(3)\]
\end{itemize}
Combine we get
\[\nu(S_1T_1)= \frac{1}{N}E +\frac{\beta^2}{2} H \nu(S_1^2) + \beta^2 (F - 3G)\nu(S_1T_1) + O_N(3)\]
\end{proof}
\subsubsection{Covariance: $S_1T_1$ term}
\begin{lemma} \label{lem: var s1t1}
For $\beta \leq \beta'$, we have
    \[
    \nu(S_1T_1) = C_1^2 + O_N(3),
    \]
    where 
    \[
    C_1^2 := \frac{E}{NM}.
    \]
\end{lemma}
\begin{proof}
Note that one can deduce $\nu(S_1T_1)$ from both \ref{lem:relation1} and $\ref{lem:relation2}$. 
From lemma \ref{lem:relation1}, we get 
\begin{align*}
    \nu(S_1T_1) &= \frac{E}{M_2}\left[\frac{1}{NM_3} - 2\beta^2 A_1^2\right] + O_N(3)\\
    &=\frac{E}{M_2}\left(\frac{1}{NM_3} - \frac{1}{N}\left(\frac{1}{M_3} - \frac{M_2}{M}\right)\right) + O_N(3) = \frac{E}{MN} + O_N(3).
\end{align*}
From lemma \ref{lem:relation2}
\begin{align*}
    \nu(S_1T_1) &= \frac{1}{M_1} \left(\frac{E}{N} + \frac{\beta^2}{2}HB_1^2\right) + O_N(3) = \frac{E}{NM} + O_N(3),
\end{align*}
where the last equality follows from $E = H$.
\end{proof}
        

\section{General moments computation} \label{sec: general}
In Section \ref{sec:var}, we obtained the variance and covariance: $\nu(T_{1,2}^2), \nu(T_1^2), \nu(S_1^2)$, $\nu(S_1 T_1)$ by rewriting moments with lower order terms. In this section, we extend this idea to general moments of $T_{1, 2}, T_1, S_1, T, S$. 
\begin{lemma}[Formal version of Lemma \ref{lem:general}] \label{lem: general formal}
Fix an integer $n$, consider the following sets of integers $\{h(k, l): 1 \leq k < l \leq n\}$, $\{h(k): 1 \leq k \leq n\}$ and $\{h'(k): 1 \leq k \leq n\}$ and $h', h$. Let
\[H := \sum_{1 \leq k < l \leq n}h(k, l) + \sum_{1 \leq k \leq n} h(k) + \sum_{1 \leq l \leq n} h'(l) + h + h',\]
let $g_X$ be a centered Gaussians vector where the index $X$ belongs to
\[\{T_{k, l}, T_k, S_k, T, S: 1 \leq k < l \leq n\},\]
and its covariance matrix is
\[\cov(g_X,g_Y) = \begin{cases}
    A_2^2, & \text{ if } X = Y = T_{k, l},\\
    A_1^2, & \text{ if } X = Y = T_{k},\\
    A_0^2, & \text{ if } X = Y = T,\\
    B_1^2, & \text{ if } X = Y = S_k,\\
    B_0^2, & \text{ if } X = Y = S,\\
    C_1^2, & \text{ if } \{X, Y\} = \{T_k, S_k\},\\
    C_0^2, & \text{ if } \{X, Y\} = \{T, S\}.\\
\end{cases}\]
then for $\beta < \beta' \leq \tilde{\beta}$, we have
\begin{align*}
    \nu\left(\Pi_{k, l}T_{k, l}^{h(k, l)}\Pi_{k}T_k^{h(k)}T^{h}\Pi_{l}S_l^{h'(l)}S^{h'}\right) &= \E[\Pi_{k, l}g_{T_{k, l}}^{h(k, l)}\Pi_{k}g_{T_k}^{h(k)}\Pi_{l}g_{S_l}^{h'(l)}g_T^{h}g_S^{h}] + O_N(H + 1).
\end{align*}
\end{lemma}
Similar to the proof of CLT in SK model, the proof for Lemma \ref{lem: general formal} consists of three parts, first we separate any $T_{k, l}$ terms s.t. $h(k, l) > 0$ from the mixed moments, then $(T_k, S_k)$ terms and then the $(T, S)$ term. This is based on the Lemma \ref{lem:independence}, which states that the set of random variables is pairwise independent besides $(T_k, S_k)$ for some $k \in [n]$ and $(T, S)$. Thus, we expect the mixed moments to be decomposed into 
\[\Pi_{k, l}\nu\left(T_{k, l}^{h(k, l)}\right) \cdot \Pi_{k}\nu\left(T_k^{h(k)}S_k^{h'(k)} \right) \cdot \nu(T^hS^{h'}).\]
Therefore, it is then enough to characterize the moments of the form: $T_{k, l}^{h(k, l)}, T_k^{h(k)}S_k^{h'(k)}, T^hS^{h'}$. The formal statements can be found in Theorem \ref{thm: general T12}, \ref{thm:peel off T_kS_k} and \ref{thm: ST moments}. 

Before we start the proofs, we will introduce the necessary notations to index each term within the mixed moments. Let's first rewrite each term using (self-)overlaps by the expansion given in Claim \ref{claim: basis to overlap}. For $v \in [H]$,  denote $V_v = \{v_1, v_2, \cdots\}$ as the set of replicas appeared in the corresponding term $U_v$. Define $U_v$ as
\[U_v := \begin{cases}
    R_{v_1, v_2} - R_{v_1, v_4} - R_{v_3, v_2} + R_{v_3, v_4}, & \text{ if } v \text{-th term corresponds to } T_{k, l}, \\
    R_{v_1, v_3} - R_{v_2, v_3}, & \text{ if } v \text{-th term corresponds to } T_{k},\\
    R_{v_1, v_1} - R_{v_2, v_2}, & \text{ if } v \text{-th term corresponds to } S_{l},\\
    R_{v_1, v_2} - p, & \text{ if } v \text{-th term corresponds to } T,\\
    R_{v_1, v_1} - q,&  \text{ if } v \text{-th term corresponds to } S.
\end{cases}\]
Then the general moments can be rewritten as
\begin{align} \label{eq: general rewrite}
    \nu\left(\Pi_{k, l}T_{k, l}^{h(k, l)}\Pi_{k}T_k^{h(k)}T^{h}\Pi_{l}S_l^{h'(l)}S^{h'}\right) = \nu(\Pi_{v \geq 1}U_v).
\end{align}
By symmetry of spins, we can replace one of $U_v$ by the same expression on the last spin. To do this, let's define the following notation: For $v \in [H]$, let
\[\epsilon(v) := \begin{cases}
    \epsilon_{v_1, v_2} - \epsilon_{v_1, v_4} - \epsilon_{v_3, v_2} + \epsilon_{v_3, v_4}, & \text{ if } v \text{-th term corresponds to } T_{k, l}, \\
    \epsilon_{v_1, v_3} - \epsilon_{v_2, v_3}, & \text{ if } v \text{-th term corresponds to } T_{k},\\
    \epsilon_{v_1, v_1} - \epsilon_{v_2, v_2}, & \text{ if } v \text{-th term corresponds to } S_{l},\\
    \epsilon_{v_1, v_2} - p, & \text{ if } v \text{-th term corresponds to } T,\\
    \epsilon_{v_1, v_1} - q,&  \text{ if } v \text{-th term corresponds to } S.
\end{cases}\]
Finally, define $U^-_v$ as the part of $U_v$ that doesn't depend on the last spin 
\[U^-_v := U_v - \frac{1}{N}\epsilon(v).\]
Finally, following the cavity method, one should try to separate as many parts of the expression that depend on the last spin as possible. To this end, let's further decompose \eqref{eq: general rewrite} as
\begin{align} \label{eq: general rewrite2}
    \nu\left(\Pi_{k, l}T_{k, l}^{h(k, l)}\Pi_{k}T_k^{h(k)}T^{h}\Pi_{l}S_l^{h'(l)}S^{h'}\right) &= \nu(\Pi_{v \geq 1}U_v)\\
    & = \nu(\epsilon(1)\Pi_{v > 1}U^-_v) + \frac{1}{N}\sum_{u \geq 2} \nu(\epsilon(1)\epsilon(u)\Pi_{v \neq 1, u}U^-_v) + O_N(H + 1).
\end{align}

\subsection{Induction on $T_{k, l}$}\label{ssec:induct-Tkl}
We first generalize the result in Lemma \ref{lem:var T12} to show that $T_{1, 2}$ behaves like independent Gaussian w.r.t. other basis terms. 
\begin{theorem} \label{thm: general T12}
For $\beta < \beta'$, we have 
\[\nu(\Pi_{(k, l)} T_{k, l}^{h(k, l)} \Pi_{k} T_k^{h(k)} T^{h} \Pi_k S_k^{h'(k)} S^{h'}) = \Pi_{(k, l)} a(h(k, l))A_2^{h(k, l)} \nu(\Pi_{k} T_k^{h(k)} T^{h} \Pi_k S_k^{h'(k)} S^{h'}) + O_N(H + 1),
\]
where $a(x) = \E[g^x]$ with $g \sim N(0,1)$.
\end{theorem}
The proof of this theorem is the same as its analog in the SK model. We include the proof for completeness.
\begin{proof}
The proof goes by inducting on $\sum_{k, l}h(k, l)$. WLOG, we assume that $h(1, 2) \geq 1$ and reduce the moment of $T_{1, 2}$. For the sake of simplicity, let's define a function $g_{1, 2}(x)$ that tracks the moment of $T_{1, 2}$ s.t. 
\[g_{1, 2}(x) = \begin{cases}
    T_{1, 2}^x \Pi_{(k, l) \not\equiv (1, 2)}T_{k, l}^{h(k, l)}\Pi_{k}T_k^{h(k)}T^{h}\Pi_{l}S_l^{h'(l)}S^{h'}, & \text{if } h(1, 2) \geq 0,\\
    0, & \text{ otherwise}.
\end{cases}\]
Assume that $U_1$ corresponds to a copy of $T_{1, 2}$. Using \eqref{eq: general rewrite}, we have 
\begin{align}
    \nu(g(h(1, 2))) = \nu(\epsilon(1)\Pi_{v > 1}U^-_v) + \frac{1}{N}\sum_{u \geq 2} \nu(\epsilon(1)\epsilon(u)\Pi_{v \neq 1, u}U^-_v) + O_N(H + 1),
\end{align}
where 
\[\epsilon(1) = \epsilon_{1_1, 1_2} - \epsilon_{1_1, 1_4} - \epsilon_{1_3, 1_2} + \epsilon_{1_3, 1_4}.\]
The second term is again approximated by $\nu'_0(\cdot)$ using \eqref{eq:2nd approx}. 
\begin{lemma}\label{derivative T_kl}
For $\beta<\beta'$, suppose $h(1, 2) \geq 1$ and $U_1$ corresponds to a copy of $T_{1, 2}$, we have
\[\nu'_0(\epsilon(1)\Pi_{v > 1}U^-_v) = \beta^2 A \cdot \nu(g_{1, 2}(h(1, 2))) + O_N(H + 1).\]
\end{lemma}
The proof of the above lemma is essentially the same as in the SK model; we include it in the appendix for completeness. 
For the first term, by \eqref{eq:1st approx}, 
\begin{align*}
    \frac{1}{N}\sum_{u \geq 2} \nu(\epsilon(1)\epsilon(u)\Pi_{v \neq 1, u}U^-_v) &= \frac{1}{N}\sum_{u \geq 2} \nu_0(\epsilon(1)\epsilon(u)\Pi_{v \neq 1, u}U^-_v) + O_N(H + 1),\\
    &= \frac{1}{N}\sum_{u \geq 2} \nu_0(\epsilon(1)\epsilon(u))\nu_0(\Pi_{v \neq 1, u}U^-_v) + O_N(H + 1).
\end{align*}
Note that following a similar arguement as in Lemma \ref{lem:independence}, $\nu_0(\epsilon(1)\epsilon_{a, b}) \neq 0$ only when $a, b$ appears in the expression of $\epsilon(1)$. However, by construction, $1_2, 1_3$ do not appear in any other terms besides $U_1$, thus the only possible pair of replicas that appears in $U_1$ that also appears in other terms are when $(u_1, u_2) \equiv (1, 2)$.
\begin{align}\label{eq: constant T12}
    \nu_0(\epsilon(1)\epsilon(u)) = \begin{cases}
    A, & \text{ if } U_u \text{ corresponds to a copy of } T_{1, 2},\\
    0, & \text{otherwise}.
    \end{cases}
\end{align}
Summing up all non-zero terms and applying Corollary \ref{cor: last spin}, we have
\[ \frac{1}{N}\sum_{u \geq 2} \nu(\epsilon(1)\epsilon(u)\Pi_{v \neq 1, u}U^-_v) = (h(1, 2) - 1)\cdot \frac{A}{N} \cdot \nu(g_{1, 2}(h(1, 2) - 2)) + O_N(H + 1).\]
Combine with Lemma \ref{derivative T_kl} and rearrange gives  
\begin{align} \label{eq: recursive T12}
   \nu(g(1, 2)) &= (h(1, 2) - 1)\cdot \frac{A}{N(1 - \beta^2 A)}\cdot  \nu(g_{1, 2}(h(1, 2) - 2)) + O_N(H + 1),\\
   &= (h(1, 2) - 1)\cdot A_2^2 \cdot \nu(g_{1, 2}(h(1, 2) - 2)) + O_N(H + 1).
\end{align}
Now we are ready to perform induction. If $h(1, 2) = 1$ holds since $a(1) = 0$. For higher moments, we apply the inductive hypothesis on $\nu(g_{1, 2}(h(1, 2) - 2))$. Plug this back in \ref{eq: recursive T12} and denote $h'(k, l)$ as the moments of $T_{k, l}$ in $g_{1, 2}(h(1, 2) - 2)$, 
\begin{align*}
    \nu(g(1, 2)) &= (h(1, 2) - 1) \cdot A_2^2 \cdot \nu(g_{1, 2}(h(1, 2) - 2)) + O_N(H + 1),\\
    &= (h(1, 2) - 1)\cdot A_2^2 \cdot \Pi_{(k, l)}a(h'(k, l))A_2^{h'(k, l)} \nu(\Pi_{k} T_k^{h(k)} T^{h} \Pi_k S_k^{h'(k)} S^{h'}) + O_N(H + 1),\\
    &= \Pi_{(k, l)}a(h(k, l))A_2^{h(k, l)} \nu(\Pi_{k} T_k^{h(k)} T^{h} \Pi_k S_k^{h'(k)} S^{h'}) + O_N(H + 1).
\end{align*}
where the last equality follows from $a(h(1, 2)) = (h(1, 2) - 1)a(h'(1, 2))$.
\end{proof}

\subsection{Recursive relation for correlated "basis"}\label{ssec:induct-TS}
As we mentioned in Section \ref{sec: proof outline}, our goal is to obtain a recursive relation for moments of the basis as in \cite[Chapter 1.10]{Tal11}. We need to do a little more work for $T_1, S_1$ and $T, S$ because we expect them to be correlated. We describe the additional step here before delving into the moment computations.

By the Gaussian integration by part (see e.g. \cite{Tal11} A.4), suppose $[g_1, g_2] \sim \cc{N}(0, \Sigma)$ and some contents $a, b \geq 2$, the two ways of expanding $\E[g_1^ag_2^b]$ are
\begin{align} \label{eq: gaussian intuition}
    \E[g_1^ag_2^b] &= (a - 1) \Sigma_{1, 1}\E[g_1^{a - 2}g_1^b] + b \Sigma_{1, 2}\E[g_1^{a - 1}g_1^{b - 1}],\\
    &= a \Sigma_{1, 2}\E[g_1^{a - 1}g_1^{b - 1}] + (b - 1) \Sigma_{1, 2}\E[g_1^{a}g_1^{b - 2}].
\end{align} 
As we saw in Section \ref{sec:var}, the cavity method almost gives the above type of relations. The cavity method allows us to decouple the last spin at time $0$. Using the symmetry of spins allows us to rewrite one of the terms using only the last spin, as in e.g. \eqref{T1^2 init}, thus almost reducing the moment by $1$. However, this does not reduce the number of replicas the non-trivial part of \eqref{T1^2 init} depends on, and approximation given by Lemma \ref{lem:derivative} may increase the moment of some terms.

To get some intuition, let's consider the case $T_1^3$, recall that we can rewrite $\nu(T_1^3)$ by applying symmetry of spin on one of the $T_1$ 
\[\nu(T_1^3) = \nu_1((\epsilon_{1, 3} - \epsilon_{2, 3})T_1^2).\]
Becauase $T_1^2$ is an order $2$ function, we need to invoke \eqref{eq:2nd approx} and use $\nu'_0((\epsilon_{1, 3} - \epsilon_{2, 3})T_1^2)$ to get a good enough approximation of $\nu_1((\epsilon_{1, 3} - \epsilon_{2, 3})T_1^2)$. 
By Lemma \ref{lem:derivative}, even though $\sigma^2$ is only used by the first term i.e. $(\epsilon_{1, 3} - \epsilon_{2, 3})$, we still need to consider their contribution in $\nu'_0(T_1^3)$. Gathering terms correspond to $(a, b) \in \{(1, 1), (2, 2)\}$ gives 
\[
\nu_0((R_{1, 1} - R_{2, 2})T_1) \equiv \nu_0(T_1 S_1).
\]
Even though $S_1$ does not appear in the initial expression, taking the derivative at time $0$ would introduce a term where the moment of $S_1$ is $1$.

Still, if we restrict our attention to some fixed replica $k$, we can expand the mixed moments of $\{(T_i, S_i): i \in [n]\}$ or $(T, S)$ in two different ways similar to \eqref{eq: gaussian intuition}. 
Intuitively, this follows the pair $(T_k, S_k)$ (and $(S, T)$) being independent of all other basis terms that don't depend on replica $k$, as indicated in Lemma \ref{lem:independence}. We prove this formally in Lemma \ref{lem: relation T1S1} and \ref{lem: relation ST} below. 

To avoid repetition, let's first characterize the condition under which the relations given by the cavity method imply the desired recursive relation for proving CLT. 

\begin{lemma} \label{lem: recurse condition}
    Consider two sets of constants $\alpha_2, \alpha_1, \alpha_0$ and $\beta_2, \beta_1, \beta_0$. Suppose there exist $H \geq 0$ and $C \in \R$. Suppose a function $f: \mathbb{Z} \times \mathbb{Z} \to \R$ with $f(h, h') = 0$ if $h < 0$ or $h' < 0$ and $f(0, 0) = C$ satisfies the following relation:
    For $h, h' > 0$ and $(h, h') \neq (0, 0)$, 
    \begin{align}
        f(h, h') &= \alpha_2 (h - 1) f(h - 2, h') + \alpha_1 h' f(h - 1, h' - 1) + \alpha_0 f(h - 1, h' + 1) + O_N(h + h' + H + 1), \label{eq: recurse1}\\
        &= \beta_2 (h' - 1) f(h, h' - 2) + \beta_1 h f(h - 1, h' - 1) + \beta_0 f(h + 1, h' - 1) + O_N(h + h' + H + 1).\label{eq: recurse2}
    \end{align}
    If the sets of constants satisfy
    \begin{align} \label{eq: consistency cond}
        \alpha_1 + \alpha_0 \beta_2 = \beta_1 + \beta_0 \alpha_2 := \gamma,
    \end{align} 
    then we can find a set of constants $C(2, 0), C(0, 2), C(1, 1)$ s.t. $f$ satisfies the following recursive relations
    \begin{align}
        f(h, h') =& (h - 1)C(2, 0) f(h - 2, h') + h' C(1, 1) f(h - 1, h' - 1) + O_N(h + h' + H + 1),\\
        =& (h' - 1)C(0, 2) f(h, h' - 2) + h C(1, 1) f(h - 1, h' - 1) + O_N(h + h' + H + 1).
    \end{align}
    with 
    \[C(2, 0) = \frac{\alpha_2 + \alpha_0\beta_1}{1 - \alpha_0 \beta_0} \quad C(0, 2) =  \frac{\beta_2 + \beta_0\alpha_1}{1 - \alpha_0 \beta_0}  \quad C(1, 1) = \frac{\gamma}{1 - \alpha_0\beta_0} \]
\end{lemma}
\begin{proof}
The idea is to use \eqref{eq: recurse2} and \eqref{eq: recurse1} to rewrite $f(h - 1, h' + 1)$. One thing we need to check is that the resulting constants in front of $f(h - 1, h' + 1)$ are the same in both equations. 

\paragraph{Base case:} Note that $f(1, 0) = f(0, 1) = O_N(h + h' + H + 1)$. 
We will first handle the case when $(h, h') \in \{(2, 0), (0, 2), (1, 1)\}$. Plug in the corresponding values for $h, h'$ gives \eqref{eq: recurse1} and \eqref{eq: recurse2} gives
\[
f(2, 0) = \alpha_2f(0, 0) + \alpha_0 f(1, 1) + O_N(3 + H),
\]
\[
f(0, 2) = \beta_2f(0, 0) + \beta_0 f(1, 1)+ O_N(3 + H),
\]
\[
f(1, 1) = \alpha_1f(0, 0) + \alpha_0f(0, 2)+ O_N(3 + H) = \beta_1f(0, 0) + \beta_0 f(2, 0)+ O_N(3 + H).
\]
Solve the above system of linear equations gives 
\begin{align*}
    (1 - \alpha_0\beta_0)f(2, 0) = (\alpha_2 + \alpha_0\beta_1)f(0, 0) + O_N(3 + H),\\
    (1 - \alpha_0\beta_0)f(0, 2) = (\beta_2 + \beta_0\alpha_1)f(0, 0) + O_N(3 + H).
\end{align*}
By \eqref{eq: consistency cond} and the expression for $f(2, 0)$, $f(0, 2)$, we have
\begin{align*}
    f(1, 1) &= \left(\alpha_1 + \alpha_0\frac{\beta_2 + \beta_0\alpha_1}{1 - \alpha_0\beta_0}\right)f(0, 0)+ O_N(3 + H) = \left(\frac{\alpha_1 + \alpha_0\beta_2}{1 - \alpha_0\beta_0}\right) f(0, 0) + O_N(3 + H),\\
    &= \left(\frac{\beta_1 + \beta_0 \alpha_2}{1 - \alpha_0\beta_0}\right) f(0, 0) + O_N(3 + H) = \left(\beta_1 + \beta_0 \frac{\alpha_2 + \alpha_0\beta_1}{1 - \alpha_0 \beta_0}\right) f(0, 0) + O_N(3 + H).
\end{align*}
Rearrange the above equations gives
\[f(2, 0) = C(2, 0)f(0, 0) + O_N(3 + H), \ \ f(0, 2) = C(0, 2) f(0, 0)+ O_N(3 + H), \ \text{and} \ f(1, 1) = C(1, 1) f(0, 0) + O_N(3 + H).\]
For the case when $h' = 0$ and $h \geq 0$, the equation \eqref{eq: recurse1} becomes
\begin{align*}
    f(h, 0) &= \alpha_2 (h - 1) f(h - 2, 0) + \alpha_0 f(h - 1, 1) + O_N(h + 1 + H),\\
    &= \alpha_2 (h - 1) f(h - 2, 0) + \alpha_0 \left[ \beta_1 (h - 1)f(h - 2, 0) + \beta_0 f(h, 0)\right] + O_N(h + 1 + H).
\end{align*}
Rearrange and plug in the values of $C(2, 0)$ gives  
\[
f(h, 0) = (h - 1) C(2, 0) f(h - 2, 0) + O_N(h + 1 + H).
\]
For $h' \geq 0 $ and $h = 0$, the same arguement applies by starting from \eqref{eq: recurse2} with $h = 0$. 
\[f(0, h') = (h' - 1) C(0, 2) f(0, h' - 2) + O_N(h' + 1 + H).\]
\paragraph{General case:} Assume $h, h' \geq 1$. Start from \eqref{eq: recurse1} and expand $f(h - 1, h' + 1)$ using \eqref{eq: recurse2} gives
\begin{align*}
    f(h, h') =& \alpha_2 (h - 1) f(h - 2, h') + \alpha_1 h' f(h - 1, h' - 1) + \alpha_0 f(h - 1, h' + 1) + O_N(h + h' + H + 1),\\
    =& \alpha_2 (h - 1) f(h - 2, h') + \alpha_1 h' f(h - 1, h' - 1) + O_N(h + h' + H + 1)\\
    &+ \alpha_0 \left(\beta_2 h' f(h - 1, h' - 1) + \beta_1 (h - 1) f(h - 2, h') + \beta_0 f(h, h')\right)  + O_N(h + h' + H + 1).
\end{align*}
Rearrange, we have 
\[
f(h, h') = (h - 1)C(2, 0) f(h - 2, h') + h'C(1, 1) f(h - 1, h' - 1) + O_N(h + h' + H + 1).
\]
Similarily, start from \eqref{eq: recurse2} instead and repeat the above arguement gives
\[
f(h, h') = (h' - 1)C(0, 2) f(h, h' - 2) + hC(1, 1)f(h - 1, h' - 1) + O_N(h + h' + H + 1).
\]
\end{proof}
\YS{\begin{claim} \label{claim: psd}
If $C(2, 0) \geq 0$ and $C(0, 2) \geq 0$, then $M \succeq 0$ iff $\frac{\alpha_2\beta_2 - \alpha_1\beta_1}{1- \alpha_0\beta_0} \geq 0$
\end{claim}
\begin{proof}
Let $\lambda_1, \lambda_2$ be the eigenvalue of $M$, Since $\text{tr}(M) = \lambda_1 + \lambda_2 = C(2, 0) + C(0, 2) \geq 0$, we have $\lambda_1 \geq 0$ or $\lambda_2 \geq 0$. To show $M \succeq 0$, it is enough to show $\lambda_1\lambda_2 = \text{det}(M) \geq 0$
\begin{align*}
    \text{det}(M) &= C(2, 0)C(0, 2) - C(1, 1)^2 \\
    &= \frac{1}{(1- \alpha_0\beta_0)^2}\left((\alpha_2 + \alpha_0\beta_1)(\beta_2 + \beta_0\alpha_1) - \gamma^2\right)\\
    &= \frac{1}{(1- \alpha_0\beta_0)^2}\left[(\alpha_2 + \alpha_0\beta_1)(\beta_2 + \beta_0\alpha_1) - (\alpha_1 + \alpha_0 \beta_2)(\beta_1 + \beta_0 \alpha_2)\right]\\
    &=  \frac{1}{(1- \alpha_0\beta_0)^2}(\alpha_2\beta_2 - \alpha_1\beta_1)(1 - \alpha_0\beta_0)\\
    &= \frac{1}{1- \alpha_0\beta_0}(\alpha_2\beta_2 - \alpha_1\beta_1)
\end{align*}
where the second equality follows from the definition of $\gamma$.
\end{proof}}
\subsubsection{Induction on $T_k$ and $S_k$}
In this section, we examine the mixed moments of $T_k$ and $S_k$. Assume that there are $n$ replicas in total and the moments of $T_{k, l}$, $h(k, l) = 0$, for all $1 \leq k < \leq n$. Denote the total moments of $T_k$ and $S_l$ as
\[h_T = \sum_k h(k), \quad h_S = \sum_l h'(l), \quad H_1 = h_T + h_S + h + h'.\]
\begin{theorem}\label{thm:peel off T_kS_k}
Let $\{(g_{T_k}, g_{S_k}): k \in [n]\}$ be i.i.d Gaussian with mean $[0, 0] $ and covariance matrix 
\[ \Sigma_1 := \begin{bmatrix}
A_1^2 & C_1^2 \\
C_1^2 & B_1^2
\end{bmatrix} ,\] 
We have
\[\nu(\Pi_{k} T_k^{h(k)} \Pi_{l}S_l^{h'(l)}T^hS^{h'}) = \left(\Pi_{k}\E[g_{T_k}^{h(k)}g_{S_k}^{h'(k)}]\right)\nu(T^hS^{h'}) + O_N(H_1 + 1)\]
\end{theorem}
Following the symmetry of replicas and the idea from Lemma \ref{lem: recurse condition}, we will try to expand higher-order mixed moments by reducing the moment of $T_k$ or $S_k$ for some fixed replica $k$.

WLOG, suppose $h(1) + h'(1) > 0$. Let $g_1: \mathbb{Z}^2 \to \R$ be the function that tracks the moment of $T_1$ and $S_1$ only.
\begin{align} \label{def: g_1}
    g_1(x, y) := \begin{cases}
    \nu(T_1^{x}S_1^{y} \Pi_{k > 1} T_k^{h(k)} \Pi_{l > 1}S_l^{h'(l)}T^hS^{h'}), & \text{ if } x, y \geq 0,\\
    0, & \text{ otherwise}. 
    \end{cases}
\end{align}

The lemma below is a generalization of Lemma \ref{lem:relation1} and \ref{lem:relation2}. 
\begin{lemma}\label{lem: relation T1S1}
For $h(1) > 1$, $h'(1) \geq 0$,
\begin{align} \label{eq: general reduce T1}
   M_1\nu(g(h(1), h'(1))) &= \frac{\beta^2}{2}H \nu(g_1(h(1) - 1, h'(1) + 1)\\
    &+ (h(1) - 1)\left(\beta^2 G A_2^2 + \frac{G}{N}\right)\nu(g_1(h(1) - 2, h'(1)))\\
    &+ h'(1) \frac{H}{N}\nu(g_1(h(1)- 1, h'(1) - 1)\\
    &+ O_N(H_1 + 1).
\end{align}
For $h(1) \geq 0$, $h'(1) > 1$,
\begin{align}\label{eq: general reduce S1}
    M_2\nu(g_1(h(1), h'(1))) &= -2\beta^2E \nu(g_1(h(1) + 1, h'(1) - 1)\\
    &+ h(1)\left(\beta^2 E A_2^2 + \frac{E}{N}\right)\nu(g_1(h(1) - 1, h'(1) - 1))\\
    &+ (h'(1) - 1)\frac{D}{N}\nu(g_1(h(1), h'(1) - 2))\\
    &+ O_N(H_1 + 1).
\end{align}
\end{lemma}
\begin{remark}
    Observe that \eqref{eq: general reduce T1} is again a generalization of the recursive relation in SK model, see (1.320) (1.323) in~\cite{Tal11}. To compare our result to the SK model, recall that $F \equiv b(2) - b(1) = 1 - q$, $G \equiv b(1) - b(0) = q - \hat{q}$ and $H = 0$,  \eqref{eq: general reduce T1}  becomes
\begin{align*}
    (1 - \beta^2(1 - 4q + 3\hat{q}))\nu(f(h(1), 0))= &(h(1) - 1)(q - \hat{q})(\beta^2A^2_1 + \frac{1}{N}) \nu(f(h(1) - 2, 0))+ O_N(H + 1).
\end{align*}
\end{remark}
The proof of Lemma \ref{lem: relation T1S1} can be found in the following section. Let's first see how one can deduce Theorem \ref{thm:peel off T_kS_k} Lemma \ref{lem: relation T1S1}. Following the intuition from the beginning of this section, we apply Lemma \ref{lem: recurse condition} to get recursive relations that are of the same form as Gaussian moments.
\begin{proof}[Proof of Theorem \ref{thm:peel off T_kS_k}]
We apply Lemma \ref{lem: recurse condition} to the recursive relations in Lemma \ref{lem: relation T1S1} with the following constants
\[\alpha_2 = \frac{G}{M_1} \left(\beta^2 A_2^2 + \frac{1}{N}\right), \quad \alpha_1 = \frac{H}{NM_1} ,\quad \text{and} \ \ \alpha_0 = \frac{\beta^2}{2M_1}H.\]
\[\beta_2 = \frac{D}{M_2N}, \quad \beta_1 = \frac{E}{M_2}\left(\beta^2  A_2^2 + \frac{1}{N}\right), \quad \text{and} \ \ \beta_0 = -\frac{2\beta^2E}{M_2}.\]
To apply Lemma \ref{lem: recurse condition}, we need to check the consistance condition, then compute $C(2, 0), C(0, 2)$ and $C(1, 1)$ to get the final result.

The consistency condition is verified by Lemma \ref{lem: var s1t1}. If $g_1(x, y) = T_1^xS_1^y$, then we recover results from Section \ref{sec:var}. We include the computation for general cases for completeness. 

\paragraph{To check the consistency condition} 
Note that by Claim \ref{claim:T12}, we have 
\[\alpha_2 = \frac{G}{M_1}\frac{1}{NM_3}, \quad \beta_1 = \frac{E}{M_2}\frac{1}{NM_3}.\]
To verify \eqref{eq: consistency cond} holds for the current set of constants, check that
\begin{align*}
    \alpha_1 + \alpha_0 \beta_2 &= \frac{H}{NM_1}(1 + \frac{\beta^2}{2}\frac{D}{M_2}) = \frac{H}{NM_1M_2},\\
    \beta_1 + \beta_0 \alpha_2 &= \frac{E}{NM_2M_3}(1 - \frac{2\beta^4G}{M_1}) \\
    & = \frac{E}{NM_2M_3}\frac{M_1 - 2\beta^2G}{M_1},\\
    &= \frac{E}{NM_2M_3}\frac{M_3}{M_1} = \alpha_1 + \alpha_0 \beta_2.
\end{align*}
The only things left are to compute $C(2, 0)$, $C(0, 2)$ and $C(1, 1)$. First check that the common denominator for $C(2, 0)$, $C(0, 2)$ and $C(1, 1)$ is
\[1 - \alpha_0\beta_0 = 1 + \frac{\beta^4E^2}{M_1M_2} = \frac{M}{M_1M_2}.\]
The three constants are then given by
\begin{align*}
    C(2, 0) &= \frac{\alpha_2 + \alpha_0\beta_1}{1 - \alpha_0 \beta_0} = \frac{M_1M_2}{M}\left(\frac{G}{M_1}\frac{1}{NM_3} + \frac{\beta^2H}{2M_1}\frac{E}{M_2}\frac{1}{NM_3}\right),\\
    &= \frac{M_1M_2}{M} \cdot \frac{GM_2 + \frac{\beta^2}{2}E^2}{NM_1M_2M_3} = \frac{GM_2 + \frac{\beta^2}{2}E^2}{MN},\\
    &= A_1^2.
\end{align*}
\begin{align*}
    C(0, 2) &= \frac{\beta_2 + \beta_0\alpha_1}{1 - \alpha_0 \beta_0}= \frac{M_1M_2}{M}\left(\frac{D}{M_2N} - \frac{2\beta^2E}{M_2}\frac{H}{NM_1}\right),\\
    &= \frac{M_1M_2}{M}\frac{DM_1 - 2\beta^2E^2}{M_1M_2N} = \frac{DM_1 - 2\beta^2E^2}{NM},\\
    &= B_1^2.
\end{align*}
\begin{align*}
    C(1, 1) = \frac{\gamma}{1 - \alpha_0\beta_0} = \frac{M_1M_2}{M}\cdot \frac{H}{NM_1M_2} = \frac{H}{MN} = C_1^2.
\end{align*}
By Lemma \ref{lem: recurse condition}, we have
\begin{align} \label{eq: gaus T1S1}
    \nu(g_1(h(1), h'(1))) = & (h(1) - 1)A_1^2  \nu(g_1(h(1) - 2, h'(1))) + h'(1)C_1^2 \nu(g(h(1) - 1, h'(1) - 1)) + O_N(H_1 + 1),\\
        =& (h'(1) - 1)B_1^2 \nu(g(h(1), h'(1) - 2)) + h(1) C_1^2 \nu(g(h(1) - 1, h'(1) - 1)) + O_N(H_1 + 1).
\end{align}
The proof then is completed by induction on $H_1$. The statement holds if $H_1 = 1$, since $\E[g_{T_k}] = \E[g_{S_k}] = 0$. For $H_1 \geq 2$: suppose $h(1) + h(2) \geq 2$. The terms on the right-hand side of \eqref{eq: gaus T1S1} have total moment $H_1 - 2$. We can apply the inductive hypothesis on the right-hand side gives
\begin{align*}
    \nu(g_1(h(1), h'(1))) &= \left[(h(1) - 1)A_1^2  \E[g_{T_1}^{h(1) - 2} g_{S_1}^{h'(1)}] + h'(1)C_1^2 \E[g_{T_1}^{h(1) - 1} g_{S_1}^{h'(1) - 1}] \right] \cc{C} + O_N(H_1 + 1).
\end{align*}
where $\cc{C} = \left(\Pi_{k > 1}\E[g_{T_k}^{h(k)}g_{S_k}^{h'(k)}]\right)\nu(T^hS^{h'})$. 

Similarly, we get
\begin{align*}
    \nu(g_1(h(1), h'(1))) &= \left[(h'(1) - 1)B_1^2  \E[g_{T_1}^{h(1)} g_{S_1}^{h'(1) - 2}] + h(1)C_1^2 \E[g_{T_1}^{h(1) - 1} g_{S_1}^{h'(1) - 1}] \right] \cc{C} + O_N(H_1 + 1).
\end{align*}
from the second recursive relation from \eqref{eq: gaus T1S1}.
Note that mixed moments of $g_{T_1}, g_{S_1}$ satisfies \eqref{eq: gaussian intuition} with $a = h(1)$ and $b = h'(1)$.

\YS{We have that $\Sigma_1 \succeq 0$: Suppose not, let $\Delta$ be the matrix of corresponding error terms so that $\Sigma_0 + \Delta = \begin{bmatrix}\nu(T_1^2) & \nu(S_1T_1) \\ \nu(S_1T_1) & \nu(S_1^2)\end{bmatrix}$. Note that $\Sigma_1 = \frac{1}{N} \Sigma'_1$ where $\Sigma'_1$ is a matrix independent from $N$. Since $\Sigma_1 \not\succeq 0$, there exists a negative eigenvalue  $\lambda$ of $\Sigma_1$. Let $x$ be the corresponding eigenvector, then we have $x^T \Sigma_1 x = \frac{1}{N} \lambda < 0$. But $x \Sigma_0 x^T + x\Delta x^T \leq \frac{1}{N}\lambda + 2\max_{i, j \in [2]} |\Delta[i][j]| < 0$ where the last inequality follows from $\max_{i, j \in [2]} |\Delta[i][j]| = O_N(3) \ll \frac{1}{N}$. This is a contradiction since $\Sigma_0 + \Delta \succeq 0$.
}

\YS{To check $\begin{bmatrix}
    A_1^2 & C_1^2\\
    C_1^2 & B_1^2
\end{bmatrix} \succeq 0$: By Lemma  \ref{lem: T1 var} and \ref{lem: S1 var}, $A_1^2, B_1^2 \geq 0$. By claim \ref{claim: psd}, it is enough to verify  $\frac{\alpha_2\beta_2 - \alpha_1\beta_1}{1- \alpha_0\beta_0} \geq 0$. 
Note that 
\begin{align*}
    \alpha_2\beta_2 - \alpha_1\beta_1 &= \frac{GD}{M_1M_2N} \left(\beta^2 A_2^2 + \frac{1}{N}\right) - \frac{E^2}{M_1M_2N}\left(\beta^2  A_2^2 + \frac{1}{N}\right)\\
    &= \frac{1}{M_1M_2N}\left(\beta^2  A_2^2 + \frac{1}{N}\right) \left(GD - E^2\right)\\
    &= \frac{1}{M_1M_2M_3N^2}\left(GD - E^2\right)
\end{align*}
where the last equality follow from \ref{claim:T12}.

Observe that $\left(GD - E^2\right)$ is the expected determinant of the covariance matrix of $(\epsilon_1 - b)b$ and $\epsilon_1^2 - \tilde{b}$. To see this, recall that, by definition of $b, \tilde{b}$, $\E[\inner{(\epsilon_1 - b)b}_0] = 0$ and $\E[\inner{(\epsilon_1^2 - \tilde{b})}_0] = 0$. 
The covariance matrix is then given by, 
\[\E[\inner{((\epsilon_1 - b)b)^2}_0] = \E[\inner{(\epsilon_{1, 3} - \epsilon_{2, 3})(\epsilon_{1, 4} - \epsilon_{5, 4})}_0] = G\]
\[\E[\inner{(\epsilon_1^2 - \tilde{b})^2}_0] = \E[\inner{(\epsilon_{1}^2 - \epsilon_{2}^2)(\epsilon_{1}^2 - \epsilon_{3}^2)}_0] = D\]
\[\E[\inner{(\epsilon_1 - b)b(\epsilon_1^2 - \tilde{b})}_0] = \E[\inner{(\epsilon_{1, 3} - \epsilon_{2, 3})(\epsilon_1^2 - \epsilon_4^2)}_0] = E\]
}
This completes the induction.
\end{proof}
\begin{remark}
    Note that if $g_1(x, y) = T_1^xS_1^y$, then we have $\nu(g_1(2, 0)) = \nu(T_1^2)$, $\nu(g_1(0, 2)) = \nu(S_1^2)$ and $\nu(g_1(1, 1)) = \nu(T_1S_1)$. In this case, $\nu(g(0, 0)) = 1$. Lemma \ref{lem: recurse condition} says that the same relation holds for a more general initial expression $g(0, 0)$. In the proof above, we recovered the same set of constants from \eqref{eq: general reduce T1} and \eqref{eq: general reduce S1} as from the variance calculation in Section \ref{sec:var}.
\end{remark}
\subsubsection{Proof of Lemma \ref{lem: relation T1S1}}
Recall the definition of $U_v, \epsilon(v), U^-_v$ from the beginning of this section and that we denote $V_v = \{v_1, v_2, \cdots\}$ as the set of replicas appear in term $U_v$. The first step is to approximate $g(h(1), h'(1))$ by \eqref{eq: general rewrite2}
\begin{align} 
    \nu(g_1(h(1), h'(1))) &= \nu(\Pi_{1 \leq v \leq H_1} U_v),\\
    &= \nu(\epsilon(1) \Pi_{1 < v \leq H_1} U^-_v) + \frac{1}{N}\sum_{1 < v \leq H_1}\nu(\epsilon(1)\epsilon(v)\Pi_{u \neq 1, v}U^-_u).\label{eq: T1S1 init}
\end{align}
The idea is to apply the cavity method when $U_1$ corresponds to $T_1$ and $S_1$. 

\paragraph{To reduce the moment of $T_1$}
Suppose $U_1$ corresponds to $T_1$,  then
\[\epsilon(1) = \epsilon_{1_1, 1_3} - \epsilon_{1_2, 1_3}.\]
As usual, the first term in \eqref{eq: T1S1 init} is an order $H_1 - 1$ function, thus needs to be approximated using \eqref{eq:2nd approx} as shown in Lemma \ref{derivative T_k}. The proof is deferred to Appendix.
\begin{lemma}[First order derivative structure for $T_k$] \label{derivative T_k}
For $h(1) \geq 1$ and $h'(1) \geq 0$,  suppose $U_1$ corresponds to a copy of $T_1$ 
\begin{align*}
    \nu(\epsilon(1)\Pi_{v > 1}U_v) =& \beta^2(F-3G)\nu(g_1(h(1), h'(1)) \\
    &+\frac{\beta^2}{2}H \nu(g_1(h(1) - 1, h'(1) + 1)\\
    &+ \beta^2 (h(1) - 1)G A_2^2 \nu(g_1(h(1) - 2, h'(1)))\\
    & + O_N(H_1 + 1).
\end{align*}  
\end{lemma}
The second term is approximated using \eqref{eq:1st approx}.
\begin{align*}
    \frac{1}{N}\sum_{1 < v \leq H_1}\nu(\epsilon(1)\epsilon(v)\Pi_{u \neq 1, v}U^-_u) &= \frac{1}{N}\sum_{1 < v \leq H_1}\nu_0(\epsilon(1)\epsilon(v))\nu_0(\Pi_{u \neq 1, v}U^-_u) + O_N(H_1 + 1).
\end{align*}
By Lemma \ref{lem:independence}, $\nu_0(\epsilon(1)\epsilon(v)) = 0$ if $U_v$ doesn't correspond to $T_1$ or $S_1$. Moreover,
\[\nu_0(\epsilon(1)\epsilon(v)) = \begin{cases}
    G, & \text{ if } U_v \text{ corresponds to } T_1,\\
    H, & \text{ if } U_v \text{ corresponds to } S_1.
\end{cases}\]
There are $(h(1) - 1)$ terms $T_1$ and $h'(1)$ terms $S_1$ in $\Pi_{v \geq 1}U^-_v$. Summing up all terms of the same type and applying Corollary \ref{cor: last spin} on all terms,
\begin{align*}
    \frac{1}{N}\sum_{1 < v \leq H_1}\nu(\epsilon(1)\epsilon(v)\Pi_{u \neq 1, v}U^-_u) =& (h(1) - 1)\frac{G}{N}\nu(g_1(h(1) - 2, h'(1)),\\
    &+ h'(1) \frac{H}{N}\nu(g_1(h(1)- 1, h'(1) - 1) + O_N(H_1 + 1).
\end{align*}

Combine results for both first and second term of \eqref{eq: T1S1 init} and rearrange gives  \ref{eq: general reduce T1} 
\YS{remove this in the final version. Here for checking the result}
\begin{align*}
    (1 - \beta^2(F-3G))\nu(g(h(1), h'(1))) &= \frac{\beta^2}{2}H \nu(g_1(h(1) - 1, h'(1) + 1)\\
    &+ (h(1) - 1)\left(\beta^2 G A_2^2 + \frac{G}{N}\right)\nu(g_1(h(1) - 2, h'(1)))\\
    &+ h'(1) \frac{H}{N}\nu(g_1(h(1)- 1, h'(1) - 1)\\
    &+ O_N(H_1 + 1).
\end{align*}
\paragraph{To reduce the moment of $S_1$} Suppose, in this case, $U_1$ corresponds to $S_1$ term.
\[\epsilon(1) = \epsilon_{1_1, 1_1} - \epsilon_{1_2, 1_2}.\]
The first term in \eqref{eq: T1S1 init} is characterized by the following lemma. 
\begin{lemma}[First order derivative structure for $S_k$] \label{derivative S_k}
If $h'(1) \geq 1$ and $h(1) \geq 0$, suppose $U_1$ corresponds to a copy of $S_1$ 
\begin{align*}
    \nu(\epsilon(1)\Pi_{v > 1}U_v) =& \frac{\beta^2}{2}D \nu(g_1(h(1), h'(1)))\\
    &+ \beta^2 h(1)E A_2^2\nu(g_1(h(1) - 1, h'(1) - 1))\\
    & -2\beta^2E \nu(g_1(h(1) + 1, h'(1) - 1)\\
    &+ O_N(H_1 + 1)).
\end{align*}   
\end{lemma}
For the second term, again, we have
\begin{align*}
    \frac{1}{N}\sum_{1 < v \leq H_1}\nu(\epsilon(1)\epsilon(v)\Pi_{u \neq 1, v}U^-_u) &= \frac{1}{N}\sum_{1 < v \leq H_1}\nu_0(\epsilon(1)\epsilon(v))\nu_0(\Pi_{u \neq 1, v}U^-_u) + O_N(H_1 + 1).
\end{align*}
Check that 
\[\nu_0(\epsilon(1)\epsilon(v)) = \begin{cases}
    D, & \text{ if } U_v \text{ corresponds to a copy of } S_1,\\
    E,  & \text{ if } U_v \text{ corresponds to a copy of } T_1,\\
    0, & \text{ otherwise. }
\end{cases}\]
Plug in the above equation gives 
\begin{align*}
    \frac{1}{N}\sum_{1 < v \leq H_1}\nu(\epsilon(1)\epsilon(v)\Pi_{u \neq 1, v}U^-_u) &= (h'(1) - 1)\frac{D}{N}\nu(g_1(h(1), h'(1) - 2))\\
    &+ h(1) \frac{E}{N} \nu(g_1(h(1) - 1, h'(1) - 1)) + O_N(H_1 + 1).
\end{align*}
Combine the estimations of the two terms gives \eqref{eq: general reduce S1}\YS{similarily, remove this in the final version?}
\begin{align*}
    (1 - \frac{\beta^2}{2}D)\nu(g_1(h(1), h'(1))) &= -2\beta^2E \nu(g_1(h(1) + 1, h'(1) - 1)\\
    &+ h(1)\left(\beta^2 E A_2^2 + \frac{E}{N}\right)\nu(g_1(h(1) - 1, h'(1) - 1))\\
    &+ (h'(1) - 1)\frac{D}{N}\nu(g_1(h(1), h'(1) - 2))\\
    &+ O_N(H_1 + 1).
\end{align*}
\subsubsection{Induction on $T$ and $S$}
In this section, we consider functions in the form of $T^hS^{h'}$ for $h, h' \in \mathbb{N}_{\geq 0}$. As in previous sections, the idea is to write $T^hS^{h'}$ as a formula of $T^{h - 1}S^{h' - 1}$ and $\{T^{h - 2}S^{h'}, T^{h}S^{h'- 2}\}$. To this end, let's define 
\[g(h, h') = \begin{cases}
    T^hS^{h'}, & \text{ if } h, h' \geq 0,\\
    0, & \text{ otherwise. }
\end{cases}\]
\begin{theorem} \label{thm: ST moments}
Let $\{(g_{T}, g_{S})\}$ be a Gaussian vector with mean $[0, 0] $ and covariance matrix 
\[\begin{bmatrix}
A_0^2 & C_0^2 \\
C_0^2 & B_0^2
\end{bmatrix},\]
where $A_0^2, B_0^2, C_0^2$ are given in Theorem \ref{thm: var ST}. Then we have
   \[\nu(g(h, h')) = \E[g_T^hg_S^{h'}] + O_N(h + h' + 1).\]
\end{theorem}
The proof of Theorem \ref{thm: ST moments} uses the same idea as Theorem \ref{thm:peel off T_kS_k}: we first use cavity method to obtain a recursive relation, then apply Lemma \ref{lem: recurse condition} to see that moment of $S, T$ is the moments of a correlated Gaussian. The only difference lies in the structure of overlaps in cavity computation. Because of this difference, we will first introduce a more refined set of constants that will appear in the cavity computation, thus also the recursive relations of moments.

\paragraph{Constants}
To motivate the set of constants we need to compute the moment of $T, S$, recall that for variance computation, we started from \eqref{eq: var init}. 
\begin{align}
    \nu(X^2) &= \frac{1}{N}\nu(\epsilon_X^2) + \nu(\epsilon_X X^-),\\
    &= \frac{1}{N}\nu(\epsilon_X^2) + \nu'_0(\epsilon_X X^-) + O_N(3).
\end{align} 
By setting $X = T$ or $X = S$, we record the following constants corresponding to the expectation of the last spin. They mainly appears in $\nu'_0(\epsilon_X X^-)$ as a result of formula from \ref{lem:derivative}.
\begin{definition} \label{claim: constants2}
We record the following constants.
\begin{center}
\begin{tabular}{ |p{8cm}|p{8cm}|  }
 \hline
 \multicolumn{2}{|c|}{List of constants} \\
 \hline
 $I_1 = \nu_0((\epsilon_{12} - q)\epsilon_{12}) = \nu_0(\epsilon_{12}\epsilon_{12} - q^2)$ & \\
 \hline
 $I_2 = \nu_0((\epsilon_{12} - q)\epsilon_{13}) = \nu_0(\epsilon_{12}\epsilon_{13} - q^2)$ & \\
 \hline
 $I_3 = \nu_0((\epsilon_{12} - q)\epsilon_{34}) = \nu_0(\epsilon_{12}\epsilon_{34} - q^2)$& \\
 \hline
 $I_4 = \nu_0((\epsilon_{12} - q)\epsilon_{11}) = \nu_0(\epsilon_{12}\epsilon_{11} - pq)$ & $K_1 = \nu_0((\epsilon_{1, 1} - p)\epsilon_{12}) = \nu_0(\epsilon_{11}\epsilon_{12}) - pq = I_4$\\
 \hline
 & $K_2 = \nu_0((\epsilon_{1, 1} - p)\epsilon_{11}) = \nu_0(\epsilon_{11}\epsilon_{11}) - p^2$\\
 \hline
 $I_5 = \nu_0((\epsilon_{12} - q)\epsilon_{33}) = \nu_0(\epsilon_{12}\epsilon_{33} - pq)$ & $K_3 = \nu_0((\epsilon_{1, 1} - p)\epsilon_{23}) = \nu_0(\epsilon_{11}\epsilon_{23}) - pq = I_5$\\
 \hline
 &$K_4 = \nu_0((\epsilon_{1, 1} - p)\epsilon_{22}) = \nu_0(\epsilon_{11}\epsilon_{22}) - p^2$\\
 \hline
\end{tabular}
\end{center}
\end{definition}
The constants defined in \ref{constants} are linear combinations of the ones defined above.
\begin{claim}\label{claim: constants ST}
\[I_1 - I_2 = F,  \quad  I_2 - I_3 = G, \quad I_4 - I_5 = E,\]
\[K_1 - K_3 = E, \quad K_2 - K_4 = D.\]
\end{claim} 
\begin{proof}
The claim follows by checking the definition of in \ref{constants} and comparing with the ones in Definition \ref{claim: constants2}
\end{proof}
\paragraph{Proof of Theorem \ref{thm: ST moments}}
In this section, we prove Theorem \ref{thm: ST moments} and record the variance of $T, S$ as a special case. As in the proof of Theorem \ref{thm:peel off T_kS_k}, we first give two recursive formulas for mixed moments of $T, S$ using the cavity method. The proof of Theorem \ref{thm: ST moments} follows from rewriting the relations using Lemma \ref{lem: recurse condition}. The proof of the Lemma \ref{lem: relation ST} will be shown in the next subsection.
\begin{lemma} \label{lem: relation ST}
For $h \geq 1; h' \geq 0$, we have 
\begin{align} 
M_1\nu(g(h, h')) =& \beta^2E\nu(g(h - 1, h' + 1))\\
    &+ \beta^2(h - 1)\left[I_5C_1^2 + 2(2G -I_3)A_1^2 + I_3A_2^2 + \frac{1}{\beta^2N}I_3\right]\nu(g(h -2, h'))\\
    &+ \beta^2h'\left[\frac{1}{2}I_5B_1^2 + (2G -I_3)C_1^2 + \frac{1}{\beta^2N}I_5\right]\nu( g(h - 1, h' - 1))\label{eq: general reduce T}\\
    &+ O_N(h + h' + 1).
\end{align}
For $h \geq 0; h' \geq 1$
\begin{align}
   M_2\nu(g(h, h')) =&-\beta^2E\nu(g(h + 1, h' - 1))\\
   & + \beta^2 h\left[K_4 C_1^2 + 2(E - K_3)A_1^2 + K_3A_2^2 + \frac{1}{\beta^2 N} K_3\right]\nu( g(h - 1, h' - 1))\\
    &+ \beta^2(h' - 1)\left[\frac{1}{2}K_4  B_1^2 + (E - K_3)C_1^2 + \frac{1}{\beta^2 N} K_4\right]\nu( g(h, h' - 2))\label{eq: general reduce S}\\
    &+ O_N(h + h' + 1).
\end{align}
\begin{remark}
    First check that \eqref{eq: general reduce T} reduces to the equation (1.262) of~\cite{Tal11} in SK model: For SK mode, there's no self-overlap terms and $C_1^2 = 0$. To check other constants, $I_3 = \hat{q} - q^2$, 
    \[2G - I_3 = 2I_2 - 3I_3 = 2(q - q^2) - 3(\hat{q} - q^2) = 2q + q^2 - 3\hat{q}.\]
    Combined with Remark \ref{SK T_1^2}, we have
    \[\beta^2\left[I_5C_1^2 + 2(2G -I_3)A_1^2 + I_3A_2^2 + \frac{1}{\beta^2N}I_3\right] \equiv \frac{\hat{q} - q^2}{N} + \beta^2(\hat{q} - q^2)A^2 + 2\beta^2(2q + q^2 - 3\hat{q})B^2 + O_N(3).\]
    where $A, B, C$ are defined in \cite[Chapter 1.8]{Tal11}.
\end{remark}
\end{lemma}
To apply Lemma \ref{lem: recurse condition} on $f(h, h') = \nu(g(h, h'))$, we first need to check \eqref{eq: general reduce T} and \eqref{eq: general reduce S} satisfies the consistency condition. That is, the goal is to verify \eqref{eq: consistency cond} with
\[\alpha_2 = \frac{\beta^2}{M_1} \left[I_5C_1^2 + 2(2G -I_3)A_1^2 + I_3A_2^2 + \frac{1}{\beta^2N}I_3\right],\]
\[\alpha_1 = \frac{\beta^2}{M_1}\left[\frac{1}{2}I_5B_1^2 + (2G -I_3)C_1^2 + \frac{1}{\beta^2N}I_5\right],\]
\[\alpha_0 = \frac{\beta^2E}{M_1}.\]
and 
\[\beta_2 = \frac{\beta^2}{M_2}\left[\frac{1}{2}K_4  B_1^2 + (E - K_3)C_1^2 + \frac{1}{\beta^2 N} K_4\right],\]
\[\beta_1 = \frac{\beta^2}{M_2}\left[K_4 C_1^2 + 2(E - K_3)A_1^2 + K_3A_2^2 + \frac{1}{\beta^2 N} K_3\right],\]
\[\beta_0 = -\frac{\beta^2E}{M_2}.\]
We begin by recording some useful expressions of the constants that will simplify the proof.
\begin{claim} \label{claim:rel1}
\[2A_1^2 -  (A_2^2 + \frac{1}{\beta^2N})= -\frac{M_2}{\beta^2NM},\]
\[\frac{B_1^2}{2} + \frac{1}{\beta^2N} = \frac{M_1}{\beta^2MN},\]
\[M_1 - M_3 = 2\beta^2G.\]
\end{claim}
\begin{proof}
The third equation follows from the definition of $M_1$ and $M_3$.

For the first and second equation: By Claim \ref{claim:T12}, 
\[\beta^2A_2^2 + \frac{1}{N}= \frac{1}{NM_3}.\]
Recall the definition of $A_1^2, B_1^2$
\[A_1^2 := \frac{GM_2 + \frac{\beta^2}{2}EH}{M}\frac{1}{NM_3} = \frac{1}{2\beta^2N}\left(\frac{1}{M_3} - \frac{M_2}{M}\right),\]
\[B_1^2 = \frac{DM_1 - 2\beta^2E^2}{NM} = \frac{2}{N\beta^2}(\frac{M_1}{M} - 1).\]
Combine and rearrange to give the desired result.
\end{proof}
To simplify notations, denote $\text{LHS}$ and $\text{RHS}$ as 
\[\text{LHS}:= \alpha_1 + \alpha_0 \beta_2,   \quad \text{and} \ \ \beta_1 + \beta_0 \alpha_2 =: \text{RHS}.\]
We now begin to verify \eqref{eq: consistency cond} by comparing the coefficients in front of $K_4, I_5, I_3$ in $\text{LHS}$ and $\text{RHS}$. 

\paragraph{For $K_4$}
\begin{align*}
    \text{LHS} = \frac{\beta^4E}{M_1M_2}\left(\frac{1}{2}B_1^2 + \frac{1}{\beta^2N}\right) \stackrel{\ref{claim:rel1}}{=} \frac{\beta^4E}{M_1M_2}\frac{M_1}{\beta^2MN} = \frac{\beta^2}{M_2}\frac{E}{MN} \stackrel{\ref{lem: var s1t1}}{=}\frac{\beta^2C_1^2}{M_2} = \text{RHS}.
\end{align*}
\paragraph{For $I_5$}
Recall that by definition, $I_5 = K_3$. On one hand,
\begin{align*}
    \text{LHS} &=\frac{\beta^2}{M_1}\left(\frac{1}{2}B_1^2 + \frac{1}{\beta^2N}\right) - \frac{\beta^2E}{M_1} \frac{\beta^2}{M_2} C_1^2 \stackrel{\ref{claim:rel1}}{=} \frac{\beta^2}{M_1}\frac{M_1}{\beta^2MN} - \frac{\beta^4E}{M_1M_2}\frac{E}{MN},\\
    &= \frac{1}{MN}\left(1 - \frac{\beta^4E^2}{M_1M_2}\right).
\end{align*}
\begin{align*}
    \text{RHS} &= \frac{\beta^2}{M_2}(-2A_1^2 + A_2^2 + \frac{1}{\beta^2N}) - \frac{\beta^2E}{M_2}\frac{\beta^2}{M_1}C_1^2
    \stackrel{\ref{claim:rel1}}{=} \frac{\beta^2}{M_2}\frac{M_2}{\beta^2MN} - \frac{\beta^4E^2}{M_2M_1MN},\\
    &= \frac{1}{MN}(1 - \frac{\beta^4E^2}{M_1M_2}) = \text{LHS}.
\end{align*}
\paragraph{For $I_3$}
\begin{align*}
    \text{RHS} &= -\frac{\beta^2E}{M_2}\frac{\beta^2}{M_1} (-2A_1^2+ A_2^2 + \frac{1}{\beta^2N}) \stackrel{\ref{claim:rel1}}{=} -\frac{\beta^4E}{M_2M_1}\frac{M_2}{\beta^2MN}, \\
    &=  -\frac{\beta^2E}{M_1MN} = -\frac{\beta^2}{M_1}C_1^2 = \text{LHS}.
\end{align*}
\paragraph{The remaining terms} With some abuse of notations, check that
\begin{align*}
    \text{LHS} &= \frac{\beta^2}{M_1}2GC_1^2 + \frac{\beta^2E}{M_1}\frac{\beta^2}{M_2}EC_1^2 = \frac{C_1^2\left(2\beta^2GM_2 + \beta^4E^2\right)}{M_1M_2}
    \stackrel{\ref{lem: T1 var}}{=} \frac{2E\beta^2A_1^2M_3}{M_1M_2},\\
    \text{RHS} &= \frac{\beta^2}{M_2}2EA_1^2 - \frac{\beta^2E}{M_2}\frac{\beta^2}{M_1}4GA_1^2 = 2\beta^2EA_1^2\frac{(M_1 - \beta^22G)}{M_1M_2},\\
    & \stackrel{\ref{claim:rel1}}{=} 2\beta^2EA_1^2\frac{M_3}{M_1M_2} = \text{LHS}.
\end{align*}
This allows us to apply Lemma \ref{lem: recurse condition} to obtain a recursive relation for $\nu(T^hS^{h'})$. We start by computing the variance of $T$ and $S$.
\begin{theorem} \label{thm: var ST}
For $\beta < \beta'$, we have
     \[\nu(T^2) = A_0^2 + O_N(3),\]
    where 
    \begin{align*}
        A_0^2 &= \frac{\beta^2}{M}\left(\beta^2EK_4 + M_2I_5\right)C_1^2 + 2\frac{\beta^2}{M}\left(\beta^2E(E - K_3) + M_2 (2G -I_3)\right) A_1^2 \\
    &+ \frac{\beta^2}{M}\left(\beta^2EK_3 + M_2I_3\right)A_2^2 + \frac{1}{MN}\left(\beta^2E K_3 + M_2I_3\right).
    \end{align*}
And we further have
    \[\nu(S^2) = B_0^2 + O_N(3),\]
    where 
    \[B_0^2 = \frac{\beta^2}{2M}\left(M_1K_4 - \beta^2EI_5\right)B_1^2 + \frac{\beta^2}{M}\left(M_1(E-K_3)-E\beta^2(2G-I_3)\right)C_1^2 + \frac{1}{MN}\left(M_1K_4 - \beta^2EI_5\right).\]
Finally we also have 
    \[\nu(ST) = C_0 + O_N(3),\]
where
\[C_0 = \frac{\beta^2}{M}\left(\frac{1}{2}M_2I_5 + \beta^2EK_4\right)B_1^2 + \frac{\beta^2}{M}\left(M_2(2G -I_3) +\beta^2E(E - K_3)\right)C_1^2 + \frac{1}{MN} \left(M_2I_5 + \beta^2EK_4\right).\]
\end{theorem}
\begin{proof}
Since the coefficients in \eqref{eq: general reduce T} and \eqref{eq: general reduce S} satisfy the condition \eqref{eq: consistency cond}, we can apply Lemma \ref{lem: recurse condition} with $h, h' \in \{0, 2\}$ to obtain the desired result. First, the common denominator of $C(2, 0), C(0, 2), C(1, 1)$ is
\[1 - \alpha_0\beta_0 = 1 + \frac{\beta^4E^2}{M_1M_2} = \frac{M}{M_1M_2}.\]
For variance of $T$:
\begin{align*}
    (1 - \alpha_0\beta_0)\nu(T^2) &= (1 - \alpha_0\beta_0)\nu(g(2, 0)) = \alpha_2 + \alpha_0\beta_1\\
    & = \frac{\beta^2}{M_1} \left[I_5C_1^2 + 2(2G -I_3)A_1^2 + I_3A_2^2 + \frac{1}{\beta^2N}I_3\right]\\
    &+ \frac{\beta^4E}{M_1M_2}\left[K_4 C_1^2 + 2(E - K_3)A_1^2 + K_3A_2^2 + \frac{1}{\beta^2 N} K_3\right] + O_N(3).
\end{align*}
Rearrange gives 
\begin{align*}
    \nu(T^2) &= \frac{\beta^2}{M}\left(\beta^2EK_4 + M_2I_5\right)C_1^2 + 2\frac{\beta^2}{M}\left(\beta^2E(E - K_3) + M_2 (2G -I_3)\right) A_1^2 \\
    &+ \frac{\beta^2}{M}\left(\beta^2EK_3 + M_2I_3\right)A_2^2 + \frac{1}{MN}\left(\beta^2E K_3 + M_2I_3\right) + O_N(3).
\end{align*}
Next, we compute the variance of $S$:
\begin{align*}
    (1 - \alpha_0\beta_0)\nu(S^2) &= (1 - \alpha_0\beta_0)\nu(g(0, 2)) = \beta_2 + \beta_0\alpha_1\\
    &= \frac{\beta^2}{M_2}\left[\frac{1}{2}K_4  B_1^2 + (E - K_3)C_1^2 + \frac{1}{\beta^2 N} K_4\right] - \frac{\beta^4E}{M_1M_2}\left[\frac{1}{2}I_5B_1^2 + (2G -I_3)C_1^2 + \frac{1}{\beta^2N}I_5\right] + O_N(3).
\end{align*}
Rearrange gives
\begin{align*}
    \nu(S^2)&= \frac{\beta^2}{2M}\left(M_1K_4 - \beta^2EI_5\right)B_1^2 + \frac{\beta^2}{M}\left(M_1(E-K_3)-E\beta^2(2G-I_3)\right)C_1^2 + \frac{1}{MN}\left(M_1K_4 - \beta^2EI_5\right) + O_N(3).
\end{align*}
For the covariance $\nu(TS)$,
\begin{align*}
    (1 - \alpha_0\beta_0)\nu(TS)& = (1 - \alpha_0\beta_0)\nu(g(1, 1)) = \alpha_1 + \alpha_0 \beta_2,\\
    &= \frac{\beta^2}{M_1}\left[\frac{1}{2}I_5B_1^2 + (2G -I_3)C_1^2 + \frac{1}{\beta^2N}I_5\right] + \frac{\beta^4E}{M_1M_2}\left[\frac{1}{2}K_4  B_1^2 + (E - K_3)C_1^2 + \frac{1}{\beta^2 N} K_4\right].
\end{align*}
Rearrange gives 
\begin{align*}
   \nu(TS) = \frac{\beta^2}{M}\left(\frac{1}{2}M_2I_5 + \beta^2EK_4\right)B_1^2 + \frac{\beta^2}{M}\left(M_2(2G -I_3) +\beta^2E(E - K_3)\right)C_1^2 + \frac{1}{MN} \left(M_2I_5 + \beta^2EK_4\right) + O_N(3).
\end{align*}

\YS{We will now check $\nu(T^2)$, $\nu(S^2)$ are non-negative. First, 
\begin{align*}
    \frac{M_1}{\beta^2}(1 - \alpha_0\beta_0)\nu(T^2) &= \left[I_5C_1^2 + 2G 2A_1^2 + I_3\frac{M_2}{\beta^2MN}\right]\\
    &+ \left[\frac{K_4}{M_2} C_1^2 + \frac{E}{M_2}2A_1^2 + K_3\frac{1}{\beta^2MN}\right] + O_N(3)
\end{align*}
First, note that $\begin{bmatrix}
    K_4 & I_5 \\
    K_3 & I_3
\end{bmatrix}$ is a covariance matrix, thus $|I_5|= |K_3| \leq \sqrt{K_4 I_3} \leq \frac{K_4 + I_3}{2}$. }
\end{proof}
Now we turn to the proof of the general moments $T^hS^{h'}$.
\begin{proof}[Proof of Theorem \ref{thm: ST moments}]
By Lemma \ref{lem: recurse condition}, we have the following recursive relation for moments of $S, T$.
\begin{align}
    \nu(T^hS^{h'}) &=(h - 1) A_0^2 \nu(T^{h - 2}S^{h'}) + h' B_0^2 \nu(T^{h - 1}S^{h' - 1}) + O_N(h + h' + 1),\\
    &= h A_0^2 \nu(T^{h - 1}S^{h' - 1}) + (h' - 1)B_0^2 \nu(T^{h}S^{h' - 2}) + O_N(h + h' + 1).
\end{align}
The proof then proceeds with induction on $h' + h$. If $h + h' = 1$, the expression holds as odd moments of Gaussian is $0$. For $h + h' \geq 2$, applying the inductive hypothesis on two terms on the right-hand side gives 
\begin{align}
    \nu(T^hS^{h'}) &= (h - 1) A_0^2 \E[g_T^{h - 2}g_S^{h'}] + h' B_0^2 \E[g_T^{h - 1}g_S^{h' - 1}] + O_N(h + h' + 1),\\
    &= h A_0^2 \E[g_T^{h - 1}g_S^{h' - 1}] + (h' - 1)B_0^2 \E[g_T^{h}g_S^{h' - 2}] + O_N(h + h' + 1).
\end{align}
\YS{$\Sigma_0 \succeq 0$ follows from a similar arguement as in the proof of Lemma \ref{thm:peel off T_kS_k}.}
\YS{To check $\Sigma_0 = \begin{bmatrix}A_0^2 & C_0^2\\
C_0^2 & B_0^2\end{bmatrix} \succeq 0$: Suppose not, let $\Delta$ be the matrix of corresponding error terms so that $\Sigma_0 + \Delta = \begin{bmatrix}\nu(S^2) & \nu(ST) \\ \nu(ST) & \nu(T^2)\end{bmatrix}$. Note that $\Sigma_0 = \frac{1}{N} \Sigma'_0$ where $\Sigma'_0$ is independent from $N$. Let $\lambda$ be the negative eigenvalue of $\Sigma_0$ and $x$ be the corresponding eigenvector, then we have $x^T \Sigma_0 x = \frac{1}{N} \lambda < 0$. Then
\[x \Sigma_0 x^T + x\Delta x^T \leq \frac{1}{N}\lambda + 2\|\Delta\|_{\infty} < 0\]
where the last inequality follows from $\|\Delta\|_{\infty} = O_N(3) \ll \frac{1}{N}$. This is a contradiction since $\Sigma_0 + \Delta \succeq 0$.
}

Using Gaussian integration by parts \eqref{eq: gaussian intuition} to rewrite RHS completes the proof.
\end{proof}
\subsubsection{Proof of Lemma \ref{lem: relation ST}} \label{sec: reduce ST}
In this section, we derive Lemma \ref{lem: relation ST} using the cavity method.
Recall the definition of $U_v, \epsilon(v), U^-_v$ from the beginning of this section and that we denote $V_v = \{v_1, v_2, \cdots\}$ as the set of replicas appears in term $U_v$. Here $|V_v| = 2$ if $U_v$ corresponds to $T$ and $|V_v| = 1$ if $U_v$ corresponds to $S$. 

\paragraph{To reduce the moment of $T$} \label{sec: reduce T}
We start by proving \eqref{eq: general reduce T}. 
As usual, the first term in \eqref{eq: general rewrite2} is approximated using \eqref{eq:2nd approx}. We record the result in Lemma \ref{derivative T} The proof is technical but straight forward, thus pushed to the appendix (Lemma \ref{ap: derivative T}).
\begin{lemma}[First order derivative structure for $T$] \label{derivative T}
If $|V_1| = 2$, then 
\begin{align*}
    \nu(g(h, h'))
    =& \beta^2 (F-3G)\nu(g(h, h') + \beta^2E\nu(g(h - 1, h' + 1))\\
    &+ \beta^2(h - 1)\left[I_5C_1^2 + 2(2G -I_3)A_1^2 + I_3A_2^2\right]\nu(g(h -2, h'))\\
    &+ \beta^2h'\left[\frac{1}{2}I_5B_1^2 + (2G -I_3)C_1^2\right]\nu( g(h - 1, h' - 1))\\
    &+ O_N(h + h' + 1).
\end{align*}
\end{lemma}
\paragraph{For the second term in \eqref{eq: general rewrite2}}
\begin{align} \label{eq: T decomp 2nd}
    \frac{1}{N}\sum_{v = 2}^{h' + h} \nu(\epsilon(1)\epsilon(v)\Pi_{u \neq v} U^-_u) &= \frac{1}{N}\sum_{v = 2}^{h' + h} \nu((\epsilon_{1_1, 1_2}) - Q_{1_1, 1_2})\epsilon((\epsilon_{v_1, v_2}) - Q_{v_1, v_2}))\Pi_{u \neq v} U^-_u),\\
    &\stackrel{(\ref{lemma:last spin})} {=} \frac{h - 1}{N} I_3 \nu(T^{h -2}S^{h'}) + \frac{h'}{N} I_5 \nu(T^{h -1}S^{h' - 1}) + O_N(h' + h + 1),\\
    &= \frac{h - 1}{N} I_3 \nu(g(h - 2, h')) + \frac{h'}{N} I_5 \nu(g(h - 1, h' - 1)) + O_N(h' + h + 1).
\end{align}
Combine Lemma \ref{derivative T} and \eqref{eq: T decomp 2nd} gives \eqref{eq: general reduce T}\YS{Remove this equation?}
\begin{align*}
    \left(1 - \beta^2 (F-3G)\right)\nu(g(h, h')) =& \beta^2E\nu(g(h - 1, h' + 1))\\
    &+ \beta^2(h - 1)\left[I_5C_1^2 + 2(2G -I_3)A_1^2 + I_3A_2^2 + \frac{1}{\beta^2N}I_3\right]\nu(g(h -2, h'))\\
    &+ \beta^2h'\left[\frac{1}{2}I_5B_1^2 + (2G -I_3)C_1^2 + \frac{1}{\beta^2N}I_5\right]\nu( g(h - 1, h' - 1))\\
    &+ O_N(h + h' + 1).
\end{align*}
\paragraph{To reduce the moment of $S$} \label{sec: reduce S}  Similarily, we approximate the first term in \eqref{eq: general rewrite2} using \eqref{eq:2nd approx} to get Lemma \ref{derivative S} The proof can be found in appendix (Lemma \ref{ap: derivative S}).
\begin{lemma}[First order derivative structure for $S$] \label{derivative S}
    Suppose $|V_1| = 1$,
    \begin{align*}
    \nu(g(h, h'))
    =&
    \frac{\beta^2}{2}D\nu(g(h, h'))\\
    &+ 
    \beta^2 h\left[K_4 C_1^2 + 2(E - K_3)A_1^2 + K_3A_2^2\right]\nu( g(h - 1, h' - 1))\\
    &+ \beta^2(h' - 1)\left[\frac{1}{2}K_4  B_1^2 + (E - K_3)C_1^2\right]\nu( g(h, h' - 2))\\
    &-\beta^2E\nu(g(h + 1, h' - 1))\\
    &+ O_N(h + h' + 1).\\
\end{align*}
\end{lemma}
\paragraph{For the second term in \eqref{eq: general rewrite2}}
\begin{align} \label{eq: S decomp 2nd}
    \frac{1}{N}\sum_{v = 2}^{h' + h} \nu(\epsilon(1)\epsilon(v)\Pi_{u \neq v} U^-_u) &= \frac{1}{N}\sum_{v = 2}^{h' + h} \nu((\epsilon_{1_1, 1_2}) - Q_{1_1, 1_2})\epsilon((\epsilon_{v_1, v_2}) - Q_{v_1, v_2}))\Pi_{u \neq v} U^-_u),\\
    &\stackrel{(\ref{lemma:last spin})} {=} \frac{h}{N} K_3 \nu(g(h - 1, h'-1)) + \frac{h' - 1}{N} K_4 \nu(g(h, h' - 2)) + O_N(h' + h + 1).
\end{align}
Combining results from Lemma \ref{derivative S} and \eqref{eq: S decomp 2nd} gives the desired result. \YS{Again, remove this equation?}
\begin{align*}
    \left(1 - \frac{\beta^2}{2}D\right)\nu(g(h, h')) =& \beta^2 h\left[K_4 C_1^2 + 2(E - K_3)A_1^2 + K_3A_2^2 + \frac{1}{\beta^2 N} K_3\right]\nu( g(h - 1, h' - 1))\\
    &+ \beta^2(h' - 1)\left[\frac{1}{2}K_4  B_1^2 + (E - K_3)C_1^2 + \frac{1}{\beta^2 N} K_4\right]\nu( g(h, h' - 2))\\
    &-\beta^2E\nu(g(h + 1, h' - 1))\\
    &+ O_N(h + h' + 1).
\end{align*}
\subsection{Proof of Lemma~\ref{lem: general formal}}\label{ssec:pf-main-lem}
In this section, we put all the pieces together to compute the general mixed moments of $T_{k, l}, T_K, S_k, S, T$. 
\begin{proof}
For $k, l \in [n]$ $k \neq l$, let $\{g_{T_{k, l}}\}$ be the family of independent centered Gaussian random varaible with $\E[g^2_{T_{k, l}}] = A^2_2$ as in Theorem \ref{thm: general T12}. For $k \in [n]$, let $\{(g_{T_k}, g_{S_k})\}$ be the family of independent centered Gaussian with covariance matrix $\Sigma_1$ as in Theorem \ref{thm:peel off T_kS_k} and independent from $\{g_{T_{k, l}}\}$. Similarily, let $(g_S, g_T)$ be the Gaussian random vector with covariance matrix $\Sigma_0$ as in Theorem \ref{thm: ST moments} and independent from $\{g_{T_{k, l}}\}$ and $\{(g_{T_k}, g_{S_k})\}$. Apply Theorem \ref{thm: general T12}, then \ref{thm:peel off T_kS_k}, and finally \ref{thm: ST moments} gives the desired result.
\end{proof}

\bibliographystyle{alpha}
\bibliography{main}

\section{Appendix}\label{sec:app}
In this section, we prove the technical lemmas (Lemma \ref{derivative T_kl}, \ref{lem: relation T1S1}, \ref{lem: relation ST}) that characterize the recursive relations of general moments using cavity method. Recall the decomposition of general moments given in \eqref{eq: general rewrite2} 
\begin{align}
    \nu\left(\Pi_{k, l}T_{k, l}^{h(k, l)}\Pi_{k}T_k^{h(k)}T^{h}\Pi_{l}S_l^{h'(l)}S^{h'}\right) &= \nu(\Pi_{v \geq 1}U_v)\\
    & = \nu(\epsilon(1)\Pi_{v > 1}U^-_v) + \frac{1}{N}\sum_{u \geq 2} \nu(\epsilon(1)\Pi_{v \neq 1, u}U^-_v) + O_N(H + 1).
\end{align}
Note that the first term is of order $H - 1$, we should apply the second order approximation, \eqref{eq:2nd approx} and compute its first order derivative at time $0$. With some abuse of notation, we will always assume the first term $U_1$ corresponds to the type of basis, $T_1, S_1, T, S$, that we wish to "peel off" from the expression. Note that regardless of the type of $U_1$, $\nu_0(\epsilon(1)) = 0$ by symmetry. 
\begin{align} \label{eq: deriv init}
    \nu(\epsilon(1)\Pi_{v > 1}U^-_v) &= \nu'_0(\epsilon(1)\Pi_{v > 1}U^-_v) + O_N(H + 1).
\end{align}
This section is dedicated to characterizing the structure of such terms. 
\subsection{Proof of Lemma \ref{derivative T_kl}}
\begin{lemma}
Suppose $h(1, 2) \geq 1$ and $U_1$ corresponds to a copy of $T_{1, 2}$
\[\nu'_0(\epsilon(1)\Pi_{v > 1}U^-_v) = \beta^2 A\nu(g_{1, 2}(h(1, 2))) + O_N(H + 1).\]
\end{lemma}
\begin{proof}[Proof of lemma \ref{derivative T_kl}]
    Let $U_1$ be of type $T_{1, 2}$, by Claim \ref{claim: basis to overlap},
    \[\epsilon(1) = \epsilon_{1_1, 1_2} - \epsilon_{1_1, 1_4} - \epsilon_{1_3, 1_2} + \epsilon_{1_3, 1_4},\]
    By \eqref{derivative-2}, denote $m$ as the total number of replicas used by $\epsilon(1)\Pi_{v > 1}U^-_v$.
    \begin{align*}
        \nu'_0(\epsilon(1)\Pi_{v > 1}U^-_v) &= \frac{\beta^2}{2}\sum_{1 \leq a, b \leq 2m} \text{sgn}(a, b) \nu_0(\epsilon(1)\epsilon_{a, b})\nu_0((R^-_{a, b} - \mu_{a, b})\Pi_{v > 1}U^-_v) - \mathcal{R}_{m, \epsilon(1)\Pi_{v > 1}U^-_v}.
    \end{align*}
    As we noted in \eqref{eq: constant T12}, $\nu_0(\epsilon(1)\epsilon_{a, b}) = 0$ unless $\epsilon_{a, b}$ is a monomial in $\epsilon(1)$. Since $1_3, 1_4$ can not appear in any other terms $\{U_v: v > 1\}$, $\nu_0(\epsilon(1)\epsilon_{a, b}) = A > 0$ only when $a, b \subset V_1$ and $a \neq b$. Summing over all such pairs of replicas gives the desired result.
    \begin{align}
        \nu'_0(\epsilon(1)\Pi_{v > 1}U^-_v) &= A\beta^2 \nu(g_{1, 2}(h(1, 2))) + O_N(H + 1).
    \end{align}
\end{proof}
\subsection{Proof of Lemma \ref{derivative T_k}, \ref{derivative S_k}}
In this section, we derive the structure of \eqref{eq: deriv init} with 
\[\nu(\Pi_{v \geq 1}U_v) = \nu\left(\Pi_{k}T_k^{h(k)}T^{h}\Pi_{l}S_l^{h'(l)}S^{h'}\right).\]
Recall that the total moments of each type are
\[h_T = \sum_k h(k), \quad h_S = \sum_l h'(l), \quad H_1 = h_T + h_S + h + h'.\]
and $g_1$ is the function indexed by moments of $T_1, S_1$ s.t.
\[\nu(g_1(h(1), h'(1))) = \nu(\Pi_{1 \leq v \leq H_1}U_v).\]
We begin by introducing the following notations for referencing different terms in $\nu'_t(\cdot)$ given in \eqref{derivative-2}. Denote $m$ as as the number of total replicas used by $g(h(1), h'(1))$
\[m:= n + 2h_T + h_S\] 
For $a \in [2m]$, denote $a^{''}$ as the new replicas that first appear in \eqref{derivative-2}. For $a, b \in [2m]$, let $\text{sgn}(a, b) := -1^{|\{a, b\} \cap [m]|}$ .

Our goal is to compute the following derivative with $U_1$ corresponding to a copy of $T_1$ or $S_1$
\begin{align} \label{eq: deriv T1S1}
    \nu_0'(\epsilon(1)\Pi_{v > 1}U^-_v) &= \frac{\beta^2}{2}\sum_{1 \leq a, b \leq 2m} \text{sgn}(a, b) \nu_0(\epsilon(1)\epsilon_{a, b})\nu_0((R^-_{a, b} - Q_{a, b})\Pi_{v > 1}U^-_v) - \mathcal{R}_{m, \epsilon(1)\Pi_{v > 1}U^-_v}.
\end{align}
In both cases, we need to consider contributions from terms 
\[\text{sgn}(a, b) \nu_0(\epsilon(1)\epsilon_{a, b})\nu_0((R^-_{a, b} - \mu_{a, b})\Pi_{v > 1}U^-_v) - \mathcal{R}_{m, \epsilon(1)\Pi_{v > 1}U^-_v}.\]
Before proceeding, let's exploit the symmetry of $T_1$ and $S_1$ to rule out certain types of replica pair $(a, b)$. 
\begin{lemma} \label{lem: symmetry constant T1S1}
Suppose $\epsilon(v) = \epsilon_{v_1, v_3} - \epsilon_{v_2, v_3}$ or $\epsilon(v) = \epsilon_{v_1, v_1} - \epsilon_{v_2, v_2}$. If $|\{a, b\} \cap \{v_1, v_2\}| \neq 1$, then \[
\nu_0(\epsilon(v)\epsilon_{a, b}) = 0.
\]
Moreover, for any replica $k \in [2m]\backslash \{v_1, v_2\}$, and $(a, b) \in \{(v_1, k), (v_2, k)\}$, we have
\begin{align} \label{eq: symmetry constant T1S1}
    \nu_0(\epsilon(v) \epsilon_{v_1, k}) = - \nu_0(\epsilon(v) \epsilon_{v_2, k}).
\end{align}
\end{lemma}
\begin{proof}
The value of $\nu(\epsilon_{a, b}\epsilon_{c, d})$ depends only on the size of union and intersection of $\{a, b\}$ and $\{c, d\}$. 
Check that if $|\{a, b\} \cap \{v_1, v_2\}| \neq 1$, the two terms in $\epsilon(v)\epsilon_{a, b}$ are equvilent:

If $|\{a, b\} \cap \{v_1, v_2\}| = 0$, we have 
\[\nu_0(\epsilon_{v_1, v_3}\epsilon_{a, b}) = \nu_0(\epsilon_{1, 2}\epsilon_{3, 4}) = \nu_0(\epsilon_{v_2, v_3}\epsilon_{a, b}),\]
and
\[\nu_0(\epsilon_{v_1, v_1}\epsilon_{a, b}) = \nu_0(\epsilon_{1, 1}\epsilon_{2, 3}) = \nu_0(\epsilon_{v_2, v_2}\epsilon_{a, b}).\]
If $|\{a, b\} \cap \{v_1, v_2\}| = 2$, we have
\[\nu_0(\epsilon_{v_1, v_3}\epsilon_{v_1, v_2}) = \nu_0(\epsilon_{1, 2}\epsilon_{1, 3}) = \nu_0(\epsilon_{v_2, v_3}\epsilon_{v_1, v_2}),\]
and
\[\nu_0(\epsilon_{v_1, v_1}\epsilon_{v_1, v_2}) = \nu_0(\epsilon_{1, 1}\epsilon_{1, 2}) = \nu_0(\epsilon_{v_2, v_2}\epsilon_{v_1, v_2}).\]

Suppose $(a, b) \in \{\{v_1, k\}, \{v_2, k\}\}$. To check \eqref{eq: symmetry constant T1S1}:

If $U_v$ corresponds to $T_1$,
\begin{align*}
    \nu_0(\left(\epsilon_{v_1, v_3} - \epsilon_{v_2, v_3}\right)\epsilon_{v_1, k}) &= \nu_0(\epsilon_{v_1, v_1}\epsilon_{v_3, k} - \epsilon_{v_1, v_2}\epsilon_{v_3, k}),\\
    &= \nu_0(\epsilon_{v_2, v_2}\epsilon_{v_3, k} - \epsilon_{v_1, v_2}\epsilon_{v_3, k}),\\
    &=  -\nu_0(\left(\epsilon_{v_1, v_3} - \epsilon_{v_2, v_3}\right)\epsilon_{v_2, k}).
\end{align*}
If $U_v$ corresponds to $S_1$, 
\begin{align*}
    \nu_0(\left(\epsilon_{v_1, v_1} - \epsilon_{v_2, v_2}\right)\epsilon_{v_1, k}) &= \nu_0(\epsilon_{v_1, v_1}\epsilon_{v_1, k} - \epsilon_{v_2, v_2}\epsilon_{v_1, k}),\\
    &= \nu_0(\epsilon_{v_2, v_2}\epsilon_{v_2, k} - \epsilon_{v_1, v_1}\epsilon_{v_2, k}),\\
    &=  -\nu_0(\left(\epsilon_{v_1, v_1} - \epsilon_{v_2, v_2}\right)\epsilon_{v_2, k}).
\end{align*}
where the second equality follows from $k \notin \{v_1, v_2\}$ and by the linearity of expectation, exchanging the index of replica doesn't affect the expectation under $\nu_0$.
\end{proof}
By Lemma \ref{lem: symmetry constant T1S1}, the set of replica pairs $(a, b)$ s.t. $\nu_0(\epsilon(1)(\epsilon_{a, b} - q)) \neq 0$ is given by 
\[\cc{P}_1 = \{(a, b): |\{a, b\} \cap \{1_1, 1_2\}| = 1; 1 \leq a, b \leq [2m]\}.\]
Summing up all non-trivial terms in \eqref{eq: deriv T1S1}, the goal can be simplified to 
\[\nu_0'(\epsilon(1)\Pi_{v > 1}U^-_v) = \frac{\beta^2}{2}\sum_{(a, b) \in \cc{P}_1} sgn(a, b) \nu_0(\epsilon(1)\epsilon_{a, b})\nu_0((R^-_{a, b} - \mu_{a, b})\Pi_{v > 1}U^-_v).\]
Now we proceed to study the case when $U_1$ corresponds to a copy of $T_1$.
\begin{lemma}
[restatment of Lemma \ref{derivative T_k}]
For $h(1) \geq 1$ and $h'(1) \geq 0$,  suppose $U_1$ corresponds to a copy of $T_1$ 
\begin{align*}
    \nu(\epsilon(1)\Pi_{v > 1}U_v) =& \beta^2(F-3G)\nu(g_1(h(1), h'(1)) \\
    &+\frac{\beta^2}{2}H \nu(g_1(h(1) - 1, h'(1) + 1)\\
    &+ \beta^2 (h(1) - 1)G A_2^2 \nu(g_1(h(1) - 2, h'(1)))\\
    & + O_N(H_1 + 1).
\end{align*}
\end{lemma}
\begin{proof}[Proof of Lemma \ref{derivative T_k}]
The proof follows the same idea as in Section \ref{sec:var}.

Let's assume that $h(1) \geq 1$, and $U_1$ corresponds to a copy of $T_{1}$ (By definition, this corresponds to $1_1 = 1$ but we will use the notation $1_1$ for the sake of consistency.). 
\[
\epsilon(1) = \epsilon_{1_1, 1_3} - \epsilon_{1_2, 1_3}.
\]
To compute $\nu_0'(\epsilon(1)\Pi_{v > 1}U^-_v)$, we will count the contribution of terms from $(a, b) \in \cc{P}_1$. By Lemma \ref{lem: symmetry constant T1S1}, it make sense to group $(a, b)$ based on $\{a, b\} \cap [2m] \backslash \{v_1, v_2\}$. For each of those subset, we will first apply \eqref{eq: symmetry constant T1S1} to compute the $\nu_0(\epsilon(1)\epsilon_{a, b})$ part, then characterize the structure of $\nu_0((R^-_{a, b} - \mu_{a, b})\Pi_{v > 1}U^-_v)$.
\begin{itemize}
    \item If $a = b$: In this case, we have $(a, b) \in \{(v_1, v_1), (v_2, v_2)\}$. By a similar argument as the proof of \eqref{eq: symmetry constant T1S1}, 
    \[\nu_0(\epsilon(1)\epsilon_{1_1, 1_1}) = - \nu_0(\epsilon(1)\epsilon_{1_2, 1_2}) = H\]
    The sum of the two terms are 
    \[H \nu_0((R^-_{1_1, 1_1} - R^-_{1_2, 1_2})\Pi_{v > 1}U^-_v) = H \nu(g_1(h(1) - 1, h'(1) + 1) + O_N(H_1 + 1).\]
    \item If $\{a, b\} \in \{\{1_1, 1_3\}, \{1_2, 1_3\}\}$. For those terms, 
    \[\nu_0(\epsilon(1)\epsilon_{1_1, 1_3}) =  F.\]
    Summing over the two terms gives 
    \[2F \nu_0((R^-_{1_1, 1_3} - R^-_{1_2, 1_3})\Pi_{v > 1}U^-_v) = 2F \nu(g_1(h(1), h'(1))) + O_N(H_1 + 1).\]
    \item To sum over remaining $(a, b) \in \cc{P}_1$ corresponds to iterating over all replicas $k \in [2m] \backslash V_1$ and sum over pairs $\{1_1, k\}, \{1_2, k\}$. It is easier if we account for $k \in [m] \backslash V_1$ and the corresponding new replica $k' \in [2m] \backslash [m]$ introduced by Lemma \ref{lem:derivative}.
    \begin{itemize}
        \item For $k \in [m] \backslash V_1$, let $k' := m + k$ be the corresponding new replica from Lemma \ref{lem:derivative}. Summing over all four terms gives 
        \begin{align*}
           &G \nu_0((R^-_{1_1, k} - R^-_{1_2, k} - R^-_{1_1, k'} + R^-_{1_2, k'})\Pi_{v > 1}U^-_v)\\
           &= G \nu((T_{1_1, k} - T_{1_2, k} - T_{1_1, k'} + T_{1_2, k'})\Pi_{v > 1}U_v) + O_N(H_1 + 1).
        \end{align*}
        We now will explore the structure of this term by viewing it as some general moment of $T_{1, 2}, T_1, S_1, T, S$. By Theorem \ref{thm: general T12}, only even moments of $T_{k, l}$ give a non-trivial contribution to the sum. By construction, the replica $k'$ and $1_2$ do not appear in any other term $U_v$. Thus the non-trivial portion of $(T_{1_1, k} - T_{1_2, k} - T_{1_1, k'} + T_{1_2, k'})\Pi_{v > 1}U_v$ can only come from 
        \[\nu(T_{1_1, k}\Pi_{v > 1}U_v)\]
        Now the goal becomes checking if $T_{1_1, k}$ occurs in $\Pi_{v > 1}U_v$. Becuase $g_1(h(1), h'(1))$ contains terms of the form $T_k, S_k, T, S$, the only terms where $\{1_1, k\}$ appear together are terms correspond to $T_1$ or $S_1$. We will show that only $T_1$ terms are non-trivial. Let's first rewrite $U_v$ using the "basis". For $U_v$ corresponds to a copy of $T_{1_1}$ 
        \[U_v = T_{v_1, v_3}- T_{v_2, v_3} + T_{v_1} - T_{v_2}.\] 
        Thus if $k = u_3$ for some $u > 1$, then 
        \[\nu(T_{1_1, k}\Pi_{v > 1}U_v) = \nu(T_{1_1, k}^2 \Pi_{v \neq 1, u}U_v) = A_2^2 \nu(g_1(h(1) - 2, h'(1))) + O_N(H_1 + 1).\]
        If $k$ appears in a copy of $S_{1_1}$, $U_u$, then the corresponding term becomes
        \[\nu(T_{1_1, k} (S_{1_1}- S_k) \Pi_{v \neq 1, u}U_v)) = O_N(H_1 + 1).\]
        \YS{If we want to check things in detail, let $p, q$ be two replicas that are never used before, we can expand out the expression of $T_{1_1, k}$
        \begin{align*}
            &\nu(T_{1_1, k} (R_{1_1, 1_1} - R_{k, k}) \Pi_{v \neq 1, u}U_u)) \\
            &= \nu((R_{1_1, k} - R_{1_1, q} - R_{k, p} + R_{k, q})(R_{1_1, 1_1} - R_{k, k}) \Pi_{v \neq 1, u}U_u))\\
            &= \nu(R_{1_1, 1_1}(R_{1_1, k} - R_{1_1, q} - R_{k, p} + R_{k, q}) \Pi_{v \neq 1, u}U_v))\\
            &-\nu(R_{k, k}(R_{1_1, k} - R_{1_1, q} - R_{k, p} + R_{k, q})\Pi_{v \neq 1, u}U_v))
        \end{align*}
        Note that since $k \neq 1_1$ and by definition, $k$ does not appear in any other remaining terms $U_v$, 
        \begin{itemize}
            \item For the first term, since $k, q$ do not appear in the remaining terms, the first two terms of $T_{1_1, k}$ expression are of the same type, thus 
            \[\nu((R_{1_1, 1_1}(R_{1_1, k} - R_{1_1, q})\Pi_{v \neq 1, u}U_v)) = 0\]
            Also,
            \[\nu((R_{1_1, 1_1}(- R_{k, p} + R_{k, q})\Pi_{v \neq 1, u}U_v)) = 0\]
            \item For the second term, observe that if we introduce a new replica $r$, the same argument above implies that 
            \[\nu(R_{r, r}(R_{1_1, k} - R_{1_1, q} - R_{k, p} + R_{k, q})\Pi_{v \neq 1, u}U_v)) = 0\]
            Thus we have 
            \begin{align*}
                &\nu(R_{k, k}(R_{1_1, k} - R_{1_1, q} - R_{k, p} + R_{k, q})\Pi_{v \neq 1, u}U_v))\\
                &= \nu((R_{k, k} - R_{r, r})(R_{1_1, k} - R_{1_1, q} - R_{k, p} + R_{k, q})\Pi_{v \neq 1, u}U_v))\\
                &= \nu(S_kT_{1_1, k} \Pi_{v \neq 1, u}U_v))\\
                &= O_N(H + 1)
            \end{align*}
            where the last equality follows from Theorem \ref{thm: general T12}.
        \end{itemize}
        }
        Thus we only need to count countribution of $\{1_1, k\}$ where $\{1_1, k\}$ appear together in some term $U_v$ where $V_v = \{v_1, v_2, v_3\}$ and by definition, $v_1 = 1_1$ $v_2 = k$. 
        \begin{align*}
            2\times (h(1) - 1)G A_2^2 \nu(g_1(h(1) - 2, h'(1))).
        \end{align*}
        \item The only $k$ that are not counted now are those that correspond to replicas in $V_1$. For each such $k$, since such $k$ as well as $1_2$ do not appear in any other terms in $\Pi_{v > 1} U^-_v$
        \[-G\nu_0((R^-_{1_1, k} - R^-_{1_2, k}) \Pi_{v > 1} U^-_v) = -G \nu(T_{1_1} \Pi_{v > 1} U^-_v) + O_N(H_1 + 1).\]
        Summing over terms from all $6$ pairs gives 
        \[-2 \times 3G \nu_0(g_1(h(1), h'(1))).\]
    \end{itemize}
\end{itemize}
Combine all the terms gives 
\begin{align*}
    \nu(\epsilon(1)\Pi_{v > 1} U_v) &= \frac{\beta^2}{2}H\nu(g_1(h(1) - 1, h'(1) + 1)\\
    &+ \beta^2F \nu(g_1(h(1), h'(1)))\\
    &+  \beta^2(h(1) - 1)G A_2^2 \nu(g_1(h(1) - 2, h'(1)))\\
    & -\beta^2 3G \nu_0(g_1(h(1), h'(1)))
\end{align*}
Rearrange gives the desired result. 
\end{proof}
\begin{lemma}
[restatement of Lemma~\ref{derivative S_k}]
If $h'(1) \geq 1$ and $h(1) \geq 0$, suppose $U_1$ corresponds to a copy of $S_1$, then
\begin{align*}
    \nu(\epsilon(1)\Pi_{v > 1}U_v) =& \frac{\beta^2}{2}D \nu(g_1(h(1), h'(1)))\\
    &+ \beta^2 h(1)E A_2^2\nu(g_1(h(1) - 1, h'(1) - 1))\\
    & -2\beta^2E \nu(g_1(h(1) + 1, h'(1) - 1)\\
    &+ O_N(H_1 + 1)).
\end{align*}
\end{lemma}
\begin{proof}[Proof of Lemma~\ref{derivative S_k}]
The proof is similar to that of Lemma \ref{derivative T_k}, but with
\[
\epsilon(1) = \epsilon_{1_1, 1_1} - \epsilon_{1_2, 1_2}.
\]
We included it here for completeness. Let's count the contribution from each pair $(a, b)$ in $\cc{P}_1$. 
\begin{itemize}
    \item $a = b \in V_1$: In this case, we have $(a, b) \in \{(v_1, v_1), (v_2, v_2)\}$. By a similar argument as the proof of \eqref{eq: symmetry constant T1S1}, 
    \[\nu_0(\epsilon(1)\epsilon_{1_1, 1_1}) = - \nu_0(\epsilon(1)\epsilon_{1_2, 1_2}) = D.\]
    Combining the two terms gives 
    \[D\nu_0((R^-_{1_1, 1_1} - R^-_{1_2, 1_2})\Pi_{v > 1}U^-_v) = D \nu(g_1(h(1), h'(1))) + O_N(H_1 + 1).\]
    \item $a \neq b$. 
    \begin{itemize}
        \item For each replica $k \in [m] \backslash V_1$, let $k' = m + k$ be the corresponding new replica introduced by the derivative formula. WLOG, first fix $a \in V_1$. Combine terms corresponds to $b \in \{k, k'\}$, we have
    \begin{align*}
        &E \nu_0((R^-_{1_1, k} - R^-_{1_2, k} - R^-_{1_1, k'} + R^-_{1_2, k'})\Pi_{v > 1}U^-_v) \\
        &= E \nu_0((T_{1_1, k} - T_{1_2, k} - T_{1_1, k'} + T_{1_2, k'})\Pi_{v > 1}U_v).
    \end{align*}
    Following the same argument as in the corresponding type of pair in Lemma \ref{derivative T_k}, the only non-trival contributions come from when $\{1_1, k\}$ appears in some $V_u$ where $U_u$ is a copy of $T_{1_1}$ and $1_1 = v_1$, $k = v_2$. The contribution from such terms are
    \[E \nu_0(T_{1_1, k}^2\Pi_{v \neq 1, u}U_v) = E A_2^2 \nu(g_1(h(1) - 1, h'(1) - 1)) + O_N(H_1 + 1)\]
    Summing up all contributions from such terms gives 
    \[2 \times h(1)E A_2^2\nu(g_1(h(1) - 1, h'(1) - 1)) + O_N(H_1 + 1).\]
    \item The only $k \in [2m] \backslash V_1$ that are not counted are the two new replicas corresponding to $V_1$. For each such replica, the contribution is
    \[- E \nu_0((R^-_{1_1, k} - R^-_{1_2, k})\Pi_{v > 1}U^-_v) = -E \nu(g_1(h(1) + 1, h'(1) - 1)) + O_N(H_1 + 1).\]
    The total contribution is 
    \[2 \times (-2E \nu(g_1(h(1) + 1, h'(1) - 1)) + O_N(H_1 + 1)).\]
    \end{itemize}
\end{itemize}
Summing over contribution from all pairs in $\cc{P}_1$,
\begin{align*}
    \nu(\epsilon(1)\Pi_{v > 1}U_v) =& \frac{\beta^2}{2}D \nu(g_1(h(1), h'(1)))\\
    &+ \beta^2 h(1)E A_2^2\nu(g_1(h(1) - 1, h'(1) - 1))\\
    & -2\beta^2E \nu(g_1(h(1) + 1, h'(1) - 1)\\
    &+ O_N(H_1 + 1)).
\end{align*}
\end{proof}
\subsection{Proof of Lemma \ref{derivative T}, \ref{derivative S}}
Recall that
\[T^hS^{h'} = g(h, h') = \Pi_{v \geq 1} U_v, \quad \text{and} \quad  \epsilon(v) = \epsilon_{v_1, v_2} - Q_{v_1, v_2}.\]
For $T^hS^{h'}$, we have the additional property that if $v \neq v'$,  then
\begin{align} \label{prop: empty intersection}
    V_v \cap V_{v'} = \emptyset.
\end{align}
As in the previous section, we first introduce some notation to characterize the formula in \eqref{derivative-2}. Denote the number of total replicas appear $g(h, h')$ as
\[
m:=2h + h'.
\] 
Denote $a^{''}$ as the new replicas in $\nu'_0(g(h, h'))$ for each $a \in [2m]$, and $\text{sgn}(a, b) := -1^{|\{a, b\} \cap [m]|}$.

Our goal is to compute the following derivative of $\nu(g(h, h'))$,
\begin{align}
    \nu_0'(\epsilon(1)\Pi_{v > 1}U^-_v) &= \frac{\beta^2}{2}\sum_{1 \leq a, b \leq 2m} \text{sgn}(a, b) \nu_0(\epsilon(1)\epsilon_{a, b})\nu_0((R^-_{a, b} - \mu_{a, b})\Pi_{v > 1}U^-_v) - R_{n, g},\\
    &= \frac{\beta^2}{2}\sum_{1 \leq a, b \leq 2m} \text{sgn}(a, b) \nu_0(\epsilon(1)\epsilon_{a, b})\nu_0((R^-_{a, b} - Q_{a, b})\Pi_{v > 1}U^-_v)\\
    &-\frac{\beta^2}{2}\sum_{1 \leq a, b \leq 2m} \text{sgn}(a, b)\nu_0(\epsilon(1)\epsilon_{a^{''}, b^{''}})\nu_0((R^-_{a^{''}, b^{''}} - Q_{a^{''}, b^{''}})\Pi_{v > 1}U^-_v).
\end{align} 
Unlike in the case of $T_1, S_1$, $\nu_0(\epsilon(1)\epsilon_{a, b}) \neq 0$ for all pairs of replica $(a, b)$. To simplify the computation, for each $(a, b) \in[2m] \times [2m]$, consider the corresponding term
\begin{align} \label{def:single term ST}
    D_{a,b} := \nu_0(\epsilon(1)\epsilon_{a, b})\nu_0((R^-_{a, b} - Q_{a, b})\Pi_{v > 1}U^-_v) - \nu_0(\epsilon(1)\epsilon_{a^{''}, b^{''}})\nu_0((R^-_{a^{''}, b^{''}} - Q_{a^{''}, b^{''}})\Pi_{v > 1}U^-_v).
\end{align}
Then we can rewrite \ref{derivative-2} as a sum of all such pairs $(a, b)$.
\[
\nu_0'(\epsilon(1)\Pi_{v > 1}U^-_v) = \frac{\beta^2}{2}\sum_{(a, b) \in [2m] \times [2m]} D_{a, b}.
\]
We first characterize paris of $(a, b)$ s.t. $D_{a, b} = 0$. Observe that if $\{a, b\} \cap [m] = \emptyset$, neither $(a, b)$ nor $a^{''}, b^{''}$ appear in any terms $U^-_v$. They also do not intersect with $V_1$, thus the two terms in $D_{a, b}$ are equivalent. In this case, we have 
\[
D_{a, b} = 0.
\]
The other observation is that $\nu(D(a, b))$ depends only on if $a = b$ and  how $a, b$ appear in the remaining terms $\Pi_{v > 1}U^-_v$. The two lemma below describe values $\nu(D(a, b))$ based on the two conditions above.

\begin{lemma} \label{lem: T terms intersections}
Suppose $a, b$ are indexes of two replicas s.t. $a \neq b$. Let $X_a, X, Y_a, Y$ be the random variables s.t.
When $a \in V_k$ for some $k \in [h + h']$,  $(X_a, X) = \begin{cases}
    (T_a, T), &\text{ if } |V_k| = 2,\\
    (S_a, T), &\text{ if } |V_k| = 1.\\
\end{cases}$ Similarly, when $b \in V_l$, $(Y_b, X) = \begin{cases}
    (T_b, T), &\text{ if } |V_l| = 2,\\
    (S_b, T), &\text{ if } |V_l| = 1.\\
\end{cases}$
Then we have the following trichotomy,
\begin{itemize}
    \item If $a \in \{k_1, k_2\}$ and $b \in \{l_1, l_2\}$ 
\[\nu_0((R_{a, b} - q)U_kU_l\Pi_{w \notin \{k, l\}} U_w) = \nu_0((T_aX_aY + T_bY_bX + TXY)\Pi_{w \notin \{k, l\}} U_w))\]
\item If only $a$ appears in $V_k$ for some term $U_k$, then 
\[\nu_0((R_{a, b} - q)U_k\Pi_{w \notin \{k\}} U_w) = \nu_0((T_aX_a + TX)\Pi_{w \notin \{k\}} U_w))\]
\item If neither $a, b$ are used by the rest of the formula, 
\[\nu_0((R_{a, b} - q)\Pi_{w > 1} U_w) = \nu_0(T\Pi_{w \notin \{k\}} U_w))\]
\end{itemize}
\end{lemma}
Similarly, we can characterize the relations for when $a = b$.
\begin{lemma}\label{lem: S term intersections}
Suppose $a$ is the index of a replica, let 
\[(X_a , X) \in \{(T_a, T), (S_a, S)\},\]
where for $a \in V_k$, $(X_a, X) = \begin{cases}
    (T_a, T), & \text{ if } |V_k| = 2,\\
    (S_a, T), & \text{ if } |V_k| = 1.\\
\end{cases}$
\begin{itemize}
    \item If $a \in V_k$ for some $1 < k \leq h + h'$, then
    \[\nu_0((R_{a, a} - p)U_k\Pi_{w \notin \{k\}} U_w) = \nu_0((S_aX_a + SX)\Pi_{w \notin \{k\}} U_w)).\]
    \item If $a$ does not appear in any $U_k$, then
    \[\nu_0((R_{a, a} - p)\Pi_{w > 1} U_w) = \nu_0(S\Pi_{w \notin \{k\}} U_w)).\]
\end{itemize}
\end{lemma}
We now give the proof of the two lemmas above.
\begin{proof}[Proof of Lemma~\ref{lem: T terms intersections}]
For the first case, where $a \in \{k_1, k_2\}$ and $b \in \{l_1, l_2\}$, WLOG, assume $a = k_1$, $b = l_1$. First rewrite $R_{a, b} - q, U_k, U_l$ using renormalized random variable in definition \ref{def:expand}
\[R_{a, b} - q = T_{a, b} + T_a + T_b + T.\]
For $ U_k, U_l$,
\[R_{p, q} - Q_{p, q} = \begin{cases}
    T_{p, q} + T_{p} + T_{q} + T, & \text{if } p \neq q,\\
    S_{p} + S, & \text{if } p = q.
\end{cases}\]

Since $a, b$ do not appear in the same $U_v$ and by \eqref{prop: empty intersection}, only odd moments $T_{p, q}$ appear in the formula. Again by $\eqref{prop: empty intersection}$, $k_2$, $l_2$ are not used by any other $U_w$ for $w \notin \{k, l\}$. Thus the only terms with even moments from $U_k, U_l$ are those that depend only on $a, b$ or $T, S$. 

Let
\[U_k = X_a + X + X', \quad \text{and} \quad  U_l = Y_b + Y + Y'.\] 
where $X' = \begin{cases}
    0, & \text{ if } k_1 = k_2,\\
    T_{a, k_2} + T_{k_2}, & \text{ if } k_1 \neq k_2.
\end{cases}$ 
and $Y'$ defined similarily for $U_l$. Thus 
\begin{align*}
    \nu_0((R_{a, b} - q)U_kU_l\Pi_{w \notin \{k, l\}} U_w) &= \nu_0(\left(T_a + T_b + T\right)\left(X_a + X + X'\right)\left(Y_b + Y + Y'\right)\Pi_{w \notin \{k, l\}} U_w)\\
    &= \nu_0(\left(T_aX_aY + T_bXY_b + TXY\right)\Pi_{w \notin \{k, l\}} U_w) + O_N(h + h' + 1).
\end{align*}
where the last equality follows from Theorem \ref{thm:peel off T_kS_k}.

If only $a \in S_k$ for some $1 < k \leq (h + h')$, the above equation becomes 
\begin{align*}
    \nu_0((R_{a, b} - q)U_k\Pi_{w \notin \{k, l\}} U_w) &= \nu_0(\left(T_a + T_b + T\right)\left(X_a + X + X'\right)\Pi_{w \notin \{k\}} U_w),\\
    &= \nu_0(\left(T_aX_a + TX\right)\Pi_{w \notin \{k, l\}} U_w) + O_N(h + h' + 1).
\end{align*}
If none of $a,b \in S_k$ for some $1 < k \leq (h + h')$, 
\begin{align*}
    \nu_0((R_{a, b} - q)\Pi_{w \notin \{k, l\}} U_w) &= \nu_0(T\Pi_{w \notin \{k\}} U_w).
\end{align*}
\end{proof}
\begin{proof}[Proof of Lemma~\ref{lem: S term intersections}]
The proof is similar to the proof of Lemma~\ref{lem: T terms intersections}, but with
\[R_{a, a} - p = S_a + S.\]
\end{proof}
\begin{lemma}[restatement of Lemma~\ref{derivative T}] \label{ap: derivative T}
    If $h \geq 1$ and $|V_1| = 2$, we have
\begin{align*}
    \nu(\epsilon(1)\Pi_{v > 1}U^-_v)
    =& \beta^2 (F-3G)\nu(g(h, h') + \beta^2E\nu(g(h - 1, h' + 1))\\
    &+ \beta^2(h - 1)\left[I_5C_1^2 + 2(2G -I_3)A_1^2 + I_3A_2^2\right]\nu(g(h -2, h'))\\
    &+ \beta^2h'\left[\frac{1}{2}I_5 B_1^2 + (2G -I_3)C_1^2\right]\nu( g(h - 1, h' - 1))\\
    &+ O_N(h + h' + 1).\\
\end{align*}
\end{lemma}
\begin{proof}[Proof of Lemma~\ref{derivative T}]
Let's first consider the case when $h' > 0$. By \eqref{eq:2nd approx}, 
\begin{align*}
    \nu(\epsilon(1)\Pi_{v > 1}U^-_v) = \nu_0'(\epsilon(1)\Pi_{v > 1}U^-_v) + O_N(h + h' + 1).
\end{align*}
Recall that 
\[\nu_0'(\epsilon(1)\Pi_{v > 1}U^-_v) = \frac{\beta^2}{2}\sum_{(a, b) \in {[2m] \choose 2}} D_{a, b}.\]
Let's first consider the terms when $a = b$, 
\begin{itemize}
    \item If $a=b \in V_1$, since $a \in V_1$, by \eqref{prop: empty intersection}, $a$ does not appear in $U^-_v$, then
    \begin{align*}
        \sum_{a \in V_1} D_{a, a} &\stackrel{\ref{lemma:last spin}}{=} 2(I_4 - I_5)\nu_0(S\Pi_{v > 1}U_v) + O_N(h + h' + 1)\\
        &= 2(I_4 - I_5)\nu_0(g(h - 1, h' + 1)) + O_N(h + h' + 1).
    \end{align*}
   \item For $a = b \notin V_1$, $a, b \leq M$, by Lemma \ref{lem: S term intersections}, suppose $a \in V_k$
   \begin{align*}
       \sum_{a \in V_k} D_{a, a} = 2I_5 \left(\nu_0((S_aX_a+ SX)\Pi_{v \neq 1, k}U^-_v) - \nu_0(S\Pi_{v > 1}U^-_v)\right),
   \end{align*}
   where $(X_a, X)$ is as defined in lemma \ref{lem: S term intersections}.
   There are $h - 1$ $T$ terms and $h'$ $S$ terms, thus
   \begin{align*}
       \sum_{1 < a \leq M} D_{a, a} =&  2(h - 1) I_5\left(\nu_0(T_aS_a + ST)\Pi_{v \neq 1, k}U^-_v) - \nu_0(S\Pi_{v > 1}U^-_v)\right)\\
       &+ h'I_5\left(\nu_0(S_a^2 + S^2)\Pi_{v \neq 1, k}U^-_v) - \nu_0(S\Pi_{v > 1}U^-_v)\right),\\
       \stackrel{\ref{lemma:last spin}}{=}&I_52(h - 1)\nu_0(T_aS_ag(h -2, h')) + I_5 h' \nu_0(S_a^2 g(h - 1, h' - 1)) + O_N(h + h' + 1).
   \end{align*}
\end{itemize}
Let's now consider the case when $a \neq b$,
\begin{itemize}
    \item If $\{a,b\} = V_1$, by \eqref{prop: empty intersection},
    \[\sum_{\{a,b\} = V_1}D_{a, b} \stackrel{\ref{lemma:last spin}}{=} 2(I_1 - I_3)\nu_0((R_{1_1, 1_2} - q)\Pi_{v \neq 1, k}U_v) = 2(I_1 - I_3) \nu_0(g(h, h')).\]
    \item If $|\{a, b\} \cap V_1| = 1$. WLOG, assume $a \in V_1$. By Lemma \ref{lem: T terms intersections}, if $b \in V_l$, 
    \[D_{a, b} = I_2 \nu_0((T_bY_b + TY)\Pi_{v \neq 1, l}U^-_v) - I_3\nu_0(T\Pi_{v > 1}U^-_v).\]
    Since there are $h - 1$ many terms $T$ and $h'$ many terms $S$, fix $a \in V_1$,
    \begin{align*}
        \sum_{b \in [m] \backslash V_1} D_{a, b} =& 2(h - 1)I_2 \nu_0(T_b^2\Pi_{v \neq 1, l}U^-_v) + 2(h - 1)(I_2 - I_3)\nu_0(T\Pi_{v > 1}U^-_v)\\
        &+ h' I_2 \nu_0(T_bS_b\Pi_{v \neq 1, l}U^-_v) + h'(I_2 - I_3)\nu_0(T\Pi_{v > 1}U^-_v),\\
        \stackrel{\ref{lemma:last spin}}{=}& 2(h - 1)I_2 \nu_0(T_b^2 g(h - 2, h')) + h' I_2 \nu_0(T_bS_b g(h - 1, h' - 1))\\
        &+ (2h - 2 + h')(I_2 - I_3)\nu_0(g(h, h')) + O_N(h + h' + 1).
    \end{align*}
    Summing over all such subsets, by symmetry of $a, b$ 
    \begin{align*}
        \sum_{|\{a, b\} \cap V_1| = 1} D_{a, b} =& 4 \left[2(h - 1)I_2 \nu_0(T_b^2 g(h - 2, h')) + h' I_2 \nu_0(T_bS_b g(h - 1, h' - 1))\right]\\
        &+ 4 (2h - 2 + h')\nu_0(g(h, h')) + O_N(h + h' + 1).
    \end{align*}
    \item For $|\{a, b\} \cap V_1| = 0$, $\{a, b\} \subset [m]$. In this case, $a \in V_k$ and $b \in V_l$ for $1 < k, l \leq h + h'$. For $k \neq l$, by Lemma~\ref{lem: T terms intersections} and Lemma~\ref{lemma:last spin}
    \begin{align*}
        D_{a, b} =& I_3 \nu_0(\left(T_aX_aY + T_bXY_b + TXY\right)\Pi_{v \neq 1, k, l}U_v) - I_3 \nu_0(T\Pi_{v > 1}U^-_v), \\
        =& I_3 \nu_0(\left(T_aX_aY + T_bXY_b\right)\Pi_{v \neq 1, k, l}U_v).
    \end{align*}
    For $k = l$, 
    \[D_{a, b} = I_3 \nu_0(\left(T_{a, b}^2 + T_a^2 + T_b^2\right)\Pi_{v \neq 1, k}U_v).\]
    Again, since there are $h - 1$ terms $T$ and $h'$ terms $S$,
    \begin{align*}
        \sum_{|\{a, b\} \cap V_1| = 0} D_{a, b} =& 4 (h - 1)(h - 2) I_3\nu_0((T_a^2 + T_b^2)g(h - 2, h'))\\
        &+ 2 \times 2(h - 1)h'I_3\nu_0((T_a^2S + T_bS_bT)g(h - 2, h' - 1))\\
        &+ 2h'(h' - 1) I_3\nu_0((T_aS_a + T_bS_b)g(h - 1, h' - 1))\\
        &+ 2 (h - 1) I_3\nu_0(\left(T_{a, b}^2 + T_a^2 + T_b^2\right)g(h - 2, h')).
    \end{align*}
    \item If $|\{a, b\} \cap V_1| = 1$ and $|\{a, b\} \cap [m]| = 1$. WLOG, assume $a \in V_1$,
    \[D_{a, b} \stackrel{\ref{lemma:last spin}}{=} (I_2 - I_3) \nu_0(g(h, h')).\]
    Summing over all such terms 
    \[2 \times 2m (I_2 - I_3) \nu_0(g(h, h')).\]
    \item If $|\{a, b\} \cap V_1| = 0$ and $|\{a, b\} \cap [m]| = 1$. WLOG, assume $a \in V_k$, by Lemma~\ref{lem: T terms intersections}
    \[D_{a, b} \stackrel{\ref{lemma:last spin}}{=} I_3 \nu_0((T_aX_a + TX) \Pi_{v \neq 1, k})U_v) - I_3 \nu_0(T\Pi_{v \neq 1})U_v).\]
    Summing over $h - 1$ terms $S$ and $h'$ terms $S$, 
    \begin{align*}
        2I_3 \left[2(h - 1)m \nu_0(T_a^2 g(h - 2, h')) + h'M\nu_0(T_aS_a g(h - 1, h' - 1)).\right]
    \end{align*}
\end{itemize}
After simplification, we have
\begin{align*}
    \frac{2}{\beta^2}\nu_0'(\epsilon(1)\Pi_{v > 1}U^-_v)
    =& \left[2(I_1 - 2I_2) -2(2I_2 - 3I_3)\right]\nu_0(g(h, h'))\\
    &+ \left[2(I_4 - I_5)\right]\nu_0(g(h - 1, h' + 1))\\
    &+ \left[I_52(h - 1)\right]\nu_0(T_1S_1g(h -2, h'))\\
    &+ \left[I_5 h' \right]\nu_0(S_1^2 g(h - 1, h' - 1))\\
    &+ 4(h - 1)\left[2I_2 -3I_3 \right]\nu_0(T_1^2 g(h - 2, h'))\\
    &+ 2h'\left[ 2I_2 -3I_3\right]\nu_0(T_1S_1 g(h - 1, h' - 1))\\
    &+ 2 (h - 1) I_3\nu_0(T_{a, b}^2g(h - 2, h')).\\
\end{align*}
Observe that for all terms on the RHS is a function of order $h + h'$, thus by \eqref{eq:1st approx} and Theorem \ref{thm: general T12}, \ref{thm:peel off T_kS_k} 
\begin{align*}
    \nu_0'(\epsilon(1)\Pi_{v > 1}U^-_v)=& \beta^2\left[(I_1 - I_3) -4(I_2 - I_3)\right]\nu(g(h, h'))\\
    &+ \beta^2(I_4 - I_5)\nu(g(h - 1, h' + 1))\\
    &+ \beta^2\left[I_5(h - 1)C_1^2 + 2(h - 1)(2I_2 -3I_3)A_1^2 + (h - 1) I_3A_2^2\right]\nu(g(h -2, h'))\\
    &+ \beta^2\left[\frac{1}{2}I_5 h'B_1^2 + (2I_2 -3I_3)h'C_1^2\right]\nu( g(h - 1, h' - 1))\\
    &+ O_N(h + h' + 1),\\
    =& \beta^2 (F-3G)\nu(g(h, h') + \beta^2E\nu(g(h - 1, h' + 1))\\
    &+ \beta^2(h - 1)\left[I_5C_1^2 + 2(2G -I_3)A_1^2 + I_3A_2^2\right]\nu(g(h -2, h'))\\
    &+ \beta^2h'\left[\frac{1}{2}I_5 B_1^2 + (2G -I_3)C_1^2\right]\nu(g(h - 1, h' - 1))\\
    &+ O_N(h + h' + 1).
\end{align*}
If $h' = 0$ or $h = 1$: Note that the above computation is a summation over terms $\{D_{a, b}: a, b \in [2m]\}$, for $D(a, b)$ is defined in \eqref{def:single term ST}. By Lemma \ref{lem: T terms intersections} and \ref{lem: S term intersections}, $D(a, b)$ depends on $|\{a, b\}|$ and $|V_k|$ (or $V_l$) if $a$ (or $b$) appears in some term $U_k$. We first partitioned $(a, b)$ into subsets based on the value of $D(a, b)$ and then counted the size of each subsets. 

It's easy to see that the size of the subsets depends only on $|\{k: ||V_k|  = 1; k > 1\}|$ and $|\{k: ||V_k|  = 2; k > 1\}|$. Note that $|V_k|$ represents if $U_k$ corresponds to a copy of $T$ or $S$, i.e. if $|V_k| = 1$, $U_k$ corresponds to a copy of $S$. In this case, $h' = |\{k: ||V_k|  = 1; k > 1\}|$ and $h - 1 = |\{k: ||V_k|  = 2; k > 1\}|$.

If $h' = 0$, this corresponds to the case $g(h, 0) = T^h$. We do not need to count $S$ terms in the above computation and can simply plug in $h' = 0$ in the computation above. One can also check the formula by summing over different types of pair $(a, b)$:
\begin{align*}
\nu_0'(\epsilon(1)\Pi_{v > 1}U^-_v)
    =& \beta^2\left[(I_1 - I_3) -8(I_2 - I_3)\right]\nu_0(g(h, h'))\\
    &+ \beta^2\left[(I_4 - I_5)\right]\nu_0(g(h - 1, h' + 1))\\
    &+ \beta^2(h - 1)\left[I_5C_1^2 + 2(2I_2 - 3I_3)A_1^2 + I_3A_2^2)\right]\nu_0(g(h -2, h'))\\
    &+ O_N(H + 1),\\
    =& \beta^2 (F-3G)\nu(g(h, h') + \beta^2E\nu(g(h - 1, h' + 1))\\
    &+ \beta^2(h - 1)\left[I_5C_1^2 + 2(2G -I_3)A_1^2 + I_3A_2^2\right]\nu(g(h -2, h'))\\
    &+ O_N(h + h' + 1).
\end{align*}
Similarly, if $h = 1$, all of the remaining terms $U_k$ correspond to an $T$ term. Recall the definition of $g(h, h')$ and plug in $h - 1 = 0$ give the desired result.
\end{proof}
Next we will derive $\nu'_0(\epsilon(1)\Pi_{v > 1}U^-_v)$ when reducing the moment of $S$.
\begin{lemma}[Restatment of Lemma \ref{derivative S}]  \label{ap: derivative S}
If $h' \geq 1$ and $|V_1| = 1$, we have 
\begin{align*}
     \nu(\epsilon(1)\Pi_{v > 1}U^-_v)
    =& \frac{\beta^2}{2}D\nu(g(h, h'))\\
    &+ \beta^2 h\left[K_4 C_1^2 + 2(E - K_3)A_1^2 + K_3A_2^2\right]\nu( g(h - 1, h' - 1))\\
    &+ \beta^2(h' - 1)\left[\frac{1}{2}K_4  B_1^2 + (E - K_3)C_1^2\right]\nu_0( g(h, h' - 2))\\
    &-\beta^2E\nu_0(g(h + 1, h' - 1))\\
    &+ O_N(h + h' + 1).
\end{align*}
\end{lemma}
\begin{proof}[Proof of Lemma~\ref{derivative S}]
As in the proof of Lemma~\ref{derivative T}, we will first consider the case $h > 0$. Let's begin by rewriting $\nu(\epsilon(1)\Pi_{v > 1}U^-_v)$ using \eqref{eq:2nd approx}. 
\begin{align} \label{eq: deri S 1st step}
    \nu(\epsilon(1)\Pi_{v > 1}U^-_v) = \nu_0'(\epsilon(1)\Pi_{v > 1}U^-_v) + O_N(h + h' + 1).
\end{align}
then expand the derivative term as
\[\nu_0'(\epsilon(1)\Pi_{v > 1}U^-_v) = \frac{\beta^2}{2}\sum_{(a, b) \in {[2m] \choose 2}} D_{a, b},\]
where 
\[D_{a,b} := \nu_0(\epsilon(1)\epsilon_{a, b})\nu_0((R^-_{a, b} - Q_{a, b})\Pi_{v > 1}U^-_v) - \nu_0(\epsilon(1)\epsilon_{a^{''}, b^{''}})\nu_0((R^-_{a^{''}, b^{''}} - Q_{a^{''}, b^{''}})\Pi_{v > 1}U^-_v).\]
Here 
\[\epsilon(1) = \epsilon_{1_1, 1_1} - p.\]
Note that still if $\{a, b\} \cap [m] = 0$, then $D_{a, b}= 0$. Thus we only need to consider $\{a, b\} \cap [m] \neq \emptyset$.

For $a = b$,
\begin{itemize}
    \item If $a = b = 1_1$, since $1_1$ does not appear in any other terms, 
    \[D_{a, b} = (K_2 - K_4)\nu_0(g(h, h')) + O_N(h + h' + 1).\]
    \item If $a = b \notin V_1$, suppose $a \in V_k$,
    \begin{align*}
        D_{a, a} &= K_4 \nu_0((R_{a, a} - p)\Pi_{v > 1}U^-_v) - K_4\nu_0((R_{a^{''}, a^{''}} - p)\Pi_{v > 1}U^-_v)\\
        &\stackrel{\ref{lem: S term intersections}}{=} K_4 \nu_0((S_aX_a + SX)\Pi_{v > 1}U^-_v) - K_4\nu_0(S\Pi_{v > 1}U^-_v)\\
        &= K_4 \nu_0((S_aX_a)\Pi_{v > 1}U^-_v).
    \end{align*} 
    Summing over $h$ terms  $T$ and $h'-1$ terms $S$,
    \begin{align*}
        \sum_{a \in [m] \backslash V_1} D_{a, a} &= K_4 2h \nu_0(S_aT_a \Pi_{v > 1}U^-_v) + K_4 (h' - 1)\nu_0(S_a^2 \Pi_{v > 1}U^-_v),\\
        &\stackrel{\ref{lemma:last spin}}{=}2hK_4  \nu_0(S_aT_a g(h - 1, h' - 1)) + (h' - 1)K_4  \nu_0(S_a^2 g(h, h' - 2)).
    \end{align*}
\end{itemize}
For $a \neq b$, since $|V_1| = 1$, $|\{a, b\} \cap V_1| \leq 1$,
\begin{itemize}
    \item If $|\{a, b\} \cap V_1| = 0$ and $\{a, b\} \subset [m]$: then $a \in V_k$ and $b \in V_l$ for $1 < k, l \leq M$. By Lemma~\ref{lem: T terms intersections}, for $k \neq l$,
    \begin{align*}
        D_{a, b} &= K_3 \nu_0((T_aX_aY + T_bY_bX + TXY)\Pi_{w \notin \{k, l\}} U_w)) - K_3 \nu_0(T\Pi_{v > 1}U_v)\\
        &= K_3 \nu_0((T_aX_aY + T_bY_bX)\Pi_{w \notin \{k, l\}} U_w)).
    \end{align*}
     For $k = l$, 
    \[D_{a, b} = K_3 \nu_0(\left(T_{a, b}^2 + T_a^2 + T_b^2\right)\Pi_{v \neq 1, k}U_v).\]
    Since there are $h$ $T$ terms and $h' - 1$ $S$ terms
    \begin{align*}
        \sum_{|\{a, b\} \cap V_1| = 0} D_{a, b} =& 4 h(h - 1) K_3 \nu_0((T_a^2 + T_b^2)g(h - 1, h' - 1))\\
        &+ 2 \times 2h(h' - 1)K_3 \nu_0((T_aS_aT + T_b^2S)g(h - 1, h' - 2))\\
        & + 2(h' - 1)(h' - 2)K_3 \nu_0((T_aS_a + T_bS_b)g(h, h' - 2))\\
        &+ 2h K_3 \nu_0\nu_0(\left(T_{a, b}^2 + T_a^2 + T_b^2\right)g(h - 1, h' - 1))\\
        &+ O_N(h + h' + 1).
    \end{align*}
    \item If $|\{a, b\} \cap V_1| = 1$ and $\{a, b\} \subset [m]$: Suppose $a = 1_1$, if $b \in V_l$ for some $k$, by Lemma~\ref{lem: T terms intersections}
    \[D_{a, b} = K_1 \nu_0((T_bY_b + TY)\Pi_{v \neq 1, l}U_v) - K_3\nu_0(T \Pi_{v > 1}U_v) + O_N(h + h' + 1)\]
    Summing up all such terms, we obtain
    \begin{align*}
        \sum_{|\{a, b\} \cap V_1| = 1, \{a, b\} \subset [m]}D_{a, b} &= 2 \times 2h K_1 \nu_0(T_b^2 g(h - 1, h' - 1))\\
        &+ 2(h' - 1)K_1 \nu_0(T_bS_b g(h, h' - 2))\\
        &+ 2(2h + h' - 1)(K_1 - K_3) \nu_0(g(h + 1, h' - 1)).
    \end{align*}
    \item  If $|\{a, b\} \cap V_1| = 1$ and $|\{a, b\} \cap [m]| = 1$,
    \[D_{a, b} = K_1 \nu_0(T\Pi_{v \neq 1, l}U_v) - K_3\nu_0(T \Pi_{v > 1}U_v).\]
    Summing over all such terms gives 
    \[-2m (K_1 - K_3)\nu_0(g(h + 1, h' - 1)).\]
    \item If $|\{a, b\} \cap V_1| = 0$ and $|\{a, b\} \subset [m]| = 1$: let $a \in V_k$,
    \[D_{a, b} = K_3 \nu_0((T_aX_a + TX)\Pi_{v \neq 1, k}U_v) - K_3\nu_0(T \Pi_{v > 1}U_v).\]
    Summing up all such terms gives 
    \[-2m 2h K_3 \nu_0(T_a^2 g(h - 1, h' - 1)) -2m (h' - 1) K_3 \nu_0(T_aS_a g(h, h' - 2)).\]
\end{itemize}
Combining and simplifying all terms above gives 
\begin{align*}
     \nu_0'(\epsilon(1)\Pi_{v > 1}U^-_v) 
    =& \frac{\beta^2}{2}(K_2 - K_4)\nu_0(g(h, h'))\\
    &+ \beta^2 hK_4  \nu_0(S_aT_a g(h - 1, h' - 1))\\
    &+ \frac{\beta^2}{2}(h' - 1)K_4  \nu_0(S_a^2 g(h, h' - 2))\\
    & + 2\beta^2 h(K_1 - 2K_3)\nu_0(T_1^2g(h - 1, h' - 1))\\
    &+ \beta^2(h' - 1)(K_1 - 2K_3)\nu_0(T_1S_1g(h, h' - 2))\\
    &+ \beta^2 h K_3\nu_0(T_{a, b}^2g(h - 1, h' - 1))\\
    &-\beta^2(K_1 - K_3)\nu_0(g(h + 1, h' - 1))\\
    &+ O_N(h + h' + 1).
\end{align*}
Observe that for all terms on the RHS is a function of order $h + h'$, thus by \eqref{eq:1st approx} and Theorem \ref{thm: general T12}, \ref{thm:peel off T_kS_k}.
\begin{align*}
     \nu_0'(\epsilon(1)\Pi_{v > 1}U^-_v)
    \stackrel{\ref{claim: constants ST}}{=}& \frac{\beta^2}{2}D\nu(g(h, h'))\\
    &+ \beta^2 h\left[K_4 C_1^2 + 2(E - K_3)A_1^2 + K_3A_2^2\right]\nu( g(h - 1, h' - 1))\\
    &+ \beta^2(h' - 1)\left[\frac{1}{2}K_4  B_1^2 + (E - K_3)C_1^2\right]\nu_0( g(h, h' - 2))\\
    &-\beta^2E\nu_0(g(h + 1, h' - 1))\\
    &+ O_N(h + h' + 1).\\
\end{align*}
\YS{To handle the case when $h' = 1$ or $h = 0$: the same arguement from the proof of Lemma \ref{ap: derivative T} applies by noting $h' - 1= |\{k: ||V_k|  = 1; k > 1\}|$ and $h = |\{k: ||V_k|  = 2; k > 1\}|$.} Plug this back in \eqref{eq: deri S 1st step} completes the proof.

\end{proof}

\end{document}